\documentclass[a4paper,oneside]{amsart}
\usepackage[a4paper]{geometry}
\usepackage[T1]{fontenc}
\usepackage{amsmath,amsthm,enumitem,amsfonts,comment,amssymb,graphicx,tikz-cd,todonotes,thmtools,hyperref}
\hypersetup{
	colorlinks,
	linkcolor={red!35!black},
	citecolor={red!35!black},
	urlcolor={red!35!black}}
\usepackage{cleveref}
\usepackage[initials, nobysame,shortalphabetic]{amsrefs}
\renewcommand{\parenthesize}[1]{} 
\DeclareUnicodeCharacter{03C3}{\ensuremath{\sigma}}
\DeclareUnicodeCharacter{03B1}{\ensuremath{\alpha}}
\DeclareUnicodeCharacter{2212}{\ensuremath{-}}
\DeclareUnicodeCharacter{0327}{\c}
\DeclareUnicodeCharacter{0119}{\k{e}}
\newtheorem{maintheorem}{Theorem}
 
\newtheorem{theorem}{Theorem}[section]

\newtheorem{proposition}[theorem]{Proposition}

\newtheorem{question}[theorem]{Question}

\theoremstyle{definition}
\newtheorem{remark}[theorem]{Remark}
\newcommand{\N}{\mathbb{N}}
\newcommand{\R}{\mathbb{R}}
\newcommand{\C}{\mathbb{C}}
\newcommand{\Z}{\mathbb{Z}}
\newcommand{\Q}{\mathbb{Q}}
\newcommand{\T}{\mathbb{T}}
\newcommand{\Sup}{\operatorname{sup}_0}
\newcommand{\ent}{\underline{\operatorname{ent}}}
\newcommand{\aut}{\operatorname{Aut}}
\newcommand{\id}{\operatorname{Id}}

\newcommand{\length}{\operatorname{length}}
\title{Isomorphisms and slow entropy of deterministic $[T,T^{-1}]$ systems}
\author{Nicanor Carrasco-Vargas}
\thanks{Faculty of Mathematics and Computer Science, Jagiellonian University, Kraków, Poland}
\date{}

\begin{document}
\begin{abstract}
We study skew product measure-preserving systems driven by an irrational rotation and the step function with values $1$ and $-1$ on each half of the circle. These systems can be seen as a deterministic version of Kalikow's $[T,T^{-1}]$ system, and are particular instances of Rokhlin cocycle extensions. 

Under the assumption of ergodicity, we show that these systems obey an interesting rigidity property for isomorphisms. That is, we show that two skew products are isomorphic if and only if the rotations in the base are isomorphic, and the fiber transformations are flip isomorphic. Under mild extra assumptions, we are able to characterize all isomorphisms.  

We prove that by choosing the rotation angle suitably, these systems can realize arbitrarily low measure-theoretic complexity in the sense of lower slow entropy. We apply this result to obtain new examples of systems that fail the variational principle for slow entropy. These examples can have a set of invariant measures as rich as desired (any metrizable Choquet simplex, up to affine homeomorphism). 
\end{abstract}

\maketitle
\section{Introduction}\label{sec:introduction}
\begin{figure}[h]
	\centering
	\begin{tikzpicture}[x=5cm,y=0.6cm]
		\def\c{0.5} 
		
		\draw[gray] (0,0) -- (1,0);
		
		\draw[very thick] (0,0) -- (0.5-\c,0);
		\draw[very thick] (0.5-\c,1) -- (0.5,1);
		\draw[very thick] (0.5,0) -- (1-\c,0);
		\draw[very thick] (1-\c,-1) -- (1,-1);
	\end{tikzpicture}
	\caption{The step function   $\tau\colon\T\to\Z$.}
	\label{fig:tau}
\end{figure}
We study the following measure-preserving transformations, which are a particular class of Rokhlin cocycle extensions (see \cite{lemanczyk_ergodicity_2001,lemanczyk_lifting_2012b}). Let $\T=[0,1)$ be the additive group of reals modulo 1, and consider the step function $\tau$ taking values $1$ and $-1$ on two different halves of $\T$:
\[\tau(\theta)=\mathbf 1_{[0,1/2)}(\theta)-\mathbf 1_{[1/2,1)}(\theta)\] 
Let $ (\T,\mathcal B_{\T},m_{\T},R_{\alpha})$ be an irrational rotation, and let $(X,\mathcal B_X,\mu,T)$ be an invertible measure-preserving system, where $(X,\mathcal B_X,\mu)$ is a standard Borel probability space. We consider the automorphism   $R_{\alpha}\rtimes_\tau T$ of the product space  $(\T\times X,\mathcal B_{\T}\otimes \mathcal B_X,m_{\T}\times\mu)$ given by 
\begin{equation}\label{eq:skew-product-definition}
(\theta,x)\mapsto (R_{\alpha}(\theta),T^{\tau(\theta)}(x))
\end{equation}
One can regard this transformation as a deterministic version of the well-studied $[T,T^{-1}]$ system, in which one applies $T$ or its inverse $T^{-1}$ independently and with equal probabilities; the base system is a Bernoulli shift. $[T,T^{-1}]$ transformations have played an important role by providing examples of Kolmogorov  transformations that are not Bernoulli  \cite{kalikow_1_1982}; see also \cite{katok_smooth_1981,rudolph_asymptotically_1988,denhollander_mixing_1997, kanigowski_nonequivalence_2018, dolgopyat_flexibility_2022}. Let us mention that distinguishing two of these systems up to isomorphism is difficult. A highly nontrivial theorem states that if  $[S,S^{-1}]$ and $[T,T^{-1}]$ are isomorphic, then $S$ and $T$ need to have the same entropy \cite{heicklen_entropy_2000,ball_entropy_2003,austin_scenery_2014}. While the $[T,T^{-1}]$ system ``remembers'' the entropy of $T$, it also ``forgets'' other dynamical characteristics. For instance, $[T,T^{-1}]$ can be isomorphic to $[S,S^{-1}]$ with $T$ an irrational rotation and $S$ a zero entropy mixing transformation   \cite{burton_mixingt_1983,leuridan_when_2024}.

Our main result concerns the isomorphism problem for the skew products defined in \Cref{eq:skew-product-definition}, under the assumption of ergodicity. We show that two such skew products are isomorphic exactly when the base rotations are isomorphic, and the fiber transformations are flip isomorphic. Recall that two systems are \emph{flip isomorphic} if one of them is isomorphic to the other, or to the inverse of the other. 
\begin{maintheorem}[\Cref{thm:isomorphisms-section}]\label{thm:isomorphisms}
    Let $\alpha,\beta\in \R\smallsetminus \Q$, and let $(Y,\mathcal B_Y,\nu,S)$ and $(X,\mathcal B_X,\mu,T)$ be ergodic invertible measure-preserving systems with nonatomic measures. Then $R_{\alpha}\rtimes_\tau S$ is isomorphic to $R_{\beta}\rtimes_\tau T$ if and only if $S$ is flip isomorphic to $T$ and $\beta=\pm \alpha\mod 1$. 
\end{maintheorem} 
Thus there is an important contrast between the $[T,T^{-1}]$ system and its deterministic version, and the possibilities for isomorphisms are much more rigid in the deterministic case. We also remark that the analogous statement is false for direct products: it is possible to have an isomorphism between $R_{\alpha}\times S$ and $R_{\alpha}\times T$, while $S$ and $T$ are not isomorphic or flip isomorphic \cite{gerber_nonclassifiability_2025}.

In some cases we are able to characterize all possible isomorphisms between two ergodic skew products   $R_{\alpha}\rtimes_{\tau} S$ and $R_{\alpha}\rtimes_{\tau} T$ with the same rotation in the base. Let $\ell\in\{-1,1\}$ and $c\in \T$, and consider the step function $\zeta\colon\T\to\Z$ given by
\[\zeta(\theta)=\tau(\theta+c)-\ell\tau(\theta)\]
Depending on the values of $\ell$ and $c$, it is possible that $\zeta$ is a coboundary, meaning that $
\zeta=f\circ R_{\alpha}-f$ almost everywhere for some measurable $f\colon\T\to\Z$. If this is the case, then we can define an isomorphism from $R_{\alpha}\rtimes_{\tau} S$ to $R_{\alpha}\rtimes_{\tau} T$ by 
\begin{equation}\label{eq:isomorphisms-coboundary}
(\theta,x)\mapsto (\theta+c,T^{f(\theta)}(\upsilon(x))),
\end{equation}
where $\upsilon$ is any isomorphism from $S$ to $T^{\ell}$ (\Cref{prop:easy-isomorphisms-cohomology-section}). Under the assumption that $\alpha$ has bounded type, or that $S$ is not a rigid transformation, we are able to prove that these are the only isomorphisms (\Cref{thm:isomorphisms-characterization}). 


Our main tool to prove \Cref{thm:isomorphisms} and related results is a fixed point method involving the odd denominators of convergents in the continued fraction expansion of $\alpha$. The same method was previously used to prove that $\tau$ defines an ergodic cocycle. This is equivalent to the statement that the cylinder flow or group extension defined by $\tau$ and $R_{\alpha}$ is ergodic, see \cite[Theorem 1.2]{aaronson_visitors_1982} and \cite[page 5]{conze_ergodicite_1976}. We adapt some of these ideas to equations in $\aut(X,\mathcal B_X,\mu)$  parametrized by elements in $\T$. 

Our next results concern the slow entropy of the same skew products. \textit{Slow entropy} is a family of entropy-type invariants introduced by Katok and Thouvenot \cite{katok_slow_1997} with the purpose of studying low complexity dynamical systems, and it can be applied both to measure-preserving and topological dynamical systems.  The measure-theoretic (resp. topological) slow entropy of a transformation is obtained by comparing a family of rate functions, called a \textit{scale}, with the number of Hamming pseudoballs (resp. Bowen balls) that are needed to cover most of the space (resp. the whole space). Slow entropy has two versions: lower slow entropy, which has a liminf in its definition, and upper slow entropy, which has a limsup. In this work, slow entropy always refers to \textit{lower} slow entropy, which we denote by the underlined term $\ent$. 

The next theorem shows that the systems that we study can have arbitrarily low measure-theoretic complexity in the sense of slow entropy. We highlight that the next statement is  uniform over $T$. 
\begin{maintheorem}[\Cref{thm:zero-slow-entropy-section}]\label{thm:zero-slow-entropy}
For every scale $\mathbf a$ for slow entropy we can find an irrational $\alpha\in\R\smallsetminus\Q$ with the following property. For every choice of an invertible measure-preserving system  $(X,\mathcal B_X,\mu,T)$ we have
	\[
	\ent_{m_{\T}\times\mu}^{\mathbf a}(R_{\alpha}\rtimes_\tau T)=0
	\]
\end{maintheorem}
Given a topological dynamical system and a fixed scale for slow entropy, the measure-theoretic slow entropy with respect to any invariant measure is bounded by the topological one. The variational principle may or may not hold, and nontrivial examples are known in both cases \cite{kanigowski_slow_2019, kanigowski_slow_2018, ahn_entropy_2010,dou_entropy_2022,cheng_slow_2025,banerjee_slow_2023,banerjee_slow_2023a}. We apply \Cref{thm:zero-slow-entropy} to obtain a new family of non-examples for  the variational principle for slow entropy.  

Let us now define the continuous counterpart of the transformation in \Cref{eq:skew-product-definition}. Let  $Y_{\alpha}\subset\{-1,1\}^{\Z}$ be the smallest subshift on symbols $1$ and $-1$ containing the sequence $(\tau(R^n_{\alpha}(0)))_{n\in\Z}$, and let $S_{\alpha}$ denote the shift transformation on $Y_{\alpha}$. Then $(Y_{\alpha},S_{\alpha})$ is a uniquely ergodic topological dynamical system, and we denote its unique invariant measure by $\hat m$. As measure-preserving systems, $(Y_{\alpha},\mathcal B_{Y_{\alpha}},\hat m,S_{\alpha})$ and $(\T,\mathcal B_{\T},m_{\T},R_{\alpha})$ are isomorphic. 

Consider the continuous map
\[\hat\tau\colon Y_{\alpha}\to\Z, \ y\mapsto y(0)\]
For an arbitrary invertible topological dynamical system $(X,T)$, we define the transformation  $S_{\alpha}\rtimes_{\hat \tau} T$ acting on the product space $Y_{\alpha}\times X$ by 
\begin{equation}\label{eq:skew-product-topological}
(y,x)\mapsto (S_{\alpha}(y),T^{\hat \tau(y)}(x))
\end{equation}
Thus $(Y_{\alpha}\times X,   S_{\alpha}\rtimes_{\hat \tau} T)$ is a topological dynamical system. The following known fact will be important for us. 
Here $\mathcal M(\cdot)$ stands for the set of invariant Borel probability measures of a transformation. 
\begin{theorem}[Corollary 13 in \cite{lemanczyk_ergodicity_2001}]\label{thm:product-measures}
    Let $\alpha\in\R\smallsetminus\Q$ and let $(X,T)$ be an invertible topological dynamical system. Then
    \[
\mathcal M(S_{\alpha}\rtimes_{\hat \tau} T)=\{\hat m\times\mu : \mu\in\mathcal M(T)\}
    \]
    In particular, $\mathcal M(T)$  is affinely homeomorphic to $\mathcal M (S_{\alpha}\rtimes_{\hat \tau} T)$ via the map $\mu\mapsto \hat m\times \mu$. 
\end{theorem}
It follows from this result that when we endow $S_{\alpha}\rtimes_{\hat\tau} T$ with the invariant measure $\hat m\times\mu$, for $\mu\in\mathcal M(T)$, we obtain a measure-preserving system isomorphic to $R_{\alpha}\rtimes_{\tau} T$ endowed with $m_{\T}\times\mu$. 

In the next result we consider slow entropy with respect to the so-called \textit{polynomial} scale $\{n^t\}_{n\in\N,t>0}$. This scale is frequently considered in the context of slow entropy (see for instance \cite{kanigowski_slow_2019,banerjee_slow_2023,banerjee_slow_2023a}), and the corresponding topological invariant also goes under the name of \textit{polynomial entropy} \cite{roth_rigidity_2024,marco_polynomial_2013}. Applying \Cref{thm:zero-slow-entropy} to this scale, we obtain the following.
\begin{maintheorem}[\Cref{thm:variational-0-section}]\label{thm:variational-0}
    Consider the slow entropy scale $\mathbf a = \{n^t\}_{n\in\N,t>0}$. There exists  $\alpha\in\R\smallsetminus\Q$ such that for every invertible topological system $(X,T)$ we have 
    \[\ent^{\mathbf a}_{top}(S_{\alpha}\rtimes_{\hat \tau} T)\geq 1 \text{ and }  \ \ 
\sup_{\nu\in\mathcal  M(S_{\alpha}\rtimes_{\hat \tau} T)} 
\ent_{\nu}^{\mathbf a}(S_{\alpha}\rtimes_{\hat \tau} T)=0\]
\end{maintheorem}
This family of non-examples of the variational principle for slow entropy can have a set of invariant measures as rich as desired (any metrizable Choquet simplex, up to affine homeomorphism). By \Cref{thm:product-measures}, it suffices to vary $T$.

Let us mention that the skew product  $S_{\alpha}\rtimes_{\hat\tau} T$ admits no ergodic invariant measure that makes it a Kronecker system (a rotation in a compact abelian group), provided $(X,T)$ has no fixed point and no periodic orbit of length two (\Cref{prop:discrete-spectrum-measures}). In this sense, our non-examples are different from previously known ones, which have the property that all their ergodic invariant measures yield Kronecker systems.  In fact, previous non-examples are constructed using the characterization of ergodic systems with vanishing slow entropy at all scales, which are precisely the Kronecker systems. See \cite[Proposition 3]{ferenczi_measuretheoretic_1997}, \cite[\S 4.5]{kanigowski_survey_2024}, \cite[\S 5]{cheng_slow_2025}, or \cite{lott_relative_2023}.

Given $\alpha\in\R\smallsetminus\Q$, we shall also be interested in the slow entropy scale $\mathbf b=\{b_n(t)\}_{n\in\N,t>0}$  defined by 
\[
b_n(t)=\sum_{w\in\mathcal L_n(Y_{\alpha})}e^{r_n(w)t}
\]
Here $\mathcal L_n(Y_{\alpha})$ is the set of words of length $n$ in $Y_{\alpha}$. For $w\in\mathcal L_n(Y_{\alpha})$,  $r_n(w)$ denotes the number of places visited by a ``walker'' in $\Z$ driven by $w$ (see \Cref{eq:def-r-n-symbolic}). This scale was introduced in  \cite{carrascovargas_topological_2025} to capture the contribution of $(X,T)$ to $(Y_{\alpha}\times X, S_{\alpha}\rtimes_{\hat \tau} T)$ in the sense of topological complexity. That is, it is proved in \cite{carrascovargas_topological_2025} that for every $\alpha\in\R\smallsetminus\Q$ and for every choice of $(X,T)$ we have 
\[\ent_{top}^{\mathbf b}(S_{\alpha}\rtimes_{\hat \tau} T)=h_{top}(T)\]
A question left in \cite{carrascovargas_topological_2025} was whether the same scale could capture the contribution of $(X,T)$ to  $(Y_{\alpha}\times X, S_{\alpha}\rtimes_{\hat \tau} T)$ in the sense of measure-theoretic complexity. Using \Cref{thm:zero-slow-entropy} we are able to show that, at least for some irrationals, this is not the case. In fact, we have the opposite situation, and the same scale  provides non-examples of the variational principle whenever $h_{top}(T)>0$.  
\begin{maintheorem}[\Cref{thm:variational-section}]\label{thm:variational}
There exists $\alpha\in\R\smallsetminus\Q$ with the following property. For every invertible topological system $(X,T)$ we have 
\begin{align*}
\sup_{\nu\in\mathcal  M(S_{\alpha}\rtimes_{\hat \tau} T)} 
\ent_{\nu}^{\mathbf b}(S_{\alpha}\rtimes_{\hat \tau} T)=0
\end{align*}
\end{maintheorem}
Unlike \Cref{thm:variational-0}, we cannot obtain \Cref{thm:variational} by a straightforward application of \Cref{thm:zero-slow-entropy}. This is because the scale $\mathbf{b}$ depends on $\alpha$. But this difficulty admits an elementary solution; see the proof of \Cref{thm:variational-section}.

As we have seen, in the family of skew products $R_{\alpha}\rtimes_{\tau} T$ or $S_{\alpha}\rtimes_{\hat \tau} T$, it is interesting to vary the rotation angle $\alpha\in\R\smallsetminus\Q$, the fiber transformation $(X,\mathcal B_X,\mu,T)$ or $(X,T)$, and a slow entropy scale, in  different orders. Theorems \ref{thm:variational-0} and \ref{thm:variational} are particular instances of this. We finish this introduction by pointing out some other interesting consequences. 
\begin{remark} If we  apply \Cref{thm:zero-slow-entropy} to an arbitrary slow entropy scale, and then apply \Cref{thm:isomorphisms}, we obtain a proof that for every slow entropy scale, there are uncountably many non-isomorphic ergodic systems whose slow entropy is zero with respect to that scale. The existence of arbitrarily slow systems in this sense was already known, see  \cite[Corollary 4.10.3]{kanigowski_survey_2024} and \cite[Theorem 2.15]{banerjee_slow_2023}. 
\end{remark}

\begin{remark}
\Cref{thm:zero-slow-entropy} cannot be improved to the existence of a single irrational that satisfies the same conclusion for all slow entropy scales. Indeed, let $(X,\mathcal B,\mu,T)$ be an ergodic measure-preserving system which is neither the one-point system nor the two-point system. \Cref{thm:zero-slow-entropy} shows that if we fix a slow entropy scale $\mathbf a$ and we make $\alpha\in\R\smallsetminus \Q$ vary, then some choice of $\alpha$ satisfies $\ent^{\mathbf a}_{m_{\T}\times\mu}(R_{\alpha}\rtimes_{\tau} T)=0$. However, for any choice of $\alpha$ we have some scale $\mathbf c$ with $\ent^{\mathbf c}_{m_{\T}\times\mu}(R_{\alpha}\rtimes_{\tau} T)>0$. This follows from Ferenczi's Theorem (see \cite[Proposition 3]{ferenczi_measuretheoretic_1997}, \cite[\S 4.5]{kanigowski_survey_2024}, \cite[\S 5]{cheng_slow_2025}, or \cite{lott_relative_2023}), and the fact that  $(\T\times X,\mathcal B_{\T}\otimes \mathcal B_X,m_{\T}\times\mu,R_{\alpha}\rtimes_{\tau} T)$ is not Kronecker, which we prove in \Cref{prop:kronecker-characterization}. 
\end{remark}
\begin{remark}
	In \cite{dou_entropy_2022} the authors study the same class of skew products with the goal of realizing values of entropy dimension. They prove that for fixed $T$ with positive entropy, all entropy dimensions between 0 and 1 are realized in the family of systems $\{R_{\alpha}\rtimes_{\tau} T : \alpha\in\R\smallsetminus\Q\}$. This statement and  \Cref{thm:zero-slow-entropy} complement each other, showing that the same family realizes a rich class of sub-exponential complexities. We also remark that by \Cref{thm:isomorphisms}, their main result \cite[Theorem 5.2]{dou_entropy_2022} actually shows the existence of uncountably many non-isomorphic systems with a prescribed value of entropy dimension. 
\end{remark}
\subsection{Manuscript organization} Preliminaries are in \Cref{sec:preliminaries}.
\Cref{thm:isomorphisms} is proved in Sections \ref{sec:cocycles} - \ref{sec:proof-of-isomorphism-theorem}, and \Cref{thm:zero-slow-entropy} and its applications are proved in Sections \ref{sec:ranges} - \ref{sec:variational}.

The proof of \Cref{thm:isomorphisms} is organized as follows. In \Cref{sec:cocycles} we review terminology and relevant results about cocycles over irrational rotations. Next, in \Cref{sec:easy-isomorphisms} we prove the easy direction of our results about isomorphisms. In \Cref{sec:kronecker} we determine the eigenvalues of the Koopman operator of $R_{\alpha}\rtimes_{\tau} T$.  Then the nontrivial direction of \Cref{thm:isomorphisms} is proved in two parts. We show the following:
\begin{itemize}
    \item An isomorphism between $R_{\alpha}\rtimes_{\tau} S$ and $R_{\alpha}\rtimes_{\tau} T$ implies that $S$ and $T$ are flip isomorphic (\Cref{sec:isomorphisms}). 
    \item An isomorphism between $R_{\alpha}\rtimes_{\tau} S$ and $R_{\beta}\rtimes_{\tau} T$ implies $\alpha=\pm\beta\mod 1$ (\Cref{sec:different-angles}). 
\end{itemize}
In \Cref{sec:proof-of-isomorphism-theorem} we obtain  \Cref{thm:isomorphisms} and related results. 

Then we turn to the proof of \Cref{thm:zero-slow-entropy} and its applications. The main construction is done in  \Cref{sec:ranges}, where we study ranges of the walks driven by the ergodic sums of $\tau$. We prove upper bounds for them associated with convergents of $\alpha$ with an even denominator. In \Cref{sec:slow-entropy} we use these results to obtain \Cref{thm:zero-slow-entropy}. Finally, in \Cref{sec:variational}, we turn to topological dynamical systems, and obtain  \Cref{thm:variational-0} and \Cref{thm:variational} about non-examples of the variational principle for slow entropy.
\subsection{Acknowledgements}
I thank A. Kanigowski for introducing me to this problem and for guidance and discussions during this research project. I also thank M. Lemańczyk for comments on a previous version of \Cref{thm:isomorphisms} which led to an important strengthening, and for posing to me the problem of characterizing all isomorphisms  (\Cref{thm:isomorphisms-characterization}). 

This work was supported by a grant from the Priority Research Area SciMat under the Strategic Programme Excellence Initiative at Jagiellonian University, the Simons Foundation grant (award no. SFI-MPS-T-Institutes-00010825), and State Treasury funds as part of a task commissioned by the Minister of Science and Higher Education under the project “Organization of the Simons Semesters at the Banach Center - New Energies in 2026-2028” (agreement no. MNiSW/2025/DAP/491).
\subsection{AI statement} The prose of this text is human-generated, and  ChatGPT was used to find typos and small mistakes. All theorems are human-generated, with the exception of the proof of item (2) in \Cref{prop:easy-automorphisms}. This argument was pointed out to me by Kimi K2.5, simplifying a previous attempt. 
\section{Preliminaries}\label{sec:preliminaries}
In what follows $(X,\mathcal B_X,\mu)$ always denotes a standard Borel probability space. By this we mean that $\mathcal B_X$ is the Borel sigma algebra of a Polish topology on $X$, and $\mu$ is a Borel probability measure on $X$. We often take $\mu$ nonatomic, but not always. 

By an automorphism $T$ of $(X,\mathcal B_X,\mu)$ we mean a bi-measurable invertible map which preserves the measure. We also say that $T$ is an invertible measure-preserving transformation on $(X,\mathcal B_X,\mu)$, and we refer to $(X,\mathcal B_X,\mu,T)$ as an invertible measure-preserving system.

Given two measurable functions $f$ and $g$ from $(X,\mathcal B_X,\mu)$ to a measurable space $(Y,\mathcal C)$, we commit the usual notational abuse of writing $f=g$ whenever there exists a full measure subset of $X$ where $f$ and $g$ agree. 

A factor map from $(X,\mathcal B_X,\mu,T)$ onto another measure-preserving system $(Y,\mathcal C,\nu,S)$ is a surjective measurable map $\pi\colon X\to Y$ such that $\pi_*\mu=\nu$ and  $\pi\circ T = S\circ \pi$. If, in addition, $\pi$ is an isomorphism of the underlying probability spaces, then we also say it is an isomorphism of measure-preserving systems. We will often omit the probability spaces from the notation, if no ambiguity arises. Thus we may say that $T$ factors onto $S$, $T$ is isomorphic to $S$, etc.

Given $x\in \R$ we denote by $||x||\in[0,1/2]$ the distance to the nearest integer. We endow $\T=[0,1)$ with the distance $d_{\T}(\theta,\theta')=||\theta-\theta'||$, the Borel sigma algebra $\mathcal B_{\T}$, the Lebesgue or Haar measure $m_{\T}$, and the group operation of addition modulo 1. In most cases there is no ambiguity with respect to addition of real numbers, so we denote both operations by $+$ (the same applies to the notation $-$ for subtraction).

Given $a\in \T$ or $a\in\R$, we always denote by $R_{a}$ the transformation  $\theta\mapsto \theta+a$, which is an automorphism of $(\T,\mathcal B_{\T},m_{\T})$.

The following standard result will be useful to us. 
\begin{proposition}\label{prop:subsequences}
Let $\mathcal P$ be a Polish topological space and let $f\colon \T\to \mathcal P$ be a Borel measurable function. Let $(s_n)_{n\in\N}$ be a sequence of elements in $\T$ converging to zero in $d_{\T}$. Then there exists a subsequence $(s_{n_k})_{k\in\N}$ such that for almost every $\theta\in\T$ we have
\[
\lim_{k\to\infty} f(\theta+s_{n_k})=f(\theta)
\]
\end{proposition}
\begin{proof}
Let $d_{\mathcal P}$ be a compatible metric for the topology on $\mathcal P$. Furthermore, let us write $f_n(\theta)=f(\theta+s_n)$. We claim that the sequence $(f_n)_{n\in\N}$ converges in measure or in probability to $f$. By this we mean that for every $\epsilon>0$ we have
\begin{equation}\label{eq:convergence-in-measure}
\lim_{n\to\infty} m_{\T}(\{\theta\in \T : d_{\mathcal P}(f(\theta),f_n(\theta))>\epsilon\})=0
\end{equation}
Let $\epsilon>0$. We take $\rho>0$ and prove that for all $n$ large enough the measure of $\{\theta\in \T : d_{\mathcal P}(f(\theta),f_n(\theta))>\epsilon\}$ is at most $2\rho$.

By Lusin's Theorem (in the form of Theorem 17.12 in \cite{kechris_classical_1995}), there exists a compact set $K\subset \T$ with $m_{\T}(K)>1-\rho$ and such that the restriction of $f$ to $K$ is continuous. Since $K$ is compact, it follows that $f$ restricted to $K$ is uniformly continuous. Thanks to uniform continuity, we can pick $\delta$ such that for $\theta,\theta'\in K$
\[d_{\T}(\theta,\theta')<\delta\Rightarrow d_{\mathcal P}(f(\theta),f(\theta'))<\epsilon.\] Since $s_n\to 0$ in $\T$, we can  pick $N$ such that for all $n\geq N$ we have $d_{\T}(s_n,0)<\delta$. Since $d_{\T}$ is invariant by translations, it follows that $d_{\T}(\theta,\theta+s_n)<\delta$ for all $\theta\in \T$. By uniform continuity of $f$, provided that $\theta$ and $\theta+s_n$ belong to $K$, it follows that $d_{\mathcal P}(f(\theta),f(\theta+s_n))<\epsilon$, which is the same as $d_{\mathcal P}(f(\theta),f_n(\theta))<\epsilon$. 

Every element in $\T$ belongs to at least one of $K\cap (K-s_n)$, $\T\smallsetminus K$, and $\T\smallsetminus (K-s_n)$. Hence we can write 
\begin{align*}
    \{\theta\in \T : d_{\mathcal P}(f(\theta),f_n(\theta))>\epsilon\}\subset \{\theta\in K\cap& (K-s_n) : d_{\mathcal P}(f(\theta),f_n(\theta))>\epsilon\}\\ &\cup (\T\smallsetminus K)\cup (\T\smallsetminus (K-s_n))
\end{align*}
For $n\geq N$ every $\theta\in K\cap (K-s_n)$ satisfies $d_{\mathcal P}(f(\theta),f_n(\theta))<\epsilon$, and hence the first set on the right hand side is empty. Since $m_{\T}$ is translation invariant, both $\T\smallsetminus K$ and $\T\smallsetminus (K-s_n)$ have measure at most $\rho$.  Thus $ \{\theta\in \T : d_{\mathcal P}(f(\theta),f_n(\theta))>\epsilon\}$ has measure at most $2\rho$ as claimed. This proves \Cref{eq:convergence-in-measure}.

It is a standard fact that  convergence in probability implies that along a subsequence, we have pointwise convergence almost everywhere. Indeed, by  \Cref{eq:convergence-in-measure}, for each $k\in\N$ we can pick $n_k$ large enough so that $\{\theta\in\T : d_{\mathcal P}(f(\theta),f_{n_k}(\theta))>2^{-k}\}$ has measure at most $2^{-k}$. Since $\sum 2^{-k}<\infty$, by the Borel-Cantelli Lemma there is a full measure set of $\theta\in \T$ which belong to finitely many  $\{\theta\in\T : d_{\mathcal P}(f(\theta),f_{n_k}(\theta))>2^{-k}\}$. Equivalently,   $d_{\mathcal P}(f(\theta),f_{n_k}(\theta))\leq 2^{-k}$ for all but finitely many $k$. For these $\theta$ we have $d_{\mathcal P}(f(\theta),f_{n_k}(\theta))\to 0$ as $k\to\infty$. 
\end{proof}
\section{Cocycles over irrational rotations}\label{sec:cocycles}
In this section we will review definitions and general results for $\Z$-valued cocycles over irrational rotations. We refer the reader to \cite{schmidt_cocycles_1977} and \cite{conze_remarks_2013}. We will also state and prove results for the specific cocycles that we will encounter.  

Let $\alpha\in\R\smallsetminus\Q$. A measurable cocycle for $R_{\alpha}$ with values in $\Z$ is a measurable function $\gamma\colon\Z\times\T\to\Z$ satisfying the cocycle equation
\[
\gamma(n+k,\theta)=\gamma(n,\theta)+
\gamma(k,R^n_{\alpha}(\theta))
\]
A measurable map $\gamma'\colon\T\to\Z$ generates a cocycle by considering its ergodic sums
\[
\gamma(n,\theta)=\begin{cases}
    \sum_{i=0}^{n-1}\gamma'(R_{\alpha}^i(\theta)) \ \ &n\geq 1\\
    0 \ \ &n=0\\
    -\sum_{i=n}^{-1}\gamma'(R_{\alpha}^i(\theta)) \ \ &n\leq -1
\end{cases}
\]
Conversely, any cocycle $\gamma$ with values in $\Z$ is generated in the sense above by the function $\gamma'\colon\T\to\Z$, $\gamma'(\theta)=\gamma(1,\theta)$. We commit the notational abuse of denoting both objects by the same symbol, provided no ambiguity arises.

The cocycle $\gamma$ is said to be a coboundary whenever there is a measurable function $f\colon\T\to\Z$ such that $\gamma=f\circ R_{\alpha}-f$ almost everywhere. 

We say that $k\in\Z$ is an essential value for $\gamma$ if for every $B\subset\T$ with positive measure, there exists $n\in\Z$ such that 
\begin{equation}\label{eq:essential-value}
m_{\T}(B\cap R_{\alpha}^{-n}(B)\cap \{\theta\in \T : \gamma(n,\theta)=k\})>0
\end{equation}
Similarly, $\infty$ is an essential value for $\gamma$ if for every $k\in\N$ and $B\subset\T$ with positive measure, there exists $n\in\Z$ such that 
\begin{equation}\label{eq:essential-value-infty}
m_{\T}(B\cap R_{\alpha}^{-n}(B)\cap \{\theta\in \T : |\gamma(n,\theta)|>k\})>0
\end{equation}
The set of essential values of $\gamma$ in $\Z\cup\{\infty\}$ will be denoted $\overline{\mathcal E}(\gamma)$, and the set of its essential values in $\Z$ will be denoted by $\mathcal E(\gamma)$. Then $\mathcal E(\gamma)$ is a subgroup of $\Z$, and by a theorem of Schmidt, $\gamma$ is a coboundary if and only if $\overline{\mathcal E}(\gamma)$ equals the trivial group $\{0\}$ (\cite[\S 3]{schmidt_cocycles_1977}). We say that $\gamma$ is ergodic if $\mathcal E(\gamma)=\Z$.

As mentioned in the introduction, we will be specifically interested in the cocycle $\tau\colon\Z\times\T\to\Z$ generated by 
\[
\tau\colon\T\to\Z, \ \tau(\theta)=\mathbf 1_{[0,1/2)}(\theta)-\mathbf{1}_{[1/2,1)}(\theta)
\]
The following two results, observed by Aaronson and Keane in \cite[\S 1]{aaronson_visitors_1982}, will be particularly useful to us.
\begin{proposition}\label{prop:nice-sequence-of-odd-numbers}
    Let $\alpha\in\R\smallsetminus\Q$. There exists an increasing sequence $(q_n)_{n\in\N}$ of odd integers such that for all $n$ we have
\begin{equation}\label{eq:nice-sequence-of-odd-numbers}
    |\alpha-p_n/q_n|\leq 1/(2q_n^2)
\end{equation}
for some integer $p_n$ coprime with $q_n$. 
\end{proposition}
\begin{proposition}\label{prop:infinitely-many-1s}
    Let $(q_n)_{n\in\N}$ be an arbitrary sequence as in  \Cref{prop:nice-sequence-of-odd-numbers}. Then for almost every $\theta\in\T$ we have $\tau(q_{n},\theta)=1$ for infinitely many values of $n$, and also $\tau(q_{n},\theta)=-1$ for infinitely many values of $n$.
\end{proposition}
The original formulation \cite[Lemma 1.1]{aaronson_visitors_1982} only states that for almost every $\theta$, we have $\tau(q_n,\theta)=1$ for infinitely many values of $n$. This directly implies the  statement in \Cref{prop:infinitely-many-1s} because
\begin{equation}\label{eq:symmetry}
    \tau(n,\theta+1/2)=-\tau(n,\theta) \ \ n\in\Z, \ \theta\in\T
\end{equation}
In turn, this is a consequence of the equality 
\begin{equation}\tau(\theta+1/2)=-\tau(\theta), \ \ \theta\in\T
\end{equation}   
\begin{remark} If $(q_n)_{n\in\N}$ satisfies \Cref{prop:nice-sequence-of-odd-numbers}, then the same is true for any of its subsequences.  Thus the statements we can prove for a sequence as in  \Cref{prop:nice-sequence-of-odd-numbers} are also valid for any of its subsequences. 
\end{remark}
\begin{proposition}\label{prop:cocycle-invariant-along-q-ns}
Let $(q_n)_{n\in\N}$ be a sequence as in  \Cref{prop:nice-sequence-of-odd-numbers}. Then for almost all $\theta\in\T$ we have $\tau(q_n,\theta)=\tau(q_n, R_{\alpha}(\theta))$ for all but finitely many values of $n$. 
\end{proposition}
\begin{proof}
Observe that $\sum_{n\in\N} 1/q_n<\infty$. Indeed, by Legendre's Theorem, $(q_n)_{n\in\N}$ is a subsequence of the sequence of denominators of the convergents in the continued fraction expansion of $\alpha$, which has summable inverses (see for instance page 71 in \cite{einsiedler_ergodic_2011}). 

    It follows from the definition that 
    \[
    \tau(q_n,R_{\alpha}(\theta)) = \tau(q_n,\theta)-\tau(\theta)+\tau(R_{\alpha}^{q_n}(\theta)), \ \ \theta\in\T, n\in\N
    \]
    We claim that 
    \[
    B=\{\theta\in\T : \tau(\theta)\ne \tau(R_{\alpha}^{q_n}(\theta)) \text{ for infinitely many $n$}\}
    \]
    has null measure. Indeed, for each $n\in\N$ let 
    \[B_n=\{\theta\in\T : \tau(\theta)\ne \tau(R_{\alpha}^{q_n}(\theta))\}\]
    Thus $B=\cap_{m\in\N}\cup_{n\geq m}B_n$ is the set of elements that belong to infinitely many $B_n$'s. 
    
    If $||q\alpha||<1/q$, then $d_{\T}(0,R_{\alpha}^q(0))<1/q$. Since $R_{\alpha}$ is an isometry of $(\T,d_{\T})$, this implies that $d_{\T}(\theta,R_{\alpha}^q(\theta))<1/q$ for all $\theta\in \T$. This shows that whenever $\theta$ is at distance at least $1/{q_n}$ from $0$ and $1/2$, then $\theta\not\in B_n$. Hence  $m_{\T}(B_n)<4/q_n$. Since $\sum_{n\in\N}1/q_n<\infty$, we obtain that $\sum_{n\in\N}m_{\T}(B_n)<\infty$. Then the Borel-Cantelli Lemma shows that $B$ has null measure.
\end{proof}

\begin{remark}\label{rem:denjoy-koksma}
By the Denjoy-Koksma inequality, if $|\alpha-p/q|<1/q^2$ for $p,q\in\N$ coprime, then
    \[|\tau(q,\theta)|\leq 4, \ \theta\in\T\]
If $q$ is odd, then $\tau(q,\theta)$ is an odd number for any $\theta$, as it is a sum of an odd number of $1$'s and $-1$'s. Thus in this case we can upgrade the conclusion to \[\tau(q,\theta)\in\{-3,-1,1,3\}, \ \theta\in\T\]
\end{remark}
\begin{proposition}\label{prop:ell-possible-values}
    Let $(q_n)_{n\in\N}$ be a sequence as in \Cref{prop:nice-sequence-of-odd-numbers}. Let $c\in\T$ (not necessarily irrational). Then there exists $\ell\in\{-3,-1,1,3\}$ such that for almost every $\theta\in\T$, we have 
    \begin{equation}\label{eq:two-dimensional-AK}
    (\tau(q_n,\theta),\tau(q_n,\theta+c))=(1,\ell) \ \ \text {for infinitely many $n$}\end{equation} 
\end{proposition}
\begin{proof}
For each $i\in\{-3,-1,1,3\}$ we define $D_i$ as the set of elements $\theta\in\T$ satisfying \Cref{eq:two-dimensional-AK} above with $(1,i)$ in place of $(1,\ell)$.

We claim that $m_{\T}(D_{-3}\cup D_{-1}\cup D_{1}\cup D_3)=1$. Let $\theta$ have the property that we can find a subsequence $(q_{n_k})_{k\in\N}$ with $\tau(q_{n_k},\theta)=1$ for all $k$. This is a full measure condition by \Cref{prop:infinitely-many-1s}. Then by \Cref{rem:denjoy-koksma}, $(\tau(q_{n_k},\theta+c))_{k\in\N}$ is a sequence of elements in $\{-3,-1,1,3\}$. Since there are finitely many options, some of them must appear infinitely often. If $i$ appears infinitely often, then $\theta\in D_{i}$. Thus almost every $\theta$ belongs to some $D_i$.

Since $m_{\T}(D_{-3}\cup D_{-1}\cup D_{1}\cup D_{3})=1$, we can pick $\ell$ with $m_{\T}(D_\ell)>0$. Let us observe that $D_{\ell}$ is $R_{\alpha}$ invariant almost everywhere. Indeed, by \Cref{prop:cocycle-invariant-along-q-ns}, for almost all $\theta$ we have 
\[
(\tau(q_n,\theta),\tau(q_n,\theta+c))=(\tau(q_n,\theta+\alpha),\tau(q_n,\theta+\alpha+c)) \ \ \text{for all but finitely many $n$}
\]
Since $R_{\alpha}$ is ergodic, it follows that $D_{\ell}$ has full measure. 
\end{proof}
In the main application of \Cref{prop:ell-possible-values} (which is \Cref{prop:ell-existence}), we will discard the possibility that $\ell\in\{-3,3\}$, and thus we will be mostly interested in the case $\ell\in\{-1,1\}$.

For $\ell\in\{-1,1\}$ and $c\in\T$ we define 
\[
\zeta(\theta)=\tau(\theta+c)-\ell\tau(\theta)
\]
\begin{figure}
	\centering
	\begin{tikzpicture}[x=5cm,y=0.6cm]
		\def\c{0.05} 
		
		\draw[gray] (0,0) -- (1,0);
		
		\draw[very thick] (0,0) -- (0.5-\c,0);
		\draw[very thick] (0.5-\c,-2) -- (0.5,-2);
		\draw[very thick] (0.5,0) -- (1-\c,0);
		\draw[very thick] (1-\c,2) -- (1,2);
	\end{tikzpicture}
	\caption{The step function   $\zeta\colon\T\to\Z$ in the case $\ell=1$ and $c=0.05$. Note that $\zeta(\theta)$ is either $-2$, $0$, or $2$. }
	\label{fig:zeta}
\end{figure}
See  \Cref{fig:zeta}. This cocycle will appear naturally in the computations of \Cref{sec:isomorphisms}. For reasons that will become clear in \Cref{sec:isomorphisms}, we will be interested in understanding when $\zeta$ is a coboundary, and when it has nonzero finite essential values. 

A first observation is that $\zeta$ is identically null if and only if $\ell=1$ and $c=0$, or $\ell=-1$ and $c=1/2$. Furthermore, we have the following natural sufficient condition for $\zeta$ to be a coboundary. 
\begin{proposition}\label{prop:sufficient-conditions-zeta-coboundary}
If $\ell=1$ and $c=k\alpha$ for $k\in\Z$, then $\zeta$ is a coboundary. In fact, we can write $\zeta=f\circ R_{\alpha}-f$ for 
\[
f(\theta)=\tau(k,\theta)
\]
Similarly, if $\ell=-1$ and $c=k\alpha+1/2$ for $k\in\Z$, then $\zeta$ is a coboundary. In this case we can write $\zeta=f\circ R_{\alpha}-f$ for 
\[
f(\theta)=-\tau(k,\theta)
\]
\end{proposition}
\begin{proof}
We prove the first item. Let $f(\theta)=\tau(k,\theta)$. Expanding $\tau(k,\theta)$ as a sum, or alternatively, using the cocycle equation, we find 
\[
f(\theta+\alpha)-f(\theta)=\tau(\theta+k\alpha)-\tau(\theta)
\]
Thus $\zeta(\theta)=f(\theta+\alpha)-f(\theta)$ as claimed. Next, let us observe that \[
\tau(\theta+k\alpha+1/2)+\tau(\theta)=-\tau(\theta+k\alpha)+\tau(\theta)=-(\tau(\theta+k\alpha)-\tau(\theta))
\]
Thus the second case follows from the first.
\end{proof}
We will now make some comments about $\Z^2$-cocycles. We omit the corresponding definitions (which are analogous), as we will consider $\Z^2$-cocycles very briefly and only as a means to obtain further information about $\zeta$. 

Consider the $\Z^2$-cocycle for $R_{\alpha}$ given by
\[\omega\colon\T\to\Z^2, \ \theta\mapsto (\ell\tau(\theta),\tau(\theta+c))\]
Then $\zeta$ is naturally related to $\omega$, in the sense that $\zeta$ equals the difference of the coordinates of $\omega$. 
The following is proved\footnote{The original statement is with $\ell=1$, but the case $\ell=-1$ follows by exchanging $c$ and $c+1/2$.} by Yuqing Zhang. 
\begin{theorem}[Theorem 1.5 in \cite{zhang_ergodicity_2011}]\label{thm:zhang}
    Let $\alpha\in\R\smallsetminus\Q$. For almost every $c\in \T$, the group of essential values of $\omega$ is equal to \[\mathcal E(\omega)=\{(n,m)\in\Z^2 : n\equiv m\mod 2\}\]
    If $\alpha$ has bounded type, then this conclusion holds for every $c$ different from $k\alpha$ or $k\alpha+1/2$, $k\in\Z$.
\end{theorem}
We say that $\alpha$ has bounded type if the sequence $(a_n)_{n\in\N}$ in its continued fraction expansion  $\alpha=[a_0;a_1,\dots]$ is bounded. 

The key relation between essential values of $\omega$ and $\zeta$ is as follows: if $\omega$ admits a finite essential value $(n,m)$, then $m-n$ is an essential value for $\zeta$. This follows directly from the definitions. This allows us to prove the following statement about $\zeta$. 
\begin{proposition}\label{prop:bounded-type-coboundary-essential-values}
Let $\alpha\in\R\smallsetminus\Q$ have bounded type. Then the sufficient condition for $\zeta$ to be a coboundary stated in \Cref{prop:sufficient-conditions-zeta-coboundary} is also necessary. Furthermore, if $\zeta$ is not a coboundary, then $\mathcal E(\zeta)=2\Z$. 
\end{proposition}
\begin{proof}
Suppose that $c$ is not of the form $k\alpha$ or $k\alpha+1/2$, $k\in\Z$. As observed before, if $(n,m)$ is an essential value for $\omega$, then $m-n$ is an essential value for $\zeta$. By \Cref{thm:zhang}, every even number can be obtained in this manner. 

If $c$ is of the form $k\alpha$ or $k\alpha+1/2$, $k\in\Z$, then there are two cases where neither \Cref{prop:sufficient-conditions-zeta-coboundary} nor \Cref{thm:zhang} can be applied.  That is, the case 
 $c=k\alpha$ and $\ell=-1$, 
and the case $c=k\alpha+1/2$ and $\ell=1$. We show that in these cases $\mathcal E(\zeta)=2\Z$. 

For $c=k\alpha$ and $\ell=-1$ we can write $\zeta$ as  
\[\zeta(\theta)=\tau(\theta+k\alpha)+\tau(\theta)
\]
It follows that
\[
\zeta(\theta)-2\tau(\theta)=\tau(\theta+k\alpha)-\tau(\theta)
\]
\Cref{prop:sufficient-conditions-zeta-coboundary} shows that the term at the right is a coboundary. Thus we have proved that the difference $\zeta-2\tau$ is a coboundary. This implies that $\mathcal E(\zeta)$ is equal to $\mathcal E(2\tau)$ (see \cite{schmidt_cocycles_1977}). Since $\mathcal E(\tau)=\Z$, it is clear that then $\mathcal E(2\tau)=2\Z$. Thus $\mathcal E(\zeta)=2\Z$ as claimed. 

A similar argument shows that for $c=k\alpha+1/2$ and $\ell=1$, we have $\mathcal E(\zeta)=2\Z$.
\end{proof}
We note that the nontrivial direction in the previous result can be alternatively obtained from results of Conze and Piękniewska for $\Z^d$-valued cocycles \cite{conze_multiple_2014} (see the explanation in the proof of \cite[Theorem 3.4]{conze_remarks_2013}, which deals with the same cocycle, denoted $u_{\gamma}$). 
\section{Easy isomorphisms}\label{sec:easy-isomorphisms}
The purpose of this section is to point out elementary isomorphisms and automorphisms for the skew products we study.
\begin{proposition}\label{prop:easy-automorphisms}
    Let $\alpha\in\R\smallsetminus\Q$ and let $(X,\mathcal B_X,\mu,T)$ be an invertible measure-preserving system. Then $R_{\alpha}\rtimes_{\tau} T$, $R_{\alpha}\rtimes_{\tau} T^{-1}$, $R_{-\alpha}\rtimes_{\tau} T$, and $R_{-\alpha}\rtimes_{\tau} T^{-1}$ are all  isomorphic to each other. Indeed:
    \begin{enumerate}
    \item $R_{\alpha}\rtimes_{\tau} T$ is isomorphic to $R_{\alpha}\rtimes_{\tau} T^{-1}$ via $(\theta,x)\mapsto(\theta+1/2,x)$. 
        \item $R_{\alpha}\rtimes_{\tau} T$ is isomorphic to $R_{-\alpha}\rtimes_{\tau} T$ via $(\theta,x)\mapsto(-\theta+1/2,x)$. 
\end{enumerate}\end{proposition}
\begin{proof}
Item (1) follows directly from the symmetry $\tau(\theta+1/2)=-\tau(\theta)$, $\theta\in\T$. The complete argument is left to the reader. Alternatively, (1) is a particular case of the more general  \Cref{prop:easy-isomorphisms-cohomology-section}. We now prove (2). 
Let us denote by $\Upsilon\colon\T\times X\to \T\times X$ the map 
\[
(\theta,x)\mapsto (-\theta+1/2,x)
\]
Then $\Upsilon$ is bijective and preserves $m_{\T}\times\mu$. Furthermore, observe that we have the identity \[\tau(-\theta)=-\tau(\theta), \ \ \theta\in\T\smallsetminus\{0,1/2\}\]
Since $\tau(\theta+1/2)=-\tau(\theta)$ for all $\theta$, we obtain
\[\tau(-\theta+1/2)=\tau(\theta), \ \ \theta\in\T\smallsetminus\{0,1/2\}\]
We now verify that  $\Upsilon\circ(R_{\alpha}\rtimes_{\tau} T) = (R_{-\alpha}\rtimes_{\tau} T)\circ\Upsilon$ almost everywhere. For $(\theta,x)$ with $\theta\not\in\{0,1/2\}$ we have:
\begin{align*}
    (R_{-\alpha}\rtimes_{\tau} T)(\Upsilon(\theta,x))&=(R_{-\alpha}\rtimes_{\tau} T)(-\theta+1/2,x)\\
    &=(-\theta+1/2-\alpha,T^{\tau(-\theta+1/2)}(x))\\
    &=(-\theta+1/2-\alpha,T^{\tau(\theta)}(x)) &\text{(because $\tau(-\theta+1/2)=\tau(\theta)$)}\\
    &=\Upsilon(\theta+\alpha,T^{\tau(\theta)}(x))\\
    &=\Upsilon((R_{\alpha}\rtimes_{\tau}T)(\theta,x))
\end{align*}
\end{proof}
We can now prove the easy direction of \Cref{thm:isomorphisms}.
\begin{proposition}\label{prop:isomorphism-theorem-easy-direction}
Let $\alpha,\beta\in\R\smallsetminus \Q$ and let $(Y,\mathcal B_Y,\nu,S)$ and $(X,\mathcal B_X,\mu,T)$ be invertible measure-preserving systems. A sufficient condition for  $R_{\alpha}\rtimes_{\tau} S$ to be isomorphic to $R_{\beta}\rtimes_{\tau} T$ is that $\alpha=\pm\beta\mod 1$, and furthermore $S$ is flip isomorphic to $T$.
\end{proposition}
\begin{proof}
It is clear that if $S$ is isomorphic to $T$, then $R_{\alpha}\rtimes_{\tau} S$ is isomorphic to $R_{\alpha}\rtimes_{\tau} T$. Indeed, if $\upsilon$ is an isomorphism from $S$ to $T$, then $(\theta,x)\mapsto (\theta,\upsilon(x))$ is an isomorphism from $R_{\alpha}\rtimes_{\tau} S$ to $R_{\alpha}\rtimes_{\tau} T$. By \Cref{prop:easy-automorphisms}, it follows that $R_{\alpha}\rtimes_{\tau} S$ is also isomorphic to  $R_{\alpha}\rtimes_{\tau} T^{-1}$, $R_{-\alpha}\rtimes_{\tau} T$, and $R_{-\alpha}\rtimes_{\tau} T^{-1}$.\end{proof} 

Provided $S$ and $T$ are flip isomorphic,  we have a natural collection of isomorphisms from $R_{\alpha}\rtimes_{\tau} S$ to $R_{\alpha}\rtimes_{\tau} T$ given by the following proposition. 
\begin{proposition}\label{prop:easy-isomorphisms-cohomology-section}
    Let $\alpha\in\R\smallsetminus\Q$. Let $c\in\T$, let $\ell\in\{-1,1\}$, and consider the step function
    \[
\zeta(\theta)=\tau(\theta+c)-\ell\tau(\theta)
    \]
    Suppose that $\zeta$ is a coboundary, and let $f\colon\T\to\Z$ be a measurable solution to $\zeta=f\circ R_{\alpha}-f$. 

    Let $(Y,\mathcal B_Y,\nu,S)$ and $(X,\mathcal B_X,\mu,T)$ be  invertible measure-preserving systems such that $S$ is isomorphic to $T^{\ell}$, and let $\upsilon$ be such an isomorphism. 
    Then  
    \[(\theta,x)\mapsto (\theta+c,T^{f(\theta)}(\upsilon(x)))\]
    defines an isomorphism from $R_{\alpha}\rtimes_{\tau} S$ to $R_{\alpha}\rtimes_{\tau} T$. 
\end{proposition}
\begin{proof}
    Let us denote by $\Upsilon$ the map 
    \[
(\theta,x)\mapsto (\theta+c,T^{f(\theta)}(\upsilon(x)))
    \]
 It is easy to see that since $\upsilon$ sends $\nu$ to $\mu$, $\Upsilon$ sends $m_{\T}\times\nu$ to $m_{\T}\times\mu$.  Let  $g(\theta)=-\ell f(\theta-c)$, and let $\Upsilon'$ be the map given by
\[
(\theta,x)\mapsto (\theta-c,S^{g(\theta)}(\upsilon^{-1}(x))) 
\]
Observe that $\Upsilon'$  sends $m_{\T}\times\mu$ to $m_{\T}\times \nu$. Furthermore, a direct computation shows that $\Upsilon\circ\Upsilon'=\id_{\T\times X}$ and $\Upsilon'\circ\Upsilon=\id_{\T\times Y}$ almost everywhere.  

It remains to verify that $\Upsilon\circ( R_{\alpha}\rtimes_{\tau} S)=(R_\alpha\rtimes_{\tau} T)\circ \Upsilon$. Note that since $\upsilon\circ S=T^{\ell}\circ\upsilon$, we also have $\upsilon\circ S^{\tau(\theta)}=T^{\ell\tau(\theta)}\circ\upsilon$ for all $\theta$. Next, for almost every $(\theta,x)$ we have 
\begin{align*}
    \Upsilon((R_{\alpha}\rtimes_{\tau} S)(\theta,x))&=\Upsilon(\theta+\alpha,S^{\tau(\theta)}(x))\\
    &=(\theta+\alpha+c,T^{f(\theta+\alpha)}(\upsilon(S^{\tau(\theta)}(x))))\\
    &= (\theta+\alpha+c,T^{f(\theta+\alpha)}T^{\ell\tau(\theta)}(\upsilon(x)))\\
    &=(\theta+\alpha+c,T^{f(\theta+\alpha)+\ell\tau(\theta)}(\upsilon(x)))
\end{align*}
On the other hand, 
\begin{align*}
    (R_{\alpha}\rtimes_{\tau} T)(\Upsilon(\theta,x))&=(R_{\alpha}\rtimes_{\tau} T)(\theta+c, T^{f(\theta)}(\upsilon(x)))\\
    &=(\theta+\alpha+c,T^{\tau(\theta+c)}(T^{f(\theta)}(\upsilon(x))))\\
    &=(\theta+\alpha+c,T^{\tau(\theta+c)+f(\theta)}(\upsilon(x)))
\end{align*}
But $f(\theta+\alpha)+\ell\tau(\theta)=\tau(\theta+c)+f(\theta)$ almost everywhere by our assumption on $f$.  
\end{proof}
\section{Eigenvalues and the Kronecker factor of $R_{\alpha}\rtimes_{\tau} T$}\label{sec:kronecker}
The purpose of this section is to determine the eigenvalues of the Koopman operator for $R_{\alpha}\rtimes_{\tau} T$, and to determine the Kronecker factor of $R_{\alpha}\rtimes_{\tau} T$ for $T$ ergodic.

These results have two important applications for us. In the ergodic case, it will follow that any possible factor map, and in particular isomorphism, from $R_{\alpha}\rtimes_{\tau} S$ to $R_{\alpha}\rtimes_{\tau} T$, must be a rotation on $\T$ in its first coordinate (\Cref{prop:isomorphism-nice-first-coordinate}). Furthermore, we will use these results in \Cref{sec:different-angles} to show that a necessary condition for $R_{\alpha}\rtimes_{\tau} S$ and $R_{\beta}\rtimes_{\tau} T$ to be isomorphic is that $\alpha=\pm \beta\mod 1$.

The Koopman operator $U_T$ of a measure-preserving system  $(X,\mathcal B_X,\mu,T)$ is defined as the linear operator  $F\mapsto F\circ T$ acting on the space $L^2_{\C}(X,\mathcal B_X,\mu)$  of complex valued $L^2$-integrable functions over $(X,\mathcal B_X,\mu)$. The following result is well known.
\begin{proposition}\label{prop:eigenvalues-rotation}
    Let $\alpha\in\R\smallsetminus\Q$. The set of eigenvalues of the Koopman operator of $R_{\alpha}$ is $\{e^{2\pi i k \alpha} : k\in\Z\}$. 
\end{proposition}
We will now compute the set of eigenvalues of the Koopman operator for $R_{\alpha}\rtimes_{\tau} T$. 
\begin{theorem}\label{thm:eigenvalues}
Let $\alpha\in\R\smallsetminus\Q$ and let  $(X,\mathcal B_X,\mu,T)$ be an invertible measure-preserving system. The set of eigenvalues for the Koopman operator of $R_{\alpha}\rtimes_{\tau} T$ is $\{e^{2\pi i k\alpha} : k\in\Z\}\cup \{-e^{2\pi i k\alpha} : k\in\Z\}$, provided $-1$ is an eigenvalue of the Koopman operator of $T$, and $\{e^{2\pi i k\alpha} : k\in\Z\}$ otherwise. 
\end{theorem}
\begin{proof}
The easy part of the proof is verifying that $U_{R_{\alpha}\rtimes_{\tau} T}$ has at least as many eigenvalues as claimed in the statement (for inclusion). Let $f\colon\T\to\C$ be an eigenfunction for $R_{\alpha}$ with eigenvalue $\lambda$. Then clearly $(\theta,x)\mapsto f(\theta)$ is an eigenfunction for $R_{\alpha}\rtimes_{\tau} T$ with eigenvalue $\lambda$. By \Cref{prop:eigenvalues-rotation}, it follows that the set of eigenvalues for $U_{R_{\alpha}\rtimes_{\tau} T}$ contains $\{e^{2\pi i k\alpha} : k\in\Z\}$.

Now suppose that $-1$ is an eigenvalue for $U_T$, and let $g$ be an eigenfunction for $U_T$ with eigenvalue $-1$. Thus $g\circ T=-g$. Since $g\circ T\circ T^{-1}=-g\circ T^{-1}$, it follows that $g\circ T^{-1}=-g$. We claim that then  $(\theta,x)\mapsto F(\theta,x)=f(\theta)\cdot g(x)$ is an eigenfunction for $U_{R_{\alpha}\rtimes_{\tau} T}$ with eigenvalue $-\lambda$. Indeed, for almost every $(\theta,x)$ we have 
\[
F(\theta+\alpha,T^{\tau(\theta)}(x))=f(\theta+\alpha)\cdot g(T^{\tau(\theta)}(x))=-\lambda f(\theta)g(x)
\]
Thus, if $-1$ is an eigenvalue for $U_T$, the set of eigenvalues for $U_{R_{\alpha}\rtimes_{\tau} T}$ also contains $\{-e^{2\pi i k\alpha} : k\in\Z\}$.

Next we will prove that $U_{R_{\alpha}\rtimes_{\tau} T}$ has at most as many eigenvalues as claimed in the statement (for inclusion). Let $\lambda\in \mathbb S=\{z\in\C : |z|=1\}$ be an arbitrary eigenvalue for $U_{R_{\alpha}\rtimes_\tau T}$. Let $F\in L^2_{\C}(\T\times X,\mathcal B_{\T}\otimes \mathcal B_X,m_{\T}\times\mu)$ be an eigenfunction for $\lambda$, so
    \begin{equation}\label{eq:eigenfunction}F\circ (R_{\alpha}\rtimes_\tau T) =\lambda\cdot F\end{equation}
    We write  $F_\theta(x)=F(\theta,x)$. Since $F$ is $L^2$ integrable over  $(\T\times X,\mathcal B_{\T}\otimes \mathcal B_X,m_{\T}\times\mu)$, it follows from Fubini's Theorem that $F_\theta $ is $L^2$ integrable over $(X,\mathcal B_X,\mu)$ for almost every $\theta\in \T$.  We have an almost everywhere defined measurable function 
    \[
    \T\to L^2_{\C}(X,\mathcal B_X,\mu), \ \theta\mapsto F_{\theta}
    \]
    Let  $(q_n)_{n\in\N}$ be a sequence as in  \Cref{prop:nice-sequence-of-odd-numbers}.   Since $\mathbb S$ is compact, the sequence $(\lambda^{q_n})_{n\in\N}$ has an accumulation point $\lambda_*\in\mathbb S$. We will show that 
    \begin{equation}\label{eq:accumulation-point}
        F_{\theta}\circ T=\lambda_* F_{\theta} \ \text{for almost every $\theta$}
    \end{equation}
    Replacing $(q_n)_{n\in\N}$ by a subsequence, we may assume that  $\lim_{n\to\infty}\lambda^{q_n}=\lambda_*$. Furthermore, replacing $(q_n)_{n\in\N}$ by a further subsequence, we can assume that for almost every $\theta\in \T$ we have $||F_{\theta +q_{n}\alpha}- F_{\theta}||_{L^2}\to 0$ as $n\to\infty$. This is possible thanks to \Cref{prop:subsequences}, the measurability of $\theta\to F_{\theta}$, and the fact that $q_n\alpha\to0$ in $d_{\T}$ as $n\to\infty$.
    
    Since $F\circ (R_{\alpha}\rtimes_\tau T) =\lambda\cdot F$, we have $F\circ (R_{\alpha}\rtimes_\tau T)^n =\lambda^n\cdot F$ for $n\geq 1$, and then  \[F(\theta+n\alpha,T^{\tau(n,\theta)}(x))=\lambda^n F(\theta,x)\]
    for almost every $(\theta,x)$. Hence for almost every $\theta\in\T$ we can write \begin{equation}\label{eq:compositions} 
    F_{\theta+n\alpha}\circ T^{\tau(n,\theta)} = \lambda^n F_{\theta}
    \end{equation}
For almost every $\theta\in\T$, we can argue as follows. By \Cref{prop:infinitely-many-1s}, we can choose a subsequence $(q_{n_k})_{k\in\N}$ of $(q_{n})_{n\in\N}$ so that $\tau(q_{n_k},\theta)=1$ for all $k$. Hence 
    \begin{equation}\label{eq:compositions-q-n-k-ones}
    F_{\theta+q_{n_{k}}\alpha}\circ T = \lambda^{q_{n_{k}}} F_{\theta}
    \end{equation}
    Recall that a subsequence of a convergent sequence converges to the same limit. Since $\lim_{n\to\infty}||F_{\theta +q_{n}\alpha}- F_{\theta}||_{L^2}=0$,  we also have $\lim_{k\to\infty}||F_{\theta +q_{n_{k}}\alpha}- F_{\theta}||_{L^2}= 0$. Since $U_T$ is an isometry for the $L^2$ norm, we also have $\lim_{k\to\infty}||F_{\theta +q_{n_{k}}\alpha}\circ T - F_{\theta}\circ T||_{L^2}=0$. By \Cref{eq:compositions-q-n-k-ones} it follows that  $\lim_{k
    \to\infty}||\lambda^{q_{n_{k}}} F_{\theta}-F_{\theta}\circ T||_{L^2}= 0 $. But $\lambda^{q_{n_{k}}}$ converges to $\lambda_*$ as $k\to\infty$, so we obtain that $||\lambda_{*} F_{\theta}-F_{\theta}\circ T ||_{L^2}=0$. Hence $F_{\theta}\circ T = \lambda_* F_{\theta}$ for almost every $\theta$ as claimed.

    The same argument shows that 
    \begin{equation}\label{eq:accumulation-point-inverse}
        F_{\theta}\circ T^{-1}=\lambda_* F_{\theta} \ \text{for almost every $\theta$}
    \end{equation}The only modification required is that when we apply \Cref{prop:infinitely-many-1s}, we need to choose a subsequence $(q_{n_{k}})_{k\in\N}$ such that $\tau(q_{n_{k}},\theta)=-1$ for all $k$. 
    
    For almost every $\theta\in\T$ we can write 
\[
        F_{\theta}= (F_{\theta}\circ T)\circ T^{-1}= \lambda_{\ast} \cdot F_{\theta} \circ T^{-1}= \lambda_* ^2 F_{\theta}
\]
     Since $F$ is not identically null, some $F_{\theta}$ is not identically null. Therefore $\lambda_{*}^2=1$, and $\lambda_*\in\{-1,1\}$.

     For almost every $\theta\in\T$, the following holds. By  \Cref{eq:accumulation-point,eq:accumulation-point-inverse} we have  $F_{\theta+\alpha}\circ T=\lambda_* F_{\theta+\alpha}$ and  $F_{\theta+\alpha}\circ T^{-1}=\lambda_* F_{\theta+\alpha}$. Thus $F_{\theta+\alpha}\circ T^{\tau(\theta)}=\lambda_*F_{\theta+\alpha}$. On the other hand, by  \Cref{eq:compositions} we have  $F_{\theta+\alpha}\circ T^{\tau(\theta)}=\lambda F_{\theta}$. Therefore $\lambda_*F_{\theta+\alpha}=\lambda F_{\theta}$. Since $\lambda_*\in\{-1,1\}$ is its own multiplicative inverse, we obtain \begin{equation}\label{eq:rotation-invariance-kronecker-slices}F_{\theta+\alpha}=\lambda_*\lambda F_{\theta} \ \text{ for almost every $\theta$}\end{equation}
    Thus for a positive measure set of $x\in X$ the function 
    \[
\T\to\C, \theta\to F(\theta,x)
    \]
     is an eigenfunction for $U_{R_{\alpha}}$, with eigenvalue $\lambda_*\lambda$.

     Thus we started with an arbitrary eigenvalue  $\lambda$  for $U_{R_{\alpha}\rtimes_{\tau} T}$, and we proved that $\lambda_*\lambda$ is an  eigenvalue of $U_{R_{\alpha}}$. Since $\lambda_*\in\{-1,1\}$, it follows from \Cref{prop:eigenvalues-rotation} that $\lambda$ belongs to $\{e^{2\pi i k\alpha} : k\in\Z\}\cup \{-e^{2\pi i k\alpha} : k\in\Z\}$. Since $\lambda_*$ is an eigenvalue for $U_T$ (\Cref{eq:accumulation-point}), if $-1$ is not an eigenvalue for $U_T$, the only possibility is $\lambda_*=1$. Thus in that case $\lambda$ belongs to $\{e^{2ki\pi\alpha} : k\in\Z\}$.
\end{proof}
\begin{remark}
In the ergodic case, one can alternatively compute the eigenvalues of $U_{R_{\alpha}\rtimes_{\tau} T}$ by combining the previous argument with a general result from \cite{lemanczyk_lifting_2012b}. For $t\in [0,1)$ define $A_t$ as the set of those $\chi\in \mathbb S$ such that for some $f\colon \T\to \mathbb S$ measurable we have 
\[
 \chi^{\tau(\theta)}=e^{2\pi i t} \frac{f(\theta)}{f(\theta+\alpha)} \ \text{for almost every $\theta$}
\]
This equation can be iterated to obtain 
\[
\chi^{\tau(n,\theta)}=(e^{2\pi i t})^n \frac{f(\theta)}{f(\theta+n\alpha)} \ \text{for almost every $\theta$, $n\geq 1$,}
\]
For this equation, one can repeat the subsequence argument from  \Cref{thm:eigenvalues} (for $\lambda=e^{2\pi i t}$), and derive that $A_t\subset\{-1,1\}$, $t\in[0,1)$. But $1\in A_t$ if and only if $e^{2\pi i t}$ is an eigenvalue for $R_{\alpha}$, and $-1\in A_t$ if and only if $-e^{2\pi i t}$ is an eigenvalue for $R_{\alpha}$, so
\[
A_t=\begin{cases}
	\{1\}, \  \ t=k\alpha \mod 1, k\in\Z \\
	\{-1\},  \ t=k\alpha +1/2 \mod 1, k\in\Z\\
	\emptyset, \ \text{otherwise} 
\end{cases}
\]
Proposition 5 in \cite{lemanczyk_lifting_2012b} states that $e^{2\pi i t}$ is an eigenvalue for $U_{R_{\alpha}\rtimes_{\tau} T}$ if and only if $A_t$ contains an eigenvalue of $U_{T}$. From this characterization we recover the fact that  $U_{R_{\alpha}\rtimes_{\tau} T}$ always has the eigenvalue $e^{2\pi i k \alpha}$, while $e^{2\pi i (k \alpha+1/2)}=-e^{2\pi i k \alpha}$ is present exactly when $-1$ is an eigenvalue for $U_T$ ($k\in\Z$). 
\end{remark}
Next we record the following fact. 
\begin{proposition}\label{prop:characterization-ergodicity}
Let $\alpha\in\R\smallsetminus\Q$ and let $(X,\mathcal B_{X},\mu,T)$ be an invertible measure-preserving system. Then $R_{\alpha}\rtimes_{\tau} T$ is ergodic if and only if $T$ is ergodic. 
\end{proposition}
\begin{proof}
It is easy to see directly that whenever $T$ is not ergodic, then $R_{\alpha}\rtimes_{\tau} T$ is not ergodic. Indeed, suppose that $T$ is not ergodic. Let $E\subset X$ be $T$-invariant with $0<\mu(E)<1$.   Then $\T\times E$ is $R_{\alpha}\rtimes_{\tau} T$-invariant, and $m_{\T}\times\mu(\T\times E)=\mu(E)$. 

Suppose now that $T$ is ergodic. By a general fact about Rokhlin cocycle extensions \cite[Corollary 5]{lemanczyk_ergodicity_2001}, the ergodicity of $\tau$ ensures that then  $R_{\alpha}\rtimes_{\tau} T$ is ergodic. Alternatively, a self-contained proof follows from the argument in \Cref{thm:eigenvalues}. Indeed, let $F\in L^2_{\C}(\T\times X,\mathcal B_{\T}\otimes \mathcal B_X,m_{\T}\times\mu)$ be an arbitrary invariant function for $R_{\alpha}\rtimes_{\tau} T$, so it is an eigenfunction for $U_{R_{\alpha}\rtimes_{\tau} T}$ with eigenvalue $\lambda=1$. Since $\lambda^n=1$ for all $n$, the value $\lambda_*$ defined in the proof of \Cref{thm:eigenvalues} is equal to $1$. Then \Cref{eq:accumulation-point,eq:rotation-invariance-kronecker-slices} show that  $F_\theta\circ T= F_{\theta}$  and  $F_{\theta+\alpha}=F_{\theta}$ for almost every $\theta$. By the ergodicity of $T$ and $R_{\alpha}$,  it follows that $F$ is constant almost everywhere.%
\end{proof}
We will now determine the Kronecker factor of $R_{\alpha}\rtimes_{\tau} T$ for $T$ ergodic. We start by recalling the relevant definitions. By a Kronecker system we mean an ergodic measure-preserving system which is isomorphic to a rotation in a compact abelian group, endowed with its Borel sigma algebra and Haar measure. Among the factors of an ergodic system that are Kronecker systems, there exists one which is maximal for the factor relation. We summarize the relevant information for us in the following proposition-definition (see \cite[\S 6.4]{einsiedler_ergodic_2011} and \cite[Proposition 14]{host_nilpotent_2018a}).
\begin{proposition}[and definition of Kronecker factor]\label{prop:characterization-kronecker-factor}
    Let $(X,\mathcal B_X,\mu,T)$ be an ergodic invertible measure-preserving system. There exists a  Kronecker system $(G,\mathcal B_{G},m_G,R)$ which is a factor of $(X,\mathcal B_X,\mu,T)$, and such that the following hold. \begin{itemize}
        \item Every Kronecker system which is a factor of $T$ is also a factor of $R$.
        \item  Let $\pi$ be a factor map from $T$ to $R$, and let $\pi'$ be a factor map from $T$ to some other Kronecker system $R'$. Then there exists a factor map $f$ from $R$ to $R'$ such that $\pi'=f\circ \pi$. This is illustrated in  \Cref{fig:kronecker-factor}. 
    \end{itemize}
    The system $(G,\mathcal B_{G},m_G,R)$ is unique up to isomorphism, so we refer to it as ``the''  Kronecker factor of $T$. 

    Furthermore, a necessary and sufficient condition for a Kronecker system to be the Kronecker factor of $T$ is that their Koopman operators have the same set of eigenvalues. 
\end{proposition}
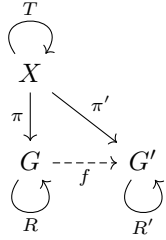
\begin{figure}[h]
    \centering
    \begin{tikzcd}
X \arrow[d, "\pi"'] \arrow["T", loop, distance=2em, in=55, out=125] \arrow[rd, "\pi'"] &                                                       \\
G \arrow["R"', loop, distance=2em, in=305, out=235] \arrow[r, "f"', dashed]            & G' \arrow["R'"', loop, distance=2em, in=305, out=235]
\end{tikzcd}
\caption{The measure-preserving transformation $T$ and its Kronecker  factor $R$ via the factor map $\pi$. Any factor from $T$ to some Kronecker system $R'$ factorizes through $\pi$, meaning that $\pi'=f\circ \pi$ for a factor map $f$ from $R$ to $R'$. The choice of $\pi$ is not relevant in the sense that any factor map from $T$ to $R$ has the same property.}
    \label{fig:kronecker-factor}
\end{figure}
Let $\Z/2\Z$ be the group of integers modulo $2$, and let $R_1\colon \Z/2\Z\to\Z/2\Z$ be the transformation which adds $1$ modulo $2$. We endow $\Z/2\Z$ with its Haar measure $m_{\Z/2\Z}$ and Borel sigma-algebra $\mathcal B_{\Z/2\Z}$. Then $R_1$ is an ergodic measure-preserving transformation on the probability space $(\Z/2\Z,\mathcal B_{\Z/2\Z},m_{\Z/2\Z})$. We will denote by $R_{\alpha}\times R_1$ the transformation \begin{equation}\label{eq:def-R-1-times-R-alpha}(\theta,x)\mapsto (\theta+\alpha,x+1)\end{equation}
acting on  
$
       (\T\times \Z/2\Z, \mathcal B_{\T}\otimes \mathcal B_{\Z/2\Z},m_{\T}\times m_{\Z/2\Z})
       $.  In what follows we always regard $\T\times \Z/2\Z$ as an abelian group with componentwise addition.
\begin{proposition}\label{prop:eigenvalues-R-alpha-times-R-1}
    The set of eigenvalues of the Koopman operator for $R_{\alpha}\times R_1$ is  $\{e^{2\pi i k\alpha} : k\in\Z\}\cup \{-e^{2\pi i k\alpha} : k\in\Z\}$. 
\end{proposition}
\begin{proof}
While this result admits an elementary proof, we can obtain it from \Cref{thm:eigenvalues} by taking $T=R_1$. Since $R_1$ equals its own inverse, the skew product $R_{\alpha}\rtimes_{\tau} R_1$ degenerates to a direct product $R_{\alpha}\times R_1$.  
\end{proof}
We can now determine the Kronecker factor of $R_{\alpha}\rtimes_{\tau} T$, for ergodic $T$. 
\begin{theorem}\label{thm:kronecker-factor}
   Let $\alpha\in\R\smallsetminus\Q$ and let $(X,\mathcal B_X,\mu,T)$ be an ergodic invertible measure-preserving system. 
    \begin{enumerate}
       \item If $-1$ is not an eigenvalue for $U_T$, then the Kronecker factor of  $R_{\alpha}\rtimes_{\tau} T$ is $R_{\alpha}$.
       \item If $-1$ is an eigenvalue for $U_T$, then the Kronecker factor of $R_{\alpha}\rtimes_{\tau} T$ is $R_{\alpha}\times R_1$ (defined in  \Cref{eq:def-R-1-times-R-alpha}). 
   \end{enumerate}
\end{theorem}
\begin{proof}
	Since $T$ is ergodic, so is $R_{\alpha}\rtimes_{\tau} T$ (\Cref{prop:characterization-ergodicity}). Then the claim follows from \Cref{thm:eigenvalues}, \Cref{prop:eigenvalues-rotation} and \Cref{prop:eigenvalues-R-alpha-times-R-1}. 
\end{proof}
In the statement of the previous result, the hypothesis that $U_T$ does not have $-1$ as an eigenvalue can be replaced by the assertion that $T^2$ is ergodic, or that $T$ does not factor onto $R_1$. We record this elementary equivalence without proof.
\begin{proposition}\label{prop:two-point-system}
Let $(X,\mathcal B_X,\mu,T)$ be an ergodic invertible measure-preserving system. The following conditions are equivalent.
\begin{enumerate}
    \item $T$ factors onto $R_1$.
    \item $T^2$ is not ergodic. 
    \item $-1$ is an eigenvalue for $U_T$. 
\end{enumerate} \end{proposition} 



We will also need the following result about factors between Kronecker systems.
\begin{proposition}\label{prop:algebraic-factors}
Let $G$ and $H$ be compact abelian groups, and let $\rho\colon G\to H$ be a measurable group homomorphism. Let $a\in G$ and let $b=\rho(a)$.

Consider the measure-preserving systems $(G,\mathcal B_G,m_G,R_a)$ and $(H,\mathcal B_H,m_H,R_b)$. Here $m_G$ and $m_H$ are Haar measures, and $R_a$ and $R_b$ are defined by $R_a(g)=g+a$, $g\in G$, and   $R_b(h)=h+b$, $h\in H$. If $R_a$ is ergodic, then every measure-theoretic factor map from $R_a$ to $R_b$ can be written, up to a null set, as \[g\mapsto \rho(g)+c\]
for some $c\in H$. 
\end{proposition}
\begin{proof}
    Let $f$ be a measure-theoretic factor map from $R_a$ to $R_b$, so $f\circ R_a=R_b\circ f$ almost everywhere, and define $f'\colon G\to H$ by $f'(g)=\rho(g)-f(g)$, $g\in G$. Then  $f'\circ R_a=f'$ almost everywhere. Indeed:
    \[
    f'(R_a(g))=\rho(R_a(g))-f(R_a(g))=\rho(R_a(g))-R_b(f(g))=(\rho(g)+b)-(f(g)+b)=\rho(g)-f(g)
    \]
    Since $R_a$ is ergodic, it follows that $f'$ is constant almost everywhere, and our claim follows.
\end{proof}
We can now prove the following result, which provides information about the first coordinate of an arbitrary factor map from one skew product to another, provided their bases have the same rotation angles. 
\begin{proposition}\label{prop:isomorphism-nice-first-coordinate}
    Let $\alpha\in\R\smallsetminus\Q$ and let $(X,\mathcal B_X,\mu,T)$ and $(Y,\mathcal B_Y,\nu,S)$ be invertible measure-preserving systems. We  assume that $S$ is ergodic. Suppose that $R_{\alpha}\rtimes_\tau S$ factors onto $R_{\alpha}\rtimes_\tau T$, and let $\Upsilon$ be an arbitrary factor map from $R_{\alpha}\rtimes_\tau S$ to $R_{\alpha}\rtimes_\tau T$. Then, up to a null set, $\Upsilon$ can be written as
    \[(\theta,x)\mapsto(\theta+c,\psi(\theta,x)).\]
    That is, the first coordinate of $\Upsilon$ depends only on the first coordinate, and it is necessarily a rotation by some $c\in\T$. 
\end{proposition}
\begin{proof}
We start by writing $\Upsilon$ as 
\[
\Upsilon(\theta,x)=(\phi(\theta,x),\psi(\theta,x))
\]
Here $\phi\colon\T\times Y\to\T$ and $\psi\colon\T\times Y\to X$ are measurable. 
By \Cref{thm:kronecker-factor}, the Kronecker factor of $R_{\alpha}\rtimes_{\tau}S$ is either $R_{\alpha}$ or $R_{\alpha}\times R_1$. Let us first consider the case of $R_{\alpha}\times R_1$. We define a factor map $\pi_S$ from  $R_{\alpha}\rtimes_{\tau} S$ to its Kronecker factor  by 
\[
\pi_S\colon\T\times Y\to\T\times \Z/2\Z, \ (\theta,x)\mapsto (\theta,\iota(x))
\] 
Here $\iota\colon Y\to\Z/2\Z$ is a factor map from $S$ to $R_1$, which exists by \Cref{prop:two-point-system}. An elementary verification shows that $\pi_S$ is indeed a factor map, and it intertwines the actions of $R_{\alpha}\rtimes_{\tau} S$ and $R_\alpha\times R_1$.

Furthermore, let us define the factor map from $R_{\alpha}\rtimes_{\tau} T$ to $R_{\alpha}$, given by the projection to the first coordinate 
\[
\pi_T\colon\T\times X\to \T, \ (\theta,x)\mapsto \theta
\]
Observe that the composition  $\pi_T\circ\Upsilon$ defines a factor map from $R_{\alpha}\rtimes_{\tau} S$ to $R_{\alpha}$.

\begin{figure}[h]
	\begin{center}
    \begin{tikzcd}
\mathbb{T}\times Y \arrow[d, "\pi_S"'] \arrow["R_{\alpha}\rtimes_{\tau}S", loop, distance=2em, in=55, out=125] \arrow[r, "\Upsilon"]   & \mathbb{T}\times X \arrow["R_{\alpha}\rtimes_{\tau} T", loop, distance=2em, in=55, out=125] \arrow[d, "\pi_T"] \\
\mathbb{T}\times \mathbb{Z}/2\mathbb{Z} \arrow["R_{\alpha}\times R_1"', loop, distance=2em, in=305, out=235] \arrow[r, "f"', dashed] & \mathbb{T} \arrow["R_{\alpha}"', loop, distance=2em, in=305, out=235]                                        
\end{tikzcd}
	\end{center}\caption{The system $R_{\alpha}\rtimes_{\tau} S$, the factor maps $\pi_S$ and $\pi_T\circ\Upsilon$, and the factor map $f$ satisfying $\pi_T\circ\Upsilon=f\circ\pi_S$, which exists because $R_{\alpha}\times R_1$ is the Kronecker factor of $R_{\alpha}\rtimes_{\tau} S$. It follows from \Cref{prop:algebraic-factors} that $f(\theta,x)=\theta+c$ almost everywhere, for some $c\in\T$. }\label{fig:rotation-kronecker}\end{figure}
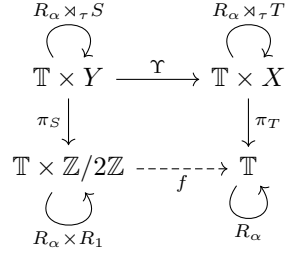
Thus we have two factor maps from  $R_{\alpha}\rtimes_{\tau} S$, $\pi_S$ goes to its Kronecker factor, while $\pi_T\circ\Upsilon$ goes to some other  Kronecker system. By item 2 in \Cref{prop:characterization-kronecker-factor}, there exists a factor map $f\colon\T\times \Z/2\Z\to\T$ from $R_{\alpha}\times R_1$ to $R_{\alpha}$ such that $ f\circ\pi_S = \pi_T\circ\Upsilon$. See \Cref{fig:rotation-kronecker}.  

We now apply \Cref{prop:algebraic-factors} to the group homomorphism  $\rho\colon\T\times \Z/2\Z\to\T$ defined by $(\theta,x)\mapsto \theta$. We take $a=(\alpha,1)$, $b=\alpha$, and the factor map $f$ from $R_a$ to $R_b$ defined in the previous paragraph. Also observe that $R_{\alpha}\times R_1$ coincides with the rotation by $a$ on $\T\times \Z/2\Z$. It follows that there exists $c\in \T$ such that $f(\theta,x)=\theta+c$ for almost every $(\theta,x)\in\T\times\Z/2\Z$. Evaluating this new information about $f$ in the relation $f\circ \pi_S= \pi_T\circ\Upsilon$, we find that $\phi(\theta,x)=\theta+c$ for almost every $(\theta,x)$. Thus our claim holds. 

The case in which the Kronecker factor of $R_{\alpha}\rtimes_{\tau} S$ is $R_{\alpha}$ follows by similar reasoning. In that case one takes $\pi_S\colon\T\times Y\to\T$ as the projection to the first coordinate,  and applies \Cref{prop:algebraic-factors} taking $\rho\colon\T\to\T$ as the identity map. 
\end{proof}
\begin{proposition}\label{prop:kronecker-characterization}
Let $\alpha\in\R\smallsetminus\Q$ and let  $(X,\mathcal B_X,\mu,T)$ be an ergodic invertible measure-preserving system. Then $R_{\alpha}\rtimes_{\tau} T$ is a Kronecker system if and only if $(X,\mathcal B_X,\mu,T)$ is the system with one point, or the system which exchanges two points (up to isomorphism).
\end{proposition}
\begin{proof}
Suppose that $R_{\alpha}\rtimes_{\tau} T$ is a Kronecker system, so it is isomorphic to its Kronecker factor. According to \Cref{thm:kronecker-factor}, there are two possibilities. Let us suppose first that the Kronecker factor is $R_{\alpha}\times R_1$. Let 
\[\Upsilon\colon\T\times X\to\T\times \Z/2\Z, \ (\theta,x)\mapsto (\phi(\theta,x),\psi(\theta,x))\]
be an isomorphism from $R_{\alpha}\rtimes_{\tau} T$ to its Kronecker factor $R_{\alpha}\times R_1$. Furthermore, let 
\[\Upsilon'\colon\T\times \Z/2\Z\to \T\times  X, \ (\theta,x)\mapsto (\phi'(\theta,x),\psi'(\theta,x))\]
be the inverse of $\Upsilon$, so $\Upsilon\circ\Upsilon'=\id_{\T\times \Z/2\Z}$ and  $\Upsilon'\circ\Upsilon=\id_{\T\times X}$. 

We apply \Cref{prop:isomorphism-nice-first-coordinate} to $\Upsilon$ and $\Upsilon'$. It follows that there is some $c\in\T$ such that $\phi(\theta,x)=\theta+c$ almost everywhere. Similarly, there is some $c'\in\T$ such that $\phi'(\theta,x)=\theta+c'$ almost everywhere. Since  $\Upsilon$ and $\Upsilon'$ are inverses of each other,  we must have $c'=-c$. 

Given $\theta\in\T$ let us write $\psi_{\theta}(x)=\psi(\theta,x)$ and $\psi'_\theta(x)=\psi'(\theta,x)$. Therefore $\psi_\theta$ is a function from $X$ to $\Z/2\Z$, and $\psi'_\theta$ is a function from $\Z/2\Z$ to $X$. 
 
For almost all  $(\theta,x)\in\T\times X$ we have  $(\theta,x)=\Upsilon'(\Upsilon(\theta,x))=\Upsilon'(\theta+c,\psi_\theta(x))=(\theta,\psi'_{\theta+c}(\psi_\theta(x)))$. Similarly, for almost all  $(\theta,x)\in\T\times \Z/2\Z$ we have $(\theta,x)=\Upsilon(\Upsilon'(\theta,x))=\Upsilon(\theta-c,\psi'_\theta(x))=(\theta,\psi_{\theta-c}(\psi'_\theta(x)))$. It follows that for almost all $\theta$, we have $\psi_{\theta}'\circ \psi_{\theta-c} = \id_X$, and $\psi_{\theta-c}\circ\psi'_{\theta}=\id_{\Z/2\Z}$. That is, the function $\psi_{\theta}'\colon\Z/2\Z\to X$ has a left and right measurable inverse, and therefore it is bijective (up to a null set). Since $\Z/2\Z$ has only two elements, it follows that $X$ has only two elements (up to a null set). The fact that $T$ is ergodic and measure-preserving implies then that the action of $T$ must exchange these two elements, and hence both have measure 1/2. 

A similar argument can be applied in the remaining case in which the Kronecker factor of $R_{\alpha}\rtimes_{\tau} T$ is $R_{\alpha}$. 
\end{proof}
\section{Isomorphisms between skew products over the same rotation}\label{sec:isomorphisms}
In this section we study possible isomorphisms between two skew products $R_{\alpha}\rtimes_{\tau} S$ and $R_{\alpha}\rtimes_{\tau} T$ with the same angle in the base. Given an arbitrary isomorphism $\Upsilon$ between $R_{\alpha}\rtimes_{\tau} S$ and $R_{\alpha}\rtimes_{\tau} T$ satisfying the conclusion of \Cref{prop:isomorphism-nice-first-coordinate}, we will prove that $S$ and $T$ are flip isomorphic. In fact, we prove that the slices of $\Upsilon$ provide flip isomorphisms between $S$ and $T$ (in the sense of \Cref{eq:flip-isomorphisms}). 

It will be convenient to assume that $S$ and $T$ act on the same probability space. This allows us to work on the topological space $\aut(X,\mathcal B_X,\mu)$ of all automorphisms of $(X,\mathcal B_X,\mu)$ (definitions below). 
\begin{theorem}\label{thm:isomorphisms-long-statement}
    Let $\alpha\in\R\smallsetminus\Q$. Let $(X,\mathcal B_X,\mu,S)$ and $(X,\mathcal B_X,\mu,T)$ be invertible measure-preserving systems over the same probability space $(X,\mathcal B_X,\mu)$. We assume that $S^n\ne\id_X$ for all $n\ne 0$. Suppose that $R_{\alpha}\rtimes_{\tau} S$ is isomorphic to $R_{\alpha}\rtimes_{\tau} T$, and let $\Upsilon$ be an isomorphism from $R_{\alpha}\rtimes_{\tau} S$ to $R_{\alpha}\rtimes_{\tau} T$, so 
    \[
    \Upsilon\circ (R_{\alpha}\rtimes_{\tau} S)=(R_{\alpha}\rtimes_{\tau} T)\circ\Upsilon
    \]
    Suppose further that $\Upsilon$ has the following special form for some $c\in\T$
    \[
(\theta,x)\mapsto (\theta+c,\psi(\theta,x))
    \]
    Then $S$ is isomorphic to $T^{\ell}$ for some $\ell\in\{-1,1\}$.
    
    More precisely, consider the family of functions parametrized by  $\theta\in\T$
    \[
    \psi_\theta\colon X\to X, \  x\mapsto \psi(\theta,x)\] 
    Then for almost every $\theta$, the map $\psi_{\theta}$ is an automorphism of $(X,\mathcal B_X,\mu)$, and 
\begin{equation}\label{eq:flip-isomorphisms}
    \psi_{\theta}\circ S= T^{\ell}\circ \psi_{\theta}
    \end{equation}
    Moreover, consider the step function
    \[
\zeta\colon\T\to\Z, \ \zeta(\theta)=\tau(\theta+c)-\ell\tau(\theta)
    \]
    If $\zeta$ is a coboundary, then we can describe $\{\psi_{\theta}:\theta\in\T\}$ almost everywhere as 
    \[
\psi_{\theta}=T^{f(\theta)}\circ\upsilon
    \]
     Here $f\colon\T\to \Z$ is a measurable solution to $\zeta=f\circ R_{\alpha}-f$, and $\upsilon$ is an automorphism of $(X,\mathcal B_X,\mu)$ satisfying $\upsilon\circ S=T^{\ell}\circ \upsilon$.  
     
     If $S$ is not rigid, or if $\alpha$ has bounded type, then $\zeta$ is necessarily a coboundary, and thus $\Upsilon$ must have the form described in the previous paragraph.
\end{theorem}
We fix for the rest of this section $S$, $T$, $\Upsilon$, and $\{\psi_{\theta} : \theta\in\T\}$ as in the statement of \Cref{thm:isomorphisms-long-statement}. Let us remark that under the further assumption that $S$ is ergodic, our assumption on $\Upsilon$ is automatically true by \Cref{prop:isomorphism-nice-first-coordinate}. However, for the proof of \Cref{thm:isomorphisms-long-statement} we instead place the assumption on $\Upsilon$.

Before we continue, we need to recall some facts about $\aut(X,\mathcal B_X,\mu)$, the collection of automorphisms of the probability space $(X,\mathcal B_X,\mu)$, where two elements are identified if they coincide almost everywhere. $\aut(X,\mathcal B_X,\mu)$ is a group with the composition operation. We commit the usual notational imprecision of identifying an automorphism of $(X,\mathcal B_X,\mu)$ with its equivalence class in $\aut(X,\mathcal B_X,\mu)$.

$\aut(X,\mathcal B_X,\mu)$ is endowed with the weak topology, and the corresponding Borel sigma-algebra. The weak topology on $\aut(X,\mathcal B_X,\mu)$ is Polish and metrizable. We note that $\lim_{n\to\infty} T_n = T$ if and only if for every $B\in\mathcal B_X$, we have $\lim_{n\to\infty}\mu(T^{-1}(B)\triangle T_n^{-1}(B))=0$. Equivalently, $T_n\to T$ as $n\to\infty$ if and only if we have convergence of the corresponding Koopman operators  $U_{T_n}\to U_{T}$ in the strong operator topology of $L^2_{\C}(X,\mathcal B_X,\mu)$. The reader is referred to \cite[\S 1]{kechris_global_2010} for further details. 

The next result is a standard fact, but we provide a complete argument for self-containment  (see for instance \cite[Lemma 5]{lemanczyk_ergodicity_2001}). 
\begin{proposition}\label{prop:theta-to-psi-theta-is-measurable}
For almost every $\theta\in\T$ the function $\psi_{\theta}$ is an automorphism of $(X,\mathcal B_X,\mu)$. Furthermore, the almost everywhere defined function 
\[
\T\to\aut(X,\mathcal B_X,\mu), \ \theta\mapsto\psi_\theta
\]
is measurable.
\end{proposition}
\begin{proof}
We first prove that for almost every $\theta\in\T$ the function $\psi_{\theta}\colon X\to X$ is bijective up to a null set. Since $\Upsilon$ is an isomorphism, it admits an inverse  $\Upsilon'\colon \T\times X\to \T\times X$. Thus  $\Upsilon\circ\Upsilon'=\Upsilon'\circ\Upsilon$ is the identity map $\id_{\T\times X}$. Let us write $\Upsilon'(\theta,x)=(\phi'(\theta,x),\psi'(\theta,x))$. Looking at the first coordinate in the relation $\Upsilon\circ\Upsilon'=\Upsilon'\circ\Upsilon$, we find $\phi'(\theta,x)+c=\theta$ for almost every $(\theta,x)$, and thus $\phi'$ coincides almost everywhere with the map $(\theta,x)\mapsto \theta-c$. Thus we can write 
$\Upsilon'(\theta,x)=(\theta-c,\psi'(\theta,x))$. We also write $\psi'(\theta,x)=\psi'_{\theta}(x)$. 

For almost every $(\theta,x)$ we have $(\theta,x)=\Upsilon'(\Upsilon(\theta,x))=\Upsilon'(\theta+c,\psi_{\theta}(x))=(\theta,\psi'_{\theta+c}(\psi_{\theta}(x)))$. It follows that for almost every $\theta$, $\psi'_{\theta+c}\circ\psi_{\theta}=\id_{X}$. A similar argument shows that for almost every $\theta$, $\psi_{\theta}\circ \psi'_{\theta+c}=\id_{X}$. Since $\psi_{\theta}$ admits a measurable right and left inverse, it follows that it is bijective almost everywhere.

Next we prove that for almost every $\theta\in \T$ the map $\psi_{\theta}$ preserves $\mu$ in the sense that $\mu(\psi_{\theta}^{-1}(B))=\mu(B)$ for all $B\in\mathcal B_X$.  

Let $P\in\mathcal B_\T$ and $Q\in\mathcal B_X$. Observe that $\Upsilon^{-1}(P
\times Q)=\{(\theta,x) : \theta+c\in P,\psi_{\theta}(x)\in Q\}$. Therefore 
\[
m_{\T}\times\mu(\Upsilon^{-1}(P\times Q))=\int_{\T} \mathbf{1}_{P}(\theta+c)\mu(\psi_{\theta}^{-1}(Q))dm_{\T}(\theta)
\]
Applying the change of variables $\theta\to \theta-c$ which preserves $m_{\T}$ we find
\begin{equation}\label{eq:change-of-variables-and-integral}
m_{\T}\times\mu(\Upsilon^{-1}(P\times Q))=\int_{\T} \mathbf{1}_{P}(\theta)\mu(\psi_{\theta-c}^{-1}(Q))dm_{\T}(\theta)
\end{equation}
On the other hand, 
\[
m_{\T}\times \mu(\Upsilon^{-1}(P\times Q))=m_{\T}\times \mu(P\times Q)=\int_{\T}\mathbf{1}_P(\theta)\mu(Q)dm_{\T}(\theta)
\]
We conclude that  
\[
\int_{\T}\mathbf{1}_P(\theta)\mu(Q)dm_{\T}(\theta)=\int_{\T} \mathbf{1}_{P}(\theta)\mu(\psi_{\theta-c}^{-1}(Q))dm_{\T}(\theta), \ \ \ P\in\mathcal B_\T, Q\in\mathcal B_X.\]
Fix $Q\in\mathcal B_X$. Since the previous equality holds for every $P\in \mathcal B_{\T}$, it follows that $\mu(Q)=\mu(\psi_{\theta-c}^{-1}(Q))$ for almost every $\theta$ in 
$\T$.

Let $\{Q_i : i\in\N\}$ be a countable algebra that generates $\mathcal B_X$. Since a countable intersection of subsets of $\T$ with full measure still has full measure, we can find a measurable  $K\subset \T$ with $m_{\T}(K)=1$ and such that every $\theta\in K$ satisfies $\mu(Q_i)=\mu(\psi_{\theta-c}^{-1}(Q_i))$ for all $i\in\N$. Since $\{Q_i : i\in\N\}$ generates $\mathcal B_X$, it follows that $\mu(B)=\mu(\psi_{\theta-c}^{-1}(B))$, for all $B\in\mathcal B_X$. The set $K-c$ also has full measure, and every $\theta\in K-c$ verifies $\mu(B)=\mu(\psi_{\theta}^{-1}(B))$, for all $B\in\mathcal B_X$. That is, $\psi_{\theta}$ preserves $\mu$.

We now prove the measurability of $\theta\to\psi_{\theta}$. Since $\aut(X,\mathcal B_X,\mu)$ is a topological group, the operation of taking inverses is continuous, and in particular measurable. Thus it is sufficient to prove that $\theta\to\psi_\theta^{-1}$ is measurable. A sub-basis for the weak topology on $\aut(X,\mathcal B_X,\mu)$ is given by sets of the form
\[
U_{a,b,A,B}=\{R\in \aut(X,\mathcal B_X,\mu) : a<\mu(R(A)\cap B)<b\}, 
\]
where $a,b\geq 0$ and $A,B\in\mathcal B_X$ (this follows from the definition of $W_{S,A_1,\dots,A_n,\epsilon}$ given on page 2 of \cite{kechris_global_2010}). Fix $a,b>0$ and $A,B$ in $\mathcal B_X$. To prove that $\{\theta\in\T : a<\mu(\psi_\theta^{-1}(A)\cap B)<b\}$ is measurable, it is sufficient to show the measurability of the map
\[f\colon \T\to [0,1], \theta\mapsto f(\theta)=\mu(\psi_\theta^{-1}(A)\cap B).\]
Observe that we can write
\[
f(\theta)=\int_{X} \mathbf{1}_{\psi^{-1}_\theta(A)}(x)\cdot \mathbf{1}_B(x)d\mu(x).
\]
Furthermore we have  
\[\mathbf{1}_{\psi^{-1}_\theta(A)}(x)=1\iff \psi(\theta,x)\in A\iff \mathbf{1}_{\psi^{-1}(A)}(\theta,x)=1\]
Thus we can re-write \[
f(\theta)=\int_X \mathbf{1}_{\psi^{-1}(A)}(\theta,x)\cdot \mathbf{1}_B(x)d\mu(x)
\]
    Since $\psi\colon \T\times X\to X$ is measurable and $A\in\mathcal B_X$, we have that $\psi^{-1}(A)$ is a measurable set. It follows that $(\theta,x)\to \mathbf{1}_{\psi^{-1}(A)}(\theta,x)\cdot \mathbf{1}_B(x)$ is measurable as a function of both $\theta$ and $x$. Then Tonelli's theorem shows that $\theta\to f(\theta)$ is measurable. 
\end{proof}
We shall now observe a relation which follows directly from the fact that $\Upsilon$ intertwines $R_{\alpha}\rtimes_{\tau} S$ and $R_{\alpha}\rtimes_{\tau} T$.
\begin{proposition}\label{prop:T-and-S-powers}
For almost every $\theta\in \T$ we have 
\begin{equation}\label{eq:T-and-S-powers}
\psi_{\theta+n\alpha}\circ S^{\tau(n,\theta)} = T^{\tau(n,\theta+c)}\circ \psi_{\theta}, \ \ \ n\in\Z
\end{equation}
\end{proposition}
\begin{proof}
Let $n\in\Z$. Since $\Upsilon \circ (R_{\alpha}\rtimes_{\tau} S)=(R_{\alpha}\rtimes_{\tau} T)\circ \Upsilon$, we also have $\Upsilon \circ (R_{\alpha}\rtimes_{\tau} S)^n=(R_{\alpha}\rtimes_{\tau} T)^n\circ \Upsilon$. Thus for almost every $(\theta,x)\in\T\times X$ we have 
\[(\theta+n\alpha+c,\psi_{\theta+n\alpha}(S^{\tau(n,\theta)}(x))) =  (\theta+n\alpha+c,T^{\tau(n,\theta+c)}(\psi_{\theta}(x)))\] 
In particular, for almost all $(\theta,x)$ we have $\psi_{\theta+n\alpha}(S^{\tau(n,\theta)}(x))=T^{\tau(n,\theta+c)}(\psi_{\theta}(x))$. By Fubini's or Tonelli's theorem, for almost every $\theta\in\T$, the same equation holds for a full measure set of $x\in X$. That is, $\psi_{\theta+n\alpha}\circ S^{\tau(n,\theta)} = T^{\tau(n,\theta+c)}\circ \psi_{\theta}$. 
\end{proof}
In the next result we show that $S$ and $T$ are flip isomorphic. The value $\ell$ obtained will be fixed for the rest of this section.
\begin{proposition}\label{prop:ell-existence}
There is $\ell\in\{-1,1\}$ such that for almost every $\theta$ we have 
	\begin{equation}\label{eq:ell}
		\psi_{\theta}\circ S = T^{\ell}\circ \psi_{\theta}
	\end{equation}
\end{proposition}
\begin{proof}
In order to prove our claim, we first show that there are $\ell,\kappa\in\{-3,-1,1,3\}$ such that for almost every $\theta$ we have 
	\begin{equation}\label{eq:ell-kappa}
		\psi_{\theta}\circ S = T^{\ell}\circ \psi_{\theta} \ \ \text{and} \ \ 	\psi_{\theta}\circ S^{\kappa} = T\circ \psi_{\theta} 
	\end{equation}
We start with $\ell$. Let $(q_n)_{n\in\N}$ be a sequence as in \Cref{prop:infinitely-many-1s}. Replacing $(q_n)_{n\in\N}$ by a subsequence if needed, we can assume that for almost every $\theta$ we have 
\[
\lim_{n\to\infty} \psi_{\theta+q_{n}\alpha} = \psi_{\theta}
\]
in $\aut(X,\mathcal B_X,\mu)$. This is possible thanks to \Cref{prop:subsequences} and the measurability of $\theta\mapsto\psi_{\theta}$ (\Cref{prop:theta-to-psi-theta-is-measurable}). 

Having chosen $(q_n)_{n\in\N}$, we apply \Cref{prop:ell-possible-values}. We obtain some $\ell\in \{-3,-1,1,3\}$ such that for almost all $\theta\in\T$ we have \[(\tau(q_{n},\theta),\tau(q_{n},\theta+c))=(1,\ell)\text{ for infinitely many values of $n$}
\]
For $\theta$ and $q_n$ satisfying $(\tau(q_{n},\theta),\tau(q_{n},\theta+c))=(1,\ell)$, observe that \Cref{eq:T-and-S-powers} simplifies to
\[
\psi_{\theta+q_{n}\alpha} \circ S = T^{\ell}\circ \psi_{\theta}
\]
For almost all $\theta$, we can argue as follows. Pick a subsequence $(q_{n_k})_{k\in\N}$, depending on $\theta$, such that $(\tau(q_{n_k},\theta),\tau(q_{n_k},\theta+c))=(1,\ell)$ for all $k$. Since a subsequence of a convergent subsequence converges to the same limit, we have that $\psi_{\theta+q_{n_{k}}\alpha}$ converges to $\psi_{\theta}$ as $k\to\infty$. Since composition is a continuous operation on $\aut(X,\mathcal B_X,\mu)$, it follows that $\psi_{\theta+q_{n_{k}}\alpha}\circ S $ converges to $\psi_{\theta}\circ S$. On the other hand, for each $k$ we have $\psi_{\theta+q_{n_k}\alpha} \circ S = T^{\ell}\circ \psi_{\theta}$. Therefore $\psi_{\theta}\circ S = T^{\ell}\circ \psi_{\theta}$ as claimed. 

The proof of the existence of  $\kappa\in\{-3,-1,1,3\}$ satisfying \Cref{eq:ell-kappa} follows by a similar reasoning. The sequence $(q_{n})_{n\in\N}$ is the same as before. The only difference is that one instead applies \Cref{prop:ell-possible-values} with $-c$ instead of $c$. This result provides a $\kappa\in\{-3,-1,1,3\}$ with the following property: for almost every $\theta$, we can find a subsequence $(q_{n_{k}})_{k\in\N}$ such that $\tau(q_{n_{k}},\theta+c)=1$ and $\tau(q_{n_{k}},\theta)=\kappa$ for all $k$. For such $\theta$, we can consider the relation $
\psi_{\theta+n\alpha}\circ S^{\tau(n,\theta)} = T^{\tau(n,\theta+c)}\circ \psi_{\theta}$ along the subsequence $(q_{n_{k}})_{k\in\N}$, and find $\psi_{\theta}\circ S^{\kappa}=T\circ\psi_{\theta}$ as desired. 

Thus we have proved \Cref{eq:ell-kappa}. We will derive from these two relations that  $\ell\in\{-1,1\}$.  In fact, we show
\[\ell\kappa=1.\]
The only solutions with $\ell,\kappa\in\Z$ are $\ell=\kappa=1$ and $\ell=\kappa=-1$.

Fix $\theta\in\T$ satisfying both \[T=\psi_{\theta}\circ S^{\kappa}\circ \psi_{\theta}^{-1}\ \ \text{ and }\ \ T^{\ell}=\psi_{\theta}\circ S\circ \psi_{\theta}^{-1}\]
It follows from the first of these equations that   $T^{\ell}=\psi_{\theta}\circ S^{\kappa\ell}\circ \psi_{\theta}^{-1}$. Comparing this with the second equation, we obtain $\psi_{\theta}\circ S\circ \psi_{\theta}^{-1}=\psi_{\theta}\circ S^{\kappa\ell}\circ \psi_{\theta}^{-1}$. But now we can cancel out the $\psi_{\theta}$ terms, and find $S=S^{\kappa\ell}$. Thus  $S^{\kappa\ell-1}=\id_{X}$. Since we assume that $S^n\ne\id_X$ for all $n\ne 0$, the only possibility is that $\kappa\ell=1$ as claimed.

Let us note that in this last argument, we used in $\aut(X,\mathcal B_X,\mu)$ an elementary fact about periods and conjugacies which is valid in any abstract group $(G,\cdot)$: if $a,b,c\in G$ satisfy conjugacy relations $aba^{-1}=c^\ell$ and $ab^{\kappa}a^{-1}=c$, then necessarily $b^{\kappa\ell-1}$ is the identity element of $G$.  
\end{proof}

At this point we have already proved the first part of \Cref{thm:isomorphisms-long-statement}. It remains to prove the ``moreover'' clause, which describes the family $\{\psi_{\theta} :\theta\in\T\}$ of flip isomorphisms between $S$ and $T$. In order to do this, we will use   \Cref{prop:ell-existence} to derive another interesting equation in $\aut(X,\mathcal B_X,\mu)$. We will now consider the step function $\zeta\colon\T\to\Z$ given by 
\begin{equation}\label{eq:zeta}
	\zeta(\theta)=\tau(\theta+c)-\ell\tau(\theta)
\end{equation}
We denote by the same symbol the associated  cocycle $\zeta\colon\Z\times\T\to\Z$ (see \Cref{sec:cocycles}). This cocycle determines the relation between $\psi_{\theta}$ and $\psi_{\theta+\alpha}$ in the sense of the following proposition. 
\begin{proposition}\label{prop:fancier-equation}
For almost every $\theta$ we have 
\begin{equation}\label{eq:fancier-equation}
	\psi_{\theta+n\alpha}\circ\psi_{\theta}^{-1}=T^{\zeta(n,\theta)}, \ \ n\in\Z
\end{equation}
\end{proposition}
\begin{proof}
Consider the relation $\psi_{\theta} \circ S \circ \psi_{\theta}^{-1}=T^\ell$ from \Cref{prop:ell-existence}. It follows that 
\begin{equation}\label{eq:post-isomorphism-equation}
\psi_{\theta} \circ S^{\tau(n,\theta)}  \circ\psi_{\theta}^{-1}=T^{\ell\tau(n,\theta)}
\end{equation}
On the other hand, recall from \Cref{eq:T-and-S-powers} that $
\psi_{\theta+n\alpha} \circ S^{\tau(n,\theta)} \circ \psi_{\theta}^{-1} = T^{\tau(n,\theta+c)}$. From these two relations we find 
\begin{equation}\label{eq:fancy-equation}
    \psi_{\theta+n\alpha} \circ \psi_{\theta}^{-1} = T^{\tau(n,\theta+c)-\ell\tau(n,\theta)}
\end{equation}
But since $\zeta(\theta)=\tau(\theta+c)-\ell\tau(\theta)$, it is clear that $
\zeta(n,\theta)=\tau(n,\theta+c)-\ell\tau(n,\theta)$. 
\end{proof}
If $\zeta$ is a coboundary, then $\Upsilon$ comes from a single flip isomorphism of $S$ and $T$, in the sense of the next result. 
\begin{proposition}\label{prop:rigid-isomorphisms-second}
Suppose that $\zeta$ is a coboundary, and let $f\colon\T\to \Z$ be a measurable function satisfying 
\[
\zeta(\theta)=f(\theta+\alpha)-f(\theta)
\]
Then there exists an automorphism $\upsilon$ of $(X,\mathcal B_X,\mu)$ such that $\upsilon\circ S=T^{\ell}\circ \upsilon$, and for almost all $\theta$ we have
\[
\psi_{\theta}=T^{f(\theta)}\circ\upsilon
\]
\end{proposition}
\begin{proof}
Let $f$ be as in the statement. Then by  \Cref{eq:fancier-equation} with $n=1$ we have 
\[
\psi_{\theta+\alpha}\circ\psi_{\theta}^{-1}=T^{f(\theta+\alpha)-f(\theta)}
\]
That is
\[
T^{-f(\theta+\alpha)}\circ \psi_{\theta+\alpha}=T^{-f(\theta)}\circ\psi_{\theta}
\]
Consider the measurable map $F\colon\T\to\aut(X,\mathcal B_X,\mu)$ given by  
\[
F(\theta)=T^{-f(\theta)}\circ\psi_{\theta}
\]
It follows that $F(\theta)=F(\theta+\alpha)$ for almost all $\theta$. Since $R_{\alpha}$ is ergodic, $F$ must be constant almost everywhere. Let $\upsilon$ be an automorphism of $(X,\mathcal B_X,\mu)$ such that $F(\theta)=\upsilon$ for almost every $\theta$. Thus $\psi_{\theta}=T^{f(\theta)}\circ\upsilon$ for almost all $\theta$. Finally, let us verify that $\upsilon$ intertwines $S$ and $T^{\ell}$:
\[
\upsilon \circ S \circ \upsilon^{-1}=(T^{-f(\theta)} \circ\psi_{\theta} ) \circ S  \circ (\psi_{\theta}^{-1} \circ T^{f(\theta)})=T^{-f(\theta)+\ell+f(\theta)}=T^{\ell}
\]
This finishes the proof.
\end{proof}
In what follows we shall be interested in understanding whether $\zeta$ is a coboundary. This is equivalent to the statement that $\overline{\mathcal E}(\zeta)=\{0\}$ (see \Cref{sec:cocycles}). We now verify that $\zeta$ does not have nonzero finite essential values, in the sense that $\mathcal E(\zeta)=\{0\}$. The proof is based on our assumption that $S^n\ne\id_X$ for all $n\ne 0$.
\begin{proposition}\label{prop:zeta-has-no-nonzero-essential-values}
    $\mathcal E(\zeta)=\{0\}$.
\end{proposition}
\begin{proof}
    Suppose that $\zeta$ admits an essential value $e\in\Z\smallsetminus\{0\}$. We will prove that this implies $T^e=\id_X$, which is a contradiction (recall that $S^n\ne\id_X$ for all $n\ne0$ by assumption in \Cref{thm:isomorphisms-long-statement}, and that we already proved that $S$ and $T$ are flip isomorphic).

    Let $K\subset \T$ be an arbitrary measurable set with positive measure. We claim that for every $\epsilon>0$ we can find $\theta_\epsilon,\theta_\epsilon'\in K$ such that $d_{\T}(\theta_\epsilon,\theta_\epsilon')<\epsilon$ and with 
    \[
\psi_{\theta_\epsilon'}\circ\psi_{\theta_\epsilon}^{-1}=T^e
    \]
    Indeed, since $K$ has positive measure, we can find an interval $I$ in $\T$ with diameter smaller than $\epsilon$, and such that $K\cap I$ has positive measure. Let $K'=K\cap I$. Next, we apply the definition of essential value for $e$ and $K'$. We find $n\in\Z$ such that 
    \[
m_{\T}(K'\cap R_{\alpha}^{-n}(K')\cap\{\theta\in\T : \zeta(n,\theta)=e\})>0
    \]
    Pick $\theta$ in this intersection  satisfying $\psi_{\theta+n\alpha}\circ\psi_{\theta}^{-1}=T^{\zeta(n,\theta)}$ (\Cref{prop:fancier-equation}). Let $\theta_\epsilon$ be this element, so $\psi_{\theta_\epsilon+n\alpha}\circ\psi_{\theta_\epsilon}^{-1}=T^e$. Since $R_{\alpha}$ is an isometry, $K'$ has diameter smaller than $\epsilon$, and $K'\cap R_{\alpha}^{-n}(K')\ne\emptyset$, we necessarily have $||n\alpha||<\epsilon$. Thus $\theta_\epsilon'=\theta_\epsilon+n\alpha$ is $\epsilon$-close to $\theta_\epsilon$, and $\psi_{\theta_\epsilon'}\circ\psi_{\theta_\epsilon}^{-1}=T^e$ as desired. Furthermore, observe that since $\theta_{\epsilon}\in K'\cap R_{\alpha}^{-n}(K')$, the element $\theta_\epsilon+n\alpha$ belongs to $K'$. Thus both elements $\theta_\epsilon$ and $\theta_\epsilon'$ belong to $K'$. 

    We have proved our first claim. Next, since $\theta\to\psi_{\theta}$ is measurable (\Cref{prop:theta-to-psi-theta-is-measurable}) and $\aut(X,\mathcal B_X,\mu)$ is Polish, we can apply Lusin's Theorem (Theorem 17.12 in \cite{kechris_classical_1995}). It follows that there exists a compact set $K\subset \T$ with positive measure and such that $\theta\to\psi_{\theta}$ is continuous over $K$. We apply the previous claim to the compact set $K$. For each $\epsilon=1/n$, we find two elements $\theta_n$ and $\theta_n'$ with $d_{\T}(\theta_n,\theta_n')<1/n$ and with 
    \begin{equation}\label{eq:fancier-equation-applied}
    \psi_{\theta_n'}\circ\psi_{\theta_n}^{-1}=T^e
    \end{equation}
    Since $K$ is compact, the sequence $(\theta_{n})_{n\in\N}$ has a subsequence $(\theta_{n_k})_{k\in\N}$ which converges to some $\theta_*\in K$. Let us observe that $(\theta_{n_k}')_{k\in\N}$ also converges to $\theta_*$. Indeed, since $\theta_{n_k}$ and $\theta_{n_k}'$ are $1/n_k$-close, by the triangle inequality we have 
    \[
d_{\T}(\theta_{n_k}',\theta_*)\leq d_{\T}(\theta_{n_k}',\theta_{n_k})+d_{\T}(\theta_{n_k},\theta_*)\leq 1/{n_k}+d_{\T}(\theta_{n_k},\theta_*)\to 0 \text{ as $k\to\infty$}
    \]
    Since $\theta\to\psi_{\theta}$ is continuous over $K$, we have
    \[
\lim_{k\to\infty}\psi_{\theta_{n_k}}=\lim_{k\to\infty}\psi_{\theta_{n_k}'}=\psi_{\theta_*}
    \]
    As inversion and composition are continuous operations on $\aut(X,\mathcal B_X,\mu)$, it follows that 
    \[
    \lim_{k\to\infty} \psi_{\theta_{n_k}'}\circ\psi_{\theta_{n_k}}^{-1}=\psi_{\theta_*}\circ\psi_{\theta_*}^{-1}=\id_X
    \]
    On the other hand, it follows from \Cref{eq:fancier-equation-applied} that for every $k$ we have $\psi_{\theta_{n_k}'}\circ\psi_{\theta_{n_k}}^{-1}=T^e$. Thus  $T^e=\id_X$, which is the desired contradiction. 
\end{proof}
If $\alpha$ has bounded type, it follows from \Cref{prop:zeta-has-no-nonzero-essential-values} and \Cref{prop:bounded-type-coboundary-essential-values} that $\zeta$ is necessarily a coboundary, and in fact the possible values of $c$ are characterized by \Cref{prop:sufficient-conditions-zeta-coboundary}.

Without assumptions on $\alpha$, we are still able to show that $\zeta$ must be a coboundary if we suppose that $S$ is not rigid. In this case we can show that $\infty$ is not an essential value for $\zeta$. Thus $\overline{\mathcal E}(\zeta)=\{0\}$, and this implies that $\zeta$ is a coboundary (see \Cref{sec:cocycles}).

An automorphism $R$ of $(X,\mathcal B_X,\mu)$ is rigid if there is an increasing sequence of natural numbers $(n_k)_{k\in\N}$ such that $(R^{n_k})_{k\in\N}$ converges to the trivial transformation $\id_X$ in $\aut(X,\mathcal B_X,\mu)$. One can equivalently take $n_k\in\Z$, and require $k\mapsto |n_k|$ to be increasing. This is simply because the inversion map $R\to R^{-1}$ is continuous on $\aut(X,\mathcal B_X,\mu)$, and it fixes $\id_X$. 
\begin{proposition}\label{prop:rigid-isomorphisms}
If $S$ (equivalently $T$) is not rigid, then $\infty$ is not an essential value for $\zeta$. Hence $\overline{\mathcal E}(\zeta)=\{0\}$, and $\zeta$ is a coboundary. 
\end{proposition}
\begin{proof}
We argue that if $\zeta$ has $\infty$ as an essential value, then $T$ is rigid. The argument is very similar to that of  \Cref{prop:zeta-has-no-nonzero-essential-values}, so we will omit some parts. The only difference is that instead of a fixed power of $T$, we find a sequence of powers of $T$. 

Let $K\subset \T$ be an arbitrary measurable set with positive measure. We claim that for all $\epsilon>0$ and all $j\in\N$, we can find $r\in\Z$ with $|r|>j$ and two elements $\theta_{\epsilon}$ and $\theta_{\epsilon}'$ in $K$ with $d_{\T}(\theta_\epsilon,\theta_\epsilon')<\epsilon$ such that
\[
\psi_{\theta_{\epsilon}'}\circ\psi_{\theta_{\epsilon}}^{-1}=T^{r}
\]
Indeed, since $K$ has positive measure, we can find an interval $I$ in $\T$ with diameter smaller than $\epsilon$, and such that $K\cap I$ has positive measure. Let $K'=K\cap I$. Next, we apply the definition of essential value for $\infty$ and $K'$. We find $n\in\Z$ such that 
    \[
m_{\T}(K'\cap R_{\alpha}^{-n}(K')\cap\{\theta\in\T : |\zeta(n,\theta)|>j\})>0
    \]
Since $\{\theta\in\T : |\zeta(n,\theta)|>j\}=\cup_{i\in\Z,|i|>j} \{\theta\in\T : \zeta(n,\theta)=i\}$, it follows that there exists $r\in\Z$ with $|r|>j$ and such that 
    \[
m_{\T}(K'\cap R_{\alpha}^{-n}(K')\cap\{\theta\in\T : \zeta(n,\theta)=r\})>0
    \]
Then $\theta_{\epsilon}$ can be chosen as an element in this intersection, and then we set $\theta_{\epsilon}'=\theta_{\epsilon}+n\alpha$. Then $\theta_{\epsilon}$ and $\theta_{\epsilon}'$ verify the desired properties (see the argument in  \Cref{prop:zeta-has-no-nonzero-essential-values}).

We have proved our first claim. Since $\theta\to\psi_{\theta}$ is measurable (\Cref{prop:theta-to-psi-theta-is-measurable}) and $\aut(X,\mathcal B_X,\mu)$ is Polish, we can apply Lusin's Theorem (Theorem 17.12 in \cite{kechris_classical_1995}). It follows that there exists a compact set $K\subset \T$ with positive measure and such that $\theta\to\psi_{\theta}$ is continuous over $K$. We apply the previous claim to the compact set $K$. We pick $(\theta_n)_{n\in\N}$ and $(\theta_n')_{n\in\N}$ as sequences of elements in $K$ with $d_{\T}(\theta_n,\theta_n')<1/n$, and satisfying 
\begin{equation}\label{eq:fancier-equation-applied-rigidity}
    \psi_{\theta_n'}\circ\psi_{\theta_n}^{-1}=T^{r_n}
\end{equation}
where $(r_n)_{n\in\N}$ is a sequence of integers such that $n\mapsto |r_n|$ is increasing. This can be ensured by picking the elements  in order, so that at step $n$ one takes $r_n$ with $|r_n|$ bigger than $j=|r_{n-1}|$. 

The rest of the argument is as in \Cref{prop:zeta-has-no-nonzero-essential-values}: since $K$ is compact, we can find a subsequence $(n_k)_{k\in\N}$ and a limit point $\theta_*\in K$ so that $(\theta_{n_k})_{k\in\N}$ and $(\theta'_{n_k})_{k\in\N}$ converge to $\theta_*$. Then by continuity of $\theta\to\psi_{\theta}$ on $K$, we have 
\[
\lim_{k\to\infty}\psi_{\theta_{n_k}'}\circ\psi_{\theta_{n_k}}^{-1}=\psi_{\theta_*}\circ\psi_{\theta_*}^{-1}=\id_{X}
\]
By \Cref{eq:fancier-equation-applied-rigidity}, for all $k$ we have $\psi_{\theta_{n_k}'}\circ \psi_{\theta_{n_k}}^{-1}=T^{r_{n_k}}$. Thus we have proved that
\[
\lim_{k\to\infty} T^{r_{n_k}}=\id_X.
\]
Thus $T$ is rigid as claimed. 
\end{proof}
This finishes our proof of  \Cref{thm:isomorphisms-long-statement}. 
%
\section{Isomorphisms between skew products over different  rotations}\label{sec:different-angles}
In this section we prove that an isomorphism between $R_{\alpha}\rtimes_{\tau} S$ and $R_{\beta}\rtimes_{\tau} T$ is only possible for $\alpha=\pm\beta\mod 1$, provided $S$ is ergodic and $S^2\ne\id_X$. This is the last nontrivial ingredient needed to prove \Cref{thm:isomorphisms}, which we finally prove in the next section. 

As in the previous section, we will assume that $S$ and $T$ act on the same probability space $(X,\mathcal B_X,\mu)$. This assumption allows us to work in the space $\aut(X,\mathcal B_X,\mu)$ of automorphisms of $(X,\mathcal B_X,\mu)$. 
\begin{theorem}\label{thm:different-angles}
    Let $\alpha,\beta\in\R\smallsetminus\Q$, and let $(X,\mathcal B_X,\mu,S)$ and $(X,\mathcal B_X,\mu,T)$ be invertible measure-preserving systems over the same probability space $(X,\mathcal B_X,\mu)$. We assume that $S$ and $T$ are ergodic, and that $S^2\ne\id_X$. Then a necessary condition for $R_{\alpha}\rtimes_{\tau} S$ to be isomorphic to $R_{\beta}\rtimes_{\tau} T$ is that $\alpha=\pm\beta\mod 1$.
\end{theorem}
\begin{proof}
    Let $S$ and $T$ be as in the statement, and assume  that $R_{\alpha}\rtimes_{\tau} S$ is isomorphic to $R_{\beta}\rtimes_{\tau} T$, $\alpha,\beta\in\R\smallsetminus\Q$. Our goal is showing $\alpha=\pm\beta\mod 1$. 
    
    Since $R_{\alpha}\rtimes_{\tau} S$ is isomorphic to $R_{\beta}\rtimes_{\tau} T$, their  Koopman operators have the same sets of eigenvalues. We computed them in  \Cref{thm:eigenvalues}. We find that there are two cases:
\begin{enumerate}
    \item $\{e^{2\pi i k \alpha} : k\in\Z\}=\{e^{2\pi i k \beta} : k\in\Z\}$, or 
    \item $
\{e^{2\pi i k \alpha} : k\in\Z\}\cup \{-e^{2\pi i k \alpha} : k\in\Z\}=\{e^{2\pi i k \beta} : k\in\Z\}\cup \{-e^{2\pi i k \beta} : k\in\Z\}$
\end{enumerate}
Since $-1=e^{\pi i} =e^{2 \pi i 1/2}$, we can rewrite these cases as 
\begin{enumerate}
    \item $\{e^{2\pi i x} : x\in\alpha\Z\}=\{e^{2\pi i x} : x\in\beta\Z\}$, or 
    \item $\{e^{2\pi i x } : x\in\alpha\Z+1/2\Z\}=\{e^{2\pi i x} : x\in\beta\Z+1/2\Z\}$
\end{enumerate}
For $x,y\in\R$, we have $e^{2\pi i x}=e^{2\pi i y}$ if and only if $x,y$ differ by an integer. Thus we can rewrite our cases as 
\begin{enumerate}
    \item $\{x \mod 1 : x\in\alpha\Z\}=\{x\mod 1: x\in\beta\Z\}$, or 
    \item $\{x \mod 1 : x\in\alpha\Z+1/2\Z\}=\{x \mod 1 : x\in\beta\Z+1/2\Z\}$
\end{enumerate}
If we are in case (1), then the standard argument shows  that $\{x \mod 1 : x\in\alpha\Z\}=\{x\mod 1: x\in\beta\Z\}$ implies $\alpha=\pm \beta\mod 1$. That is, if $\alpha=n\beta\mod 1$ and $\beta=k\alpha\mod 1$, then $\alpha=nk\alpha\mod 1$, $n,k\in\Z$. But then in $\R$ we have the equality $\alpha=nk\alpha+r$, for some $r\in\Z$. Unless $nk=1$, we have $\alpha=r/(1-nk)$, contradicting that $\alpha$ is irrational.

The same argument can be applied in case (2). However, the conclusion one obtains is that either $\alpha=\pm\beta\mod 1$, or $\alpha=\pm\beta+1/2\mod 1$. The nontrivial part of this proof is discarding the possibility $\beta=\pm\alpha+1/2\mod 1$. Since $R_{\beta+1/2}\rtimes_{\tau} T$ is isomorphic to $R_{-\beta+1/2}\rtimes_{\tau} T$ (\Cref{prop:easy-automorphisms}), it is sufficient to discard  $\beta=\alpha+1/2$. We remark that this cannot be done by looking at the sets of eigenvalues (or Kronecker factors) because $\alpha\Z+1/2\Z=(\alpha+1/2)\Z+1/2\Z$. 

Let us assume for the rest of this argument that we are in case (2) above, and that $\beta$ is equal to $\alpha+1/2$ to reach a contradiction. Thus the standing assumption is that $R_{\alpha}\rtimes_{\tau} S$ is isomorphic to $R_{\alpha+1/2}\rtimes_{\tau} T$. We will repeat the arguments used in the proof of \Cref{prop:isomorphism-nice-first-coordinate} and in  \Cref{sec:isomorphisms}, and the contradiction obtained will be $S^2=\id_X$. Since we will repeat a number of arguments used before, we will be slightly less detailed.

We start by invoking  \Cref{thm:kronecker-factor}, which states that the Kronecker factors of $R_{\alpha}\rtimes_{\tau} S$ and $R_{\alpha+1/2}\rtimes_{\tau} T$ are $R_{\alpha}\times R_1$ and $R_{\alpha+1/2}\times R_1$. Recall that $R_1$ denotes addition of $1$ mod 2 on $\Z/2\Z$. See also  \Cref{eq:def-R-1-times-R-alpha}. Next, consider the map
\[
\rho\colon\T\times\Z/2\Z\to\T\times\Z/2\Z, \ (\theta,x)\mapsto (\theta+x/2,x)
\]
Recall that $\T\times\Z/2\Z$ is endowed with the group operation of componentwise addition. It is clear that $\rho$ is a group homomorphism. Since $\rho^2=\id_{\T\times\Z/2\Z}$, it follows that $\rho$ is a group isomorphism. 

On the other hand, it is clear that $\rho$ preserves the Haar measure $m_{\T}\times m_{\Z/2\Z}$. Since $\rho(\alpha,1)=(\alpha+1/2,1)$, it follows that $\rho$ defines an isomorphism between $R_{\alpha}\times R_1$ and $R_{\alpha+1/2}\times R_1$ as measure-preserving systems\footnote{This observation shows that the hypothesis $S^2\ne\id_X$ in the statement is necessary. Note that since $R_1$ is equal to its inverse, the skew product $R_{\alpha}\rtimes_{\tau} R_1$ equals the direct product $R_{\alpha}\times R_1$. Since we just noted that $R_{\alpha}\times R_1$ is isomorphic to $R_{\alpha+1/2}\times R_1$, it follows that  $R_{\alpha}\rtimes_{\tau} R_1$ is isomorphic to $R_{\alpha+1/2}\rtimes_{\tau} R_1$.}.

It is now convenient to fix factor maps from $R_{\alpha}\rtimes_{\tau} S$ and $R_{\alpha+1/2}\rtimes_{\tau} T$ to their Kronecker factors. Let $\iota_S\colon X\to\Z/2\Z$ and $\iota_T\colon X\to\Z/2\Z$ be factor maps from $S$ and $T$ to $R_1$, respectively (\Cref{prop:two-point-system}). We now define $\pi_S$ and $\pi_T$ by 
\begin{align*}
&\pi_S\colon \T\times X\to\T\times\Z/2\Z, \ (\theta,x)\mapsto (\theta,\iota_S(x))\\
&\pi_T\colon \T\times X\to\T\times\Z/2\Z, \ (\theta,x)\mapsto (\theta,\iota_T(x))
\end{align*}
Thus $\pi_S$ and $\pi_T$ define measure-theoretic factors from $R_{\alpha}\rtimes_{\tau} S$ and $R_{\alpha+1/2}\rtimes_{\tau} T$ to their Kronecker factors $R_{\alpha}\times R_1$ and $R_{\alpha+1/2}\times R_1$, respectively.

Since we are assuming that $R_{\alpha}\rtimes_{\tau} S$ is isomorphic to $R_{\alpha+1/2}\rtimes_{\tau} T$, let us fix such an isomorphism $\Upsilon$. Thus $\Upsilon$ is an automorphism of $(\T\times X,\mathcal B_{\T}\otimes \mathcal B_X,m_{\T}\times\mu)$, and $\Upsilon\circ (R_{\alpha}\rtimes_{\tau} S)=(R_{\alpha+1/2}\rtimes_{\tau} T)\circ\Upsilon$. We write $\Upsilon$ as 
\[
\Upsilon(\theta,x)= (\phi(\theta,x),\psi(\theta,x))
\]

\begin{figure}[h]
    \centering
    \begin{tikzcd}
\mathbb{T}\times X \arrow["R_{\alpha}\rtimes_{\tau} S", loop, distance=2em, in=55, out=125] \arrow[r, "\Upsilon"] \arrow[d, "\pi_{S}"'] & \mathbb{T}\times X \arrow["R_{\alpha+1/2}\rtimes_{\tau} T", loop, distance=2em, in=55, out=125] \arrow[d, "\pi_T"] \\
\mathbb{T}\times \mathbb{Z}/2\mathbb{Z} \arrow["R_{\alpha}\times R_1"', loop, distance=2em, in=305, out=235] \arrow[r, "f"', dashed]    & \mathbb{T}\times \mathbb{Z}/2\mathbb{Z} \arrow["R_{\alpha+1/2}\times R_1"', loop, distance=2em, in=305, out=235]  
\end{tikzcd}
    \caption{The systems $R_{\alpha}\rtimes_{\tau} S$ and $R_{\alpha+1/2}\rtimes_{\tau} T$, factor maps to their Kronecker factors, and the factor map $f$ between their Kronecker factors.}\label{fig:bad-kronecker-factors}
\end{figure}

We will now obtain a more explicit description of $\phi$. Observe that $\pi_T\circ \Upsilon$ is a factor map from $R_{\alpha}\rtimes_{\tau} S$ to the Kronecker system $R_{\alpha+1/2}\times R_1$. On the other hand, the Kronecker factor of $R_{\alpha}\rtimes_{\tau} S$ is $R_{\alpha}\times R_1$ via the factor map $\pi_S$. By item 2 in \Cref{prop:characterization-kronecker-factor}, it follows that there exists a factor map 
\[
f\colon \T\times \Z/2\Z\to\T\times\Z/2\Z
\]
from $R_{\alpha}\times R_1$ to $ R_{\alpha+1/2}\times R_1$ such that $\pi_T\circ\Upsilon=f\circ \pi_S$ (see \Cref{fig:bad-kronecker-factors}). By \Cref{prop:algebraic-factors}, there exists $(c,c')\in\T\times\Z/2\Z$ such that for almost all $(\theta,x)$ we have 
\[
f(\theta,x)=\rho(\theta,x)+(c,c')
\]
Evaluating this in the relation $\pi_T\circ\Upsilon=f\circ\pi_S$ we find 
\[(\phi(\theta,x),\iota_T(\psi(\theta,x)))=f(\theta,\iota_S(x))=(\theta+\iota_S(x)/2+c,\iota_S(x)+c')\]
Therefore 
\begin{align}
    \phi(\theta,x)=\theta+\iota_S(x)/2+c\\
    \iota_T(\psi(\theta,x))=\iota_S(x)+c'\label{eq:iota-S}
\end{align}
For $\theta\in\T$ let us denote by $\psi_{\theta}$ the map $X\to X$, $x\mapsto \psi(\theta,x)$. With minor\footnote{The only part that needs a nontrivial modification is the argument that $\psi_{\theta}$ preserves $\mu$ for almost all $\theta$. The method used in \Cref{prop:theta-to-psi-theta-is-measurable} shows the following for almost all $\theta$. For $Q\in\mathcal B_X$ fully contained in $\{x\in X : \iota_T(x)=0\}$ (respectively, $\{x\in X : \iota_T(x)=1\}$), we have $\mu(Q)=\mu(\psi_{\theta}^{-1}(Q))$. From this one obtains that $\psi_{\theta}$ preserves $\mu$.} modifications, the argument used in \Cref{prop:theta-to-psi-theta-is-measurable} shows that $\psi_{\theta}$ is an automorphism of $(X,\mathcal B_X,\mu)$ for almost all $\theta$, and that the almost everywhere defined map $\T\to\aut(X,\mathcal B_X,\mu)$ given by $\theta\to \psi_{\theta}$ is measurable.

Since $\Upsilon\circ (R_{\alpha}\rtimes_{\tau} S)=(R_{\alpha+1/2}\rtimes_{\tau} T)\circ\Upsilon$, we also have $\Upsilon\circ (R_{\alpha}\rtimes_{\tau} S)^n=(R_{\alpha+1/2}\rtimes_{\tau} T)^n\circ\Upsilon$ for all $n\in\Z$. Since we will need to simultaneously consider the rotations by $\alpha$ and $\alpha+1/2$, the notation $\tau(n,\theta)$ becomes ambiguous, so we explicitly write $\sum_{i=0}^{n-1}\tau(\theta+i\alpha)$ and $\sum_{i=0}^{n-1}\tau(\theta+i(\alpha+1/2))$ for $n\geq 1$. We have
\begin{align*}\Upsilon( (R_{\alpha}\rtimes_{\tau} S)^n(\theta,x))&=\Upsilon(\theta+n\alpha, S^{\sum_{i=0}^{n-1}\tau(\theta+i\alpha)}(x))\\
&=(\theta+n\alpha+\iota_S(S^{\sum_{i=0}^{n-1}\tau(\theta+i\alpha)}(x))/2+c,\psi_{\theta+n\alpha}( S^{\sum_{i=0}^{n-1}\tau(\theta+i\alpha)}(x)))
\end{align*}
On the other hand, 
\begin{align*}
(R_{\alpha+1/2}\rtimes_{\tau} T)^n(\Upsilon(\theta,x))&=(R_{\alpha+1/2}\rtimes_{\tau} T)^n(\theta+\iota_S(x)/2+c,\psi_{\theta}(x))\\
&=(\theta+\iota_S(x)/2+c+n(\alpha+1/2),T^{\sum_{i=0}^{n-1}\tau(\theta+c+i(\alpha+1/2)+\iota_S(x)/2
)}(\psi_{\theta}(x)))
\end{align*}
Looking at the second coordinates, we see that for almost every $(\theta,x)$ we have 
\begin{equation}\label{eq:satanic-flip-equation}\psi_{\theta+n\alpha}( S^{\sum_{i=0}^{n-1}\tau(\theta+i\alpha)}(x))=T^{\sum_{i=0}^{n-1}\tau(\theta+c+i(\alpha+1/2)+\iota_S(x)/2
)}(\psi_{\theta}(x))
\end{equation}
Since $\iota_S(x)=\iota_T(\psi_\theta(x))-c'$ (\Cref{eq:iota-S}), we can equivalently write 

\begin{equation}\label{eq:satanic-flip-equation-enhanced}
\psi_{\theta+n\alpha}( S^{\sum_{i=0}^{n-1}\tau(\theta+i\alpha)}(x))=T^{\sum_{i=0}^{n-1}\tau(\theta+c+i(\alpha+1/2)
+(\iota_T(\psi_\theta(x))-c')/2)}(\psi_{\theta}(x))
\end{equation}
We would like to derive a relation between automorphisms of $(X,\mathcal B_X,\mu)$. In other words, remove the variable $x$ and instead have an equation between functions of $x$. A technical difficulty is that we cannot do this directly, because the term $(\iota_T(\psi_\theta(x))-c')/2$ in the exponent of $T$ depends on $x$. We now solve this difficulty by means of an elementary trick, which works thanks to the symmetry $\tau\circ R_{1/2}=-\tau$. 

We first describe this trick in general. Namely, the following is an abstract construction which works for an arbitrary automorphism which factors onto the two point system. Let $Z$ be an arbitrary automorphism of $(X,\mathcal B_X,\mu)$, and suppose that $Z$ factors onto $R_1$ via $\iota_Z\colon X\to\Z/2\Z$. Thus $Z$ exchanges two halves of the space $\{x : \iota_Z(x)=0\}$ and $\{x : \iota_Z(x)=1\}$, both of which must have measure $1/2$ because $\iota_Z$ is a factor to $R_1$. We then define $W$ by 
\[
W(x)=\begin{cases}
  Z(x), \ \iota_Z(x)=0\\
  Z^{-1}(x), \ \iota_Z(x)=1\\  
\end{cases}
\]
It is clear that then $W$ preserves $\mu$, and it is an automorphism of $(X,\mathcal B_X,\mu)$. Also note that $W\circ W=W^2=\id_X$ regardless of $Z$. 

For an odd number $k\in\Z$, we apply the previous construction to $T^k$, and define 
\[
W_k(x)=\begin{cases}
    T^{k}(x), \ \iota_T(x)-c'=0\\
    T^{-k}(x), \ \iota_T(x)-c'=1
\end{cases}
\]
We remark that because $k$ is odd, $T^k$ factors onto $R_1$ via $\iota_T$. It is clear that then $\iota_T-c'$ is also a factor map from $T^k$ to $R_1$. Thus we have defined a collection $\{W_k: k\in\Z \text{ odd}\}$ of automorphisms of $(X,\mathcal B_X,\mu)$. We remark that $W_k^2=\id_X$ for all $k$.

We claim that for $\theta\in\T$ and $n\geq 1$ odd, the rightmost term in \Cref{eq:satanic-flip-equation-enhanced} simplifies to
\begin{equation}\label{eq:satanic-flip-equation-enhanced-2}T^{\sum_{i=0}^{n-1}\tau(\theta+c+i(\alpha+1/2)
+(\iota_T(\psi_\theta(x))-c')/2)}(\psi_{\theta}(x))=W_{\sum_{i=0}^{n-1}\tau(\theta+c+i(\alpha+1/2))}(\psi_{\theta}(x))
\end{equation}
Indeed, fix $k=\sum_{i=0}^{n-1}\tau(\theta+c+i(\alpha+1/2))$. A first observation is that $k$ is an odd number, as it is a sum of an odd number of $1$'s and $-1$'s. Thus $W_k$ is well-defined. Thanks to the symmetry $\tau\circ R_{1/2}=-\tau$, we have 
\[
\sum_{i=0}^{n-1}\tau(\theta+c+i(\alpha+1/2)+(\iota_T(\psi_\theta(x))-c')/2
)=\begin{cases}
    k, \ \ \iota_T(\psi_\theta(x))-c'=0\\
    -k, \ \ \iota_T(\psi_\theta(x))-c'=1
\end{cases}
\]
It follows that the leftmost term in \Cref{eq:satanic-flip-equation-enhanced-2} is equal to $W_k(\psi_{\theta}(x))$. This proves \Cref{eq:satanic-flip-equation-enhanced-2}.

For almost all $\theta\in\T$ and $n\geq 1$ odd, we have from \Cref{eq:satanic-flip-equation-enhanced,eq:satanic-flip-equation-enhanced-2} that 
\[
\psi_{\theta+n\alpha}\circ S^{\sum_{i=0}^{n-1}\tau(\theta+i\alpha)} = W_{\sum_{i=0}^{n-1}\tau(\theta+c+i(\alpha+1/2))}\circ\psi_{\theta}
\]
Thus we have removed the variable $x$, as desired. 

If $n$ is odd and $\theta$ satisfies $\sum_{i=0}^{n-1}\tau(\theta+i\alpha)=1$, the previous equation simplifies to 
\[
\psi_{\theta+n\alpha}\circ S= W_{\sum_{i=0}^{n-1}\tau(\theta+c+i(\alpha+1/2))}\circ\psi_{\theta}
\]
We obtain
\[
\psi_{\theta}^{-1}\circ \psi_{\theta+n\alpha}\circ S=\psi_{\theta}^{-1} \circ W_{\sum_{i=0}^{n-1}\tau(\theta+c+i(\alpha+1/2))}\circ\psi_{\theta}
\]
Recall that $(W_k)^2=\id_X$ for all $k\in\Z$ odd. Therefore, the same is true for any conjugate of the form $\psi^{-1}\circ W_k\circ\psi$, for $\psi\in\aut(X,\mathcal B_X,\mu)$. Thus we have proved that for almost all $\theta$ 
\begin{equation}\label{eq:satanic-flip-equation-enhanced-3}
(\psi_{\theta}^{-1}\circ\psi_{\theta+n\alpha}\circ S)^2=\id_X
\end{equation}
provided $n$ is odd and $\sum_{i=0}^{n-1}\tau(\theta+i\alpha)=1$.

We will now do a fixed point argument to obtain $S^2=\id_X$. Let $(q_n)_{n\in\N}$ be a sequence of odd numbers as in \Cref{prop:infinitely-many-1s}. Replacing it by a subsequence if needed, we can assume that for almost all $\theta$ we have 
\[
\lim_{n\to\infty}\psi_{\theta+q_{n}\alpha}=\psi_{\theta}
\]
in $\aut(X,\mathcal B_X,\mu)$. This is possible by \Cref{prop:subsequences}, the measurability of $\theta\mapsto\psi_{\theta}$, and the fact that $(q_n\alpha)_{n\in\N}$ converges to $0$ in $\T$. 

The following argument is valid for almost all $\theta$. By \Cref{prop:infinitely-many-1s}, we can pick a subsequence $(q_{n_k})_{k\in\N}$ such that $\sum_{i=0}^{q_{n_k}-1}\tau(\theta+i\alpha)=1$ for all $k$. Then by \Cref{eq:satanic-flip-equation-enhanced-3} we have 
\[
(\psi_{\theta}^{-1}\circ\psi_{\theta+q_{n_k}\alpha} \circ S)^2=\id_X, \ \ k\in\N
\]
Since $\psi_{\theta+q_{n_k}\alpha}\to\psi_{\theta}$ as  $k\to\infty$, and composition and inversion are continuous operations in $\aut(X,\mathcal B_X,\mu)$, we have 
\begin{align*}
\id_X=\lim_{k\to\infty}(\psi_{\theta}^{-1}\circ\psi_{\theta+q_{n_k}\alpha} \circ S)^2&=\lim_{k\to\infty} \psi_{\theta}^{-1}\circ \psi_{\theta+q_{n_k}\alpha} \circ S\circ \psi_{\theta}^{-1}\circ \psi_{\theta+q_{n_k}\alpha} \circ S\\
&=\psi_{\theta}^{-1}\circ\psi_{\theta}\circ S\circ\psi_{\theta}^{-1}\circ\psi_{\theta}\circ S\\
&=S^2
\end{align*}
Thus we have proved the desired contradiction $S^2=\id_X$. 
\end{proof}
\section{Proof of the isomorphism theorem}\label{sec:proof-of-isomorphism-theorem}
We are finally ready to prove \Cref{thm:isomorphisms}. We repeat the statement for convenience of the reader.
\begin{theorem}[\Cref{thm:isomorphisms}]\label{thm:isomorphisms-section}
    Let $\alpha,\beta\in \R\smallsetminus \Q$, and let $(X,\mathcal B_X,\mu,T)$ and $(Y,\mathcal B_Y,\nu,S)$ be invertible measure-preserving systems. Assume that both transformations are ergodic, and that the measures are nonatomic. Then $R_{\alpha}\rtimes_\tau S$ is isomorphic to $R_{\beta}\rtimes_\tau T$ if and only if $S$ is flip isomorphic to $T$ and $\beta=\pm \alpha\mod 1$. 
\end{theorem}
\begin{proof}
We proved in \Cref{prop:isomorphism-theorem-easy-direction} that if $\beta=\pm \alpha\mod 1$ and $S$ is flip isomorphic to $T$, then $R_{\alpha}\rtimes_\tau S$ is isomorphic to $R_{\beta}\rtimes_\tau T$. 

Conversely, suppose that $R_{\alpha}\rtimes_\tau S$ is isomorphic to $R_{\beta}\rtimes_\tau T$. We must show that $\beta=\pm\alpha\mod 1$, and that $S$ is flip isomorphic to $T$. Let us first show this in the special case in which $S$ and $T$ are automorphisms of the same probability space $(X,\mathcal B_X,\mu)$. 
Since $\mu$ is assumed to be nonatomic and $S$ is ergodic, we must have $S^n\ne\id_X$ for all $n\ne 0$. Then it follows from \Cref{thm:different-angles} that $\beta=\pm\alpha\mod 1$. Thus  $R_{\alpha}\rtimes_\tau S$ is isomorphic to $R_{\alpha}\rtimes_\tau T$ or $R_{-\alpha}\rtimes_\tau T$. But these two are isomorphic to each other (\Cref{prop:easy-automorphisms}). Therefore $R_{\alpha}\rtimes_\tau S$ is isomorphic to $R_{\alpha}\rtimes_\tau T$. Finally, it follows from \Cref{prop:isomorphism-nice-first-coordinate} and \Cref{thm:isomorphisms-long-statement} that $S$ and $T$ are flip isomorphic. 

The general case follows from the previous case, as all nonatomic standard Borel probability spaces are isomorphic to each other \cite[Theorem 17.41]{kechris_classical_1995}. By this we mean the following. Let $(X,\mathcal B_X,\mu,T)$ and $(Y,\mathcal B_Y,\nu,S)$ be as in the statement.  We claim that $S$ has an isomorphic copy $S'$ that acts on $(X,\mathcal B_X,\mu)$. That is, if $\pi\colon Y\to X$ is an isomorphism of the probability spaces $(Y,\mathcal B_Y,\nu)$ and $(X,\mathcal B_X,\mu)$, then we can define an automorphism $S'$ of $(X,\mathcal B_X,\mu)$ by $S'=\pi\circ S\circ\pi^{-1}$. Then $(Y,\mathcal B_Y,\nu,S)$ and $(X,\mathcal B_X,\mu,S')$ are isomorphic as measure-preserving systems, via $\pi$. By \Cref{prop:isomorphism-theorem-easy-direction}, we have that $R_{\alpha}\rtimes_{\tau} S$ and  $R_{\alpha}\rtimes_{\tau} S'$ are isomorphic. As a consequence, if we assume that $R_{\alpha}\rtimes_{\tau} S$ and $R_{\beta}\rtimes_{\tau} T$ are isomorphic, it then follows that $R_{\alpha}\rtimes_{\tau} S'$ and $R_{\beta}\rtimes_{\tau} T$ are isomorphic. By the argument in the previous paragraph we have that $\beta=\pm\alpha\mod 1$, and $S'$ is flip isomorphic to $T$. But since $S$ is isomorphic to $S'$, it follows that $S$ is flip isomorphic to $T$.
\end{proof}
In the case where we take two skew products over the same rotation, we have the following result characterizing the isomorphism maps. 
\begin{theorem}\label{thm:isomorphisms-characterization}
Let $\alpha\in \R\smallsetminus \Q$. Let $(Y,\mathcal B_Y,\nu,S)$ and $(X,\mathcal B_X,\mu,T)$ be ergodic invertible measure-preserving systems with  nonatomic measures. Furthermore, suppose either that $\alpha$ has bounded type, or that $S$ is not rigid. If $R_{\alpha}\rtimes_\tau S$ and $R_{\alpha}\rtimes_\tau T$ are isomorphic, then every possible isomorphism from $R_{\alpha}\rtimes_\tau S$ to $R_{\alpha}\rtimes_\tau T$ can be written as
\[
(\theta,x)\mapsto (\theta+c,T^{f(\theta)}(\upsilon(x)))
\] 
Here $\upsilon$ is an isomorphism from $S$ to $T^{\ell}$ for some  $\ell\in\{-1,1\}$, $c\in\T$, and $f\colon\T\to\Z$ is a measurable solution of 
\[
\zeta(\theta)=f(\theta+\alpha)-f(\theta)
\] 
for $\zeta(\theta)=\tau(\theta+c)-\ell\tau(\theta)$.
\end{theorem}
\begin{proof}
We proved in \Cref{prop:easy-isomorphisms-cohomology-section} that the maps in the statement are isomorphisms. We prove the remaining direction.  The same argument used in \Cref{thm:isomorphisms-section} shows that it is sufficient to prove the statement in the special case in which $S$ and $T$ act on the same probability space. Then the desired conclusion follows from \Cref{prop:isomorphism-nice-first-coordinate} and \Cref{thm:isomorphisms-long-statement}.
\end{proof}
For a measure-preserving system  $(X,\mathcal B_X,\mu,T)$, the centralizer $C(T)$ is defined as
\[C(T)=\{S\in\aut(X,\mathcal B_X,\mu) : S\circ T=T\circ S\}\]
Observe that taking $S=T$, the previous result reduces to a description of the centralizer group of $R_{\alpha}\rtimes_{\tau} T$. In the case of a non-rigid fiber, a similar description of the centralizer group is known for general ergodic Rokhlin extensions  \cite[Proposition 5]{lemanczyk_ergodicity_2001} for which the conclusion of \Cref{prop:isomorphism-nice-first-coordinate} holds.We also note that analogous results are known for cylinder flows, also called group extensions \cite{conze_remarks_2013}. 

We finish this section with the following question, which we are unable to solve. 
\begin{question}
For an irrational $\alpha$, for which values of $c$ and $\ell$ is $\zeta$ a coboundary?
\end{question}
It turns out that this depends on $\alpha$. We observed in \Cref{prop:sufficient-conditions-zeta-coboundary} that this class always includes 
\begin{align*}
&\ell=1, \ \ c=k\alpha, \ k \in\Z\\ 
&\ell=-1, \ \ c=k\alpha+1/2, \ k\in\Z
\end{align*}
For $\alpha$ of bounded type, it follows from  Zhang's result on multidimensional cocycles  (\Cref{thm:zhang}) that this is a characterization. It is pointed out in \cite[p 245]{zhang_ergodicity_2011} that for $\alpha$ not of bounded type, the class of possible $c$'s is necessarily larger, and indeed uncountable, by a result of Merrill \cite[Theorem 2.5]{merrill_cohomology_1985}. However, we do not know if such result provides an equivalence.
\section{Irrationals with arbitrarily slow walks}\label{sec:ranges}
In this section we prove ingredients that will be used to prove \Cref{thm:zero-slow-entropy} in the next section. Given $\theta\in\T$ we will consider the number of places visited by the ``walker'' $i\to\tau(i,\theta)$ in the first $n$ steps:
\begin{equation}\label{eq:r_n-theta}
r_n(\theta)=|\{\tau(i,\theta) : i=0,\dots,n-1\}|
\end{equation}
We will be interested in the maximal value of $r_{n}(\theta)$ with $\theta$ ranging in $\T$, and the analogous quantity but where we allow the removal of a set with small measure $\delta\in (0,1)$. That is:
\begin{equation}\label{eq:r_n}
r_n=\sup_{\theta\in\T} r_n(\theta) \text{ \ \ \ \ and \ \ \ \  }r_{n}^{\delta} = \min\{k\in \N : m_{\T}(\{\theta\in \T : r_n(\theta)\leq k\})>1-\delta\}
\end{equation}
In \Cref{sec:frequency-trick} we will observe a natural Diophantine approximation condition on an abstract irrational $\alpha$, which provides times $n$ for which $r_n^\delta$ is small in relation to $n$. This idea, which we call frequency trick, is  briefly explained in \Cref{fig:frequency-trick-1}, and the main statement is \Cref{prop:tau-qn-is-mostly-zero-for-even-iterates-iterated}. 

In \Cref{sec:evil-irrational} we will see how to \textit{define} an irrational so that we can take advantage of the same trick. It follows from this construction that the sequence  $(r_n^\delta)_{n\in\N}$ can grow slower than any prescribed rate along subsequences (\Cref{thm:r-n-delta-upper-bound}), choosing $\alpha$ suitably. 
\subsection{The frequency trick}\label{sec:frequency-trick}
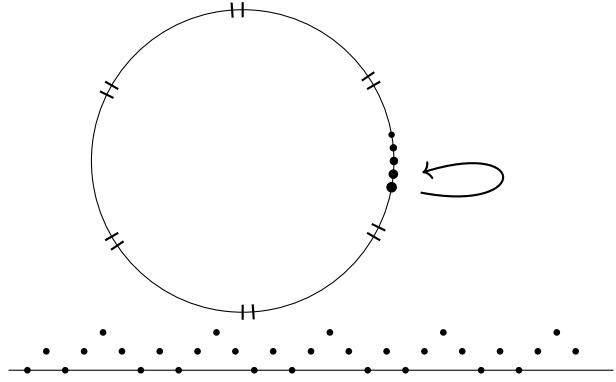
\begin{figure}
\begin{center}\begin{tikzpicture}
  \draw (0,0) circle (2cm);
  \foreach \i in {30,90,...,330} {
    \draw[thick] (\i:1.9cm) -- (\i:2.1cm);
    \draw[thick] (\i+4:1.9cm) -- (\i+4:2.1cm);
  }
  \fill (-10:2cm) circle (2pt);
  \fill (-5:2cm) circle (1.8pt);
  \fill (0:2cm) circle (1.6pt);
  \fill (5:2cm) circle (1.4pt);
  \fill (10:2cm) circle (1.2pt);
  
  \draw[->, thick, shorten >=4pt, shorten <=4pt] (-10:2.25cm) .. controls (3.8, -0.7) and (3.8, 0.3) .. (-5:2.25cm);
\end{tikzpicture}\end{center}
\begin{center}
\begin{tikzpicture}[scale=0.25]
  \draw (-1,0) -- (31,0);
  
  \foreach \x/\y in {
    0/0, 1/1, 2/0, 3/1, 4/2, 5/1, 
    6/0, 7/1, 8/0, 9/1, 10/2, 11/1,
    12/0, 13/1, 14/0, 15/1, 16/2, 17/1,
    18/0, 19/1, 20/0, 21/1, 22/2, 23/1,
    24/0, 25/1, 26/0, 27/1, 28/2, 29/1
  } {
    \fill (\x,\y) circle (5pt);
  }
\end{tikzpicture}
\end{center}
\caption{Illustration of the frequency trick for $q=6$. The partition $\vee_{i=0}^{q-1}R_{\alpha}^{-i}(\xi)$ has 12 atoms. 6 of them have length approximately 1/6, and the rest are very short. The map $\theta\to\tau(q,\theta)$ equals zero on the big atoms.  In the circle we see an element $\theta\in\T$ and some of its iterates by $R_{\alpha}^q$. Below the circle we see a representation of $i\to\tau(i,\theta)$ for $0\leq i<5q$ for the same $\theta$. The key is that $(\tau(i,\theta))_{i=0}^{q-1}$ is fully determined by the atom of $\vee_{i=0}^{q-1}R_{\alpha}^{-i}(\xi)$ containing $\theta$. Since $\theta,R_{\alpha}^q(\theta),R_{\alpha}^{2q}(\theta),R_{\alpha}^{3q}(\theta),R_{\alpha}^{4q}(\theta)$ are all in the same atom, it follows that $(\tau(i,\theta))_{i=0}^{5q-1}$ is just  $(\tau(i,\theta))_{i=0}^{q-1}$ repeated $5$ times. Thus in this example we have $r_{5q}(\theta)=3$. 
}\label{fig:frequency-trick-1}
\end{figure}
\begin{proposition}\label{prop:rational-approximations}
Suppose that 
\[|\alpha-p/q|<1/(sq^2)\]
where $p,q\in\N$ are coprime and $s\geq 1$. Let $\theta\in\T$. Then every element in $\{R_{\alpha}^i(\theta) : i=0,\dots,q-1\}$ is at distance smaller than $1/(sq)$ from some element in $\{\theta+i/q :  i=0,\dots,q-1\}$. If $s\geq 2$, then such an element is necessarily unique. The same holds for the backward orbit $\{R_{\alpha}^{-i}(\theta) : i=0,\dots,q-1\}$. 
\end{proposition}
\begin{proof}
    It is sufficient to prove the claim for $\theta=0$. This is because $R_{\theta}$ is an isometry which commutes with $R_{\alpha}$. We multiply the relation $|\alpha-p/q|<1/(sq^2)$ by $i\in\{0,\dots,q-1\}$ to find 
	\[
	|i\alpha-ip/q|\leq i/(sq^2)<1/(sq)	
	\]
	Since $p$ and $q$ are coprime, we can find $j\in\{0,\dots,q-1\}$ so that $ip\equiv j \mod q$. Then in $\T$ we have $ip/q=j/q\mod 1$. Therefore the distance $d_{\T}$ between $R_{\alpha}^{i}(0)$ and $j/q$ is strictly smaller than $1/(sq)$. 
    
    Assuming $s\geq 2$ let us verify the uniqueness of $j$. Let $j'\in\{0,\dots,q-1\}$ so that $d_\T(R_{\alpha}^{i}(0),j'/q)<1/(sq)$. Then by triangle inequality $d_\T(j/q,j'/q)<1/(sq)+1/(sq)\leq 1/q$. Since any two different elements in $\{i/q :  i=0,\dots,q-1\}$ are at distance $d_\T$ at least $1/q$, it follows that $j=j'$. 

    This proves our claim regarding $\{R_{\alpha}^i(0) : i=0,\dots,q-1\}$ and $\{i/q :  i=0,\dots,q-1\}$. The argument for the backwards orbit is similar.
\end{proof}
We will consider the partition $\xi=\{[0,1/2),[1/2,1)\}$.  
\begin{proposition}\label{prop:tau-constant}
    Let $n\geq 1$. Given $\theta\in\T$, the sequence $(\tau(R_{\alpha}^i(\theta)))_{i=0}^{n-1}$ of $1$'s and $-1$'s is completely determined by the atom of $\vee_{i=0}^{n-1}R_{\alpha}^{-i}(\xi)$ containing $\theta$. Furthermore, the map 
    \[\theta\to\tau(n,\theta)\]
    is constant on atoms of  $\vee_{i=0}^{n-1}R_{\alpha}^{-i}(\xi)$.
\end{proposition}
\begin{proof}
    The first part of the statement follows from the definitions and the fact that $\tau$ is constant on the atoms of $\xi$. The second part follows from the first, and the identity $\tau(n,\theta)=\sum_{i=0}^{n-1}\tau(R_{\alpha}^i(\theta))$.  
\end{proof}
\begin{proposition}\label{prop:tau-is-mostly-zero-for-even-iterates}
Suppose that 
\[|\alpha-p/q|<1/(sq^2)\]
where $p,q\in\N$ are coprime and $s\geq 2$. Assume further that $q$ is even. Then for every $I=[i/q,(i+1)/q)$, $i\in\{0,\dots,q-1\}$, we can find an atom $J\in\vee_{i=0}^{q-1}R_{\alpha}^{-i}(\xi)$ such that $\tau(q,\theta)=0$ for every $\theta\in J$, and
\[
\length(I\cap J)>1/q-2/(sq)
\]
\end{proposition}
\begin{proof}
    The atoms in $\vee_{i=0}^{q-1}R_{\alpha}^{-i}(\xi)$ have endpoints in $\{R_{\alpha}^{-i}(0) : i=0,\dots,q-1\}\cup \{R_{\alpha}^{-i}(1/2) : i=0,\dots,q-1\}$. By \Cref{prop:rational-approximations}, every element in $\{R_{\alpha}^{-i}(0) : i=0,\dots,q-1\}$ is at distance strictly smaller than $1/(sq)$ from a unique element in $\{i/q : i=0,\dots,q-1\}$. Similarly, every element in $\{R_{\alpha}^{-i}(1/2) : i=0,\dots,q-1\}$ is at distance strictly smaller than $1/(sq)$ from a unique element in $\{i/q +1/2 : i=0,\dots,q-1\}$. Crucially, because $q$ is even, we have 
    \[
    \{i/q +1/2 : i=0,\dots,q-1\}=\{i/q  : i=0,\dots,q-1\}
    \]
    Let $I$ be as in the statement, let $\theta_I$ be the midpoint of $I$, and let $J$ be the atom from $\vee_{i=0}^{q-1}R_{\alpha}^{-i}(\xi)$ containing $\theta_I$. The observation in the previous paragraph shows that $\length(I\cap J)> \length(I)-2/(sq)\geq 0$. 
    
    We claim that $\tau(q,\theta)=0$ for all $\theta\in J$. Since this function is constant over $J$ (\Cref{prop:tau-constant}), it is sufficient to verify that  $\tau(q,\theta_I)=0$. By \Cref{prop:rational-approximations} again, we have that each element in $\{R_{\alpha}^i(\theta_I) : i=0,\dots,q-1\}$ is at distance smaller than $1/(sq)$ from a unique element in $\{i/q+\theta_I  : i=0,\dots,q-1\}$. Let us denote by 
    \[
    f\colon \{R_{\alpha}^i(\theta_I) : i=0,\dots,q-1\} \to \{i/q+\theta_I  : i=0,\dots,q-1\}
    \]
    the associated bijection. Thus $d_{\T}(\theta,f(\theta))<1/(sq)$ for all $\theta$ in the domain of $f$. On the other hand, the fact that $\theta_I$ is at distance exactly $1/(2q)$ from $i/q$ and $(i+1)/q$ implies that every element in $\{i/q+\theta_I  : i=0,\dots,q-1\}$ is at distance at least $1/(2q)$ from $0$ and $1/2$. It follows that for all $\theta$ in the domain of $f$, both $\theta$ and $f(\theta)$ are in the same atom of $\xi=\{[0,1/2),[1/2,1)\}$. Hence
    \[
    \sum_{i=0}^{q-1} \tau(R^i_{\alpha}(\theta_I))=
    \sum_{i=0}^{q-1} \tau(R^i_{1/q}(\theta_I))=0
    \]
\end{proof}
\begin{proposition}\label{prop:tau-qn-is-mostly-zero-for-even-iterates-iterated}    
Suppose that 
\[|\alpha-p/q|<1/(s^2q^2)\]
where $p,q\in\N$ are coprime, $q$ is even, and $s\geq 4$ is an integer. Then for every $I=[i/q,(i+1)/q)$, $i\in\{0,\dots,q-1\}$, we can find an interval $J\subset I$ with the following properties:
\begin{enumerate}
    \item $\length(J)\geq 1/q-4/(sq)$.
    \item There exists a single atom in $\vee_{i=0}^{q-1}R_{\alpha}^{-i}(\xi)$  containing the iterates $\cup_{i=0}^{s} R_{\alpha}^{iq}(J)$.
    \item For every $\theta\in J$, the sequence $(\tau(R_{\alpha}^i(\theta)))_{i=0}^{qs-1}$ equals $(\tau(R_{\alpha}^i(\theta)))_{i=0}^{q-1}$ repeated $s$ times. 
    \item For every $\theta\in J$ we have $\tau(qi,\theta)=0$ for every $i=0,\dots,s$. 
    \item For every $\theta\in J$ the sequence $(\tau(i,\theta))_{i=0}^{qs-1}$ equals $(\tau(i,\theta))_{i=0}^{q-1}$ repeated $s$ times. 
    \item For every $\theta\in J$ we have 
    $r_{sq}(\theta)=r_q(\theta)$. 
\end{enumerate}
Furthermore, we have the following upper bound for $r^\delta_{sq}$, provided $4/s<\delta$:
\[
r^\delta_{sq}\leq r_q
\]
\end{proposition}
\begin{proof}
    Let $I$ be as in the statement. Let $A$ be the atom of $\vee_{i=0}^{q-1}R_{\alpha}^{-i}(\xi)$ associated to $I$ by \Cref{prop:tau-is-mostly-zero-for-even-iterates}, and let $J'=A\cap I$. It follows from \Cref{prop:tau-is-mostly-zero-for-even-iterates} that the length of $J'$ is at least $1/q-2/(sq)$.  Let $J\subset J'$ be   obtained by trimming each endpoint of $J'$ by $1/(sq)$. By \Cref{prop:tau-is-mostly-zero-for-even-iterates} we have $\length(J')>1/q-2/(sq)$, so 
    \[
        \length(J)>1/q-2/(sq)-2/(sq)=1/q-4/(sq)
    \]
    The set $J'$ is fully contained in some atom from $\vee_{i=0}^{q-1}R_{\alpha}^{-i}(\xi)$. In order to prove item (2), we show that $\cup_{i=0}^s R_{\alpha}^{iq}(J)$ is contained in $J'$. Let $\theta\in J$.  We verify that the iterates $R_{\alpha}^q(\theta)$, $\dots$, $R_{\alpha}^{sq}(\theta)$ all belong to $J'$. Fix $i\in\{1,\dots,s\}$. We multiply $|\alpha-p/q|<1/(s^2q^2)$ by $iq$. We see that $|iq\alpha-ip
|<  i/(s^2q)\leq 1/(sq)$. Since $ip$ is an integer, we have $d_{\T}(R^{iq}_{\alpha}(0),0)< 1/(sq)$. But since $R_{\alpha}$ is an isometry of $(\T,d_{\T})$, we also have $d_{\T}(R^{iq}_{\alpha}(\theta),\theta)< 1/(sq)$. Since $\theta$ is sufficiently far from the endpoints of $J'$, meaning at distance at least $1/(sq)$, it follows that $R^{iq}_{\alpha}(\theta)$ belongs to $J'$ as desired. Thus we have verified  item (2).

We now verify item (3). Let $\theta\in J$. By (2), the iterates $\{R_{\alpha}^{iq}(\theta) : i=0,\dots,s\}$ all fall into the same atom of $\vee_{k=0}^{q-1}R_{\alpha}^{-k}(\xi)$. Then it follows from \Cref{prop:tau-constant} that $(\tau(R_{\alpha}^{iq+j}(\theta)))_{j=0}^{q-1}$ is the same for $i=0,\dots,s$.

We now prove item (4). It follows from \Cref{prop:tau-is-mostly-zero-for-even-iterates} that $\tau(q,\theta)$ is zero for $\theta\in J'$. Now take $\theta\in J$. For $i\in\{0,\dots,s\}$ we can write 
    \[
    \tau(iq,\theta)=\sum_{j=0}^{i-1} \tau(q,R^{jq}_{\alpha}(\theta))
    \]
    Thus a sufficient condition for the desired property 
    \[\tau(iq,\theta)=0, \ i\in\{0,\dots,s\}\]
    is that $\theta$, $R_{\alpha}^q(\theta)$, $\dots$, $R_{\alpha}^{sq}(\theta)$ all belong to $J'$. This follows from item (2). Thus item (4) holds.

    We now prove item (5). Let $\theta\in J$, and let $i\in\{0,\dots,qs-1\}$. By Euclidean division by $q$ we can write $i=jq+i'$, where $i'\in\{0,\dots,q-1\}$ and $j\in\{0,\dots,s-1\}$. Thus $i\equiv i'\mod q$, and 
    \begin{align*}
    \tau(i,\theta)&=\tau(jq,\theta)+\tau(i',R^{jq}_{\alpha}(\theta)) &(\text{cocycle relation})\\
    & = \tau(i',R^{jq}_{\alpha}(\theta))  &(\text{item (4)})\\
    & = \tau(i',\theta)  
    \end{align*}
    In the last equality we used that $R^{jq}_{\alpha}(\theta)$ lives in the same atom of $\vee_{i=0}^{q-1}R_{\alpha}^{-i}(\xi)$ as $\theta$ by item (2), so $\tau(i',R^{jq}_{\alpha}(\theta))=\tau(i',\theta)$ by \Cref{prop:tau-constant}.
    
    Item (6) follows from item (5). For the ``furthermore'' claim, let $\mathcal J$ be the union of the $q$ intervals $J$ given by the statement.  For every $\theta$ in $\mathcal J$ we have $r_{sq}(\theta)\leq r_q$ by item (6). Furthermore, by item (1) the set $\mathcal J$ has measure at least $1-4/s$, which is larger than $1-\delta$ provided $4/s<\delta$. Thus $r^\delta_{sq}\leq r_q$ as claimed. 
\end{proof}
\begin{remark}
    The statement of \Cref{prop:tau-qn-is-mostly-zero-for-even-iterates-iterated} contains slightly more information than what we need to prove upper bounds for $r_n^\delta$. This extra information will be important to prove upper bounds for Hamming covering numbers in \Cref{prop:evil-upper-bound}.
\end{remark}
\subsection{Continued fractions}
Let us recall some basic results about continued fractions. Further details can be found, for instance, in  \cite[\S 3]{einsiedler_ergodic_2011}. If $a_1,\dots,a_n$ are positive integers then we use the following notation:
\begin{equation}
[0;a_1,\dots,a_n]=\frac{1}{a_1+\frac{1}{a_2+\dots +\frac{1}{a_{n-1}+\frac{1}{a_n}}}}\in\mathbb{Q}
\end{equation}
Similarly, if $a_1,a_2,\dots$ is an infinite sequence of positive integers, then we write
\[
[0;a_1,a_2,\dots]=\frac{1}{a_1+\frac{1}{a_2+\dots}}
\]
When $\alpha=[0;a_1,a_2,\dots]$ then the expression $[0;a_1,a_2,\dots]$ is called the continued fraction expansion for $\alpha$. We associate to $(a_n)_{n\geq 1}$ two sequences of integers $(p_n)_{n\geq -1}$ and $(q_n)_{n\geq -1}$, defined by the recursive formulas
\begin{equation}\label{eq:q-ns-recursive-expression}
p_{n+1}=a_{n+1}p_n+p_{n-1}, \ \ 
q_{n+1}=a_{n+1}q_n+q_{n-1}
\end{equation}
and the initial values $p_{-1}=1$, $q_{-1}=0$, $p_{0}=0$, $q_{0}=1$. For all $n\geq 0$ we have 
\begin{equation}\label{eq:approximation-alpha}
    |\alpha-\frac{p_n}{q_n}|<\frac{1}{a_{n+1}q_n^2}
\end{equation}
\subsection{On defining $\alpha$ to control $r_n^\delta$}\label{sec:evil-irrational} 
Let $(s_k)_{k\in\N}$ be an arbitrary sequence of even numbers with $s_k\geq 4$ for all $k$ and with $\lim_{k\to\infty}s_k=\infty$. We define an irrational by
\begin{equation}\label{eq:evil-irrational}
    \alpha=[0;1,1,s_1^2,s_1^2,s^2_2,s_2^2,s^2_3,s_3^2,s^2_4,s_4^2,s^2_5,s_5^2,s^2_6,\dots]
\end{equation}
That is, $a_1=1$, $a_2=1$, and for $n\geq 3$ 
    \[
a_n=s_k^2, n=2k+1 \text{ or } n=2k+2 
    \]
The idea is that one can control the growth of $(r_n^\delta)_{n\in\N}$ by choosing $(s_k)_{k\geq 1}$ suitably. The key is that we can combine \Cref{prop:tau-qn-is-mostly-zero-for-even-iterates-iterated} with the inequality
\[
|\alpha-\frac{p_{2k}}{q_{2k}}|<\frac{1}{s_k^2q_{2k}^2} \ \ \ k\geq 2.
\]
\begin{proposition}\label{prop:parity}Let $(p_n)_{n\geq -1}$ and $(q_n)_{n\geq -1}$ be associated to $\alpha$ as in \Cref{eq:q-ns-recursive-expression}. 
    For all $n\geq 1$ we have that $q_n$ is even exactly when $n$ is even. 
    \end{proposition}
\begin{proof}To verify this, let us compute the first values using \Cref{eq:q-ns-recursive-expression}:
    \[
    q_{-1}=0, \ q_{0}=1, \ q_{1}=1, \ q_{2}=2, \ q_{3}=2s^2_1+1, \dots
    \]
    Thus our claim holds for $n=1$ and $n=2$. By induction, take $n\geq 2$, assume the claim holds for all $i\leq n$, and let us prove it holds for $n+1$. Since  $q_{n+1}=a_{n+1}q_{n}+q_{n-1}$ and $a_{n+1}$ is even, we see that $q_{n+1}$ has the same parity as $q_{n-1}$. But $q_{n-1}$ has the same parity as $n-1$ by inductive hypothesis. To finish the argument note that $n-1$ and $n+1$ have the same parity. Thus we have proved that $q_n$ is even exactly when $n$ is even.
    \end{proof}
\begin{proposition}\label{prop:short-range-sk-q2k}
Let $\delta\in(0,1)$. For all $k$ large enough (meaning $4/s_k<\delta$) we have 
\begin{equation}\label{eq:short-range-sk-q2k}
        r_{s_kq_{2k}}^{\delta}\leq q_{2k}     \end{equation}
    \end{proposition}
    \begin{proof}
         Fix $\delta\in(0,1)$. Since $\lim_{k\to\infty}s_k=\infty$, for all $k$ large enough we have $4/s_k<\delta$. Fix $k$ with this property. By definition of $\alpha$ and \Cref{eq:approximation-alpha} we have
\[
|\alpha-\frac{p_{2k}}{q_{2k}}|<\frac{1}{s_k^2q_{2k}^2}
\]
Observe that $q_{2k}$ is even by \Cref{prop:parity}. Then $r^{\delta}_{s_kq_{2k}}\leq r_{q_{2k}}$ by \Cref{prop:tau-qn-is-mostly-zero-for-even-iterates-iterated}. But we always have the trivial bound $r_{q_{2k}}\leq q_{2k}$ (see \Cref{eq:r_n}). 
\end{proof}

\begin{proposition}\label{prop:dependence}
    Let $k\geq 2$. Then $(p_i)_{i=-1}^{2k}$ and $(q_i)_{i=-1}^{2k}$ are fully determined by $s_1,\dots,s_{k-1}$. 
\end{proposition}   
\begin{proof}
    This follows from \Cref{eq:q-ns-recursive-expression}.
\end{proof}
In the next result we use our construction to prove that $(r_n^\delta)_{n\in\N}$ can grow slower than any prescribed rate, along subsequences. In \Cref{sec:slow-entropy} we will not use this exact statement, but rather we repeat the proof replacing $f$ by a slow entropy scale. 
\begin{theorem}\label{thm:r-n-delta-upper-bound}
    Let $f\colon\N\to\N_{\geq 1}$ be a function with $\lim_{n\to\infty} f(n)=\infty$. Then there exists an irrational number $\alpha$ with the following property. For every $\delta\in(0,1)$ we have 
    \[
    \liminf_{n\to\infty} r_{n}^{\delta}/f(n)=0
    \]
\end{theorem}
\begin{proof}
The  definition of $(s_k)_{k\in\N}$ will be recursive or inductive, in the sense that we choose $s_k$ depending on $s_1,\dots,s_{k-1}$. Thanks to \Cref{prop:dependence}, we can also use  $(p_{i})_{i=-1}^{2k}$ and  $(q_{i})_{i=-1}^{2k}$ in the definition of $s_k$. 

Let $s_1=4$. Now let $k\geq 2$ and assume we have chosen $s_1,\dots,s_{k-1}$. Since $f(n)\to\infty$ as $n\to\infty$, it is also true that $f(s\cdot q_{2k})\to\infty$ as $s\to\infty$. Choose $s_k$ as an even number larger than 4, larger than $k$, and with the property \[f(s_kq_{2k})>k\cdot q_{2k}\]

Now let us verify that $\alpha$ satisfies the claim in the statement. By \Cref{prop:short-range-sk-q2k} we have for all $k$ large enough
    \[
    \frac{r_{s_kq_{2k}}^{\delta}}{f(s_kq_{2k})}\leq \frac{q_{2k}}{f(s_{k}q_{2k})}\leq \frac{1}{k}
    \]
Thus the liminf of $r_{n}^{\delta}/f(n)$ as $n\to\infty$ is zero, as witnessed by the subsequence $(s_{k}q_{2k})_{k\in\N}$.
\end{proof}
\section{Proof of \Cref{thm:zero-slow-entropy}}\label{sec:slow-entropy}
In this section we will prove \Cref{thm:zero-slow-entropy}, which states that for any slow entropy scale $\mathbf a$, we can find an irrational $\alpha$ so that $R_{\alpha}\rtimes_{\tau} T$ has zero slow entropy with respect to $\mathbf a$, for any choice of $T$. This will follow from the construction we did in \Cref{sec:evil-irrational}, combined with estimates for Hamming covering numbers. 

Before going into the proof, we need to review a number of definitions.
\subsection{Measure-theoretic slow entropy}
We will now review  necessary definitions about measure-theoretic slow entropy, see \cite[\S 4]{kanigowski_survey_2024} for further details.

Fix an invertible measure-preserving system $(X,\mathcal B_X,\mu,T)$ and a finite partition $\xi$ whose elements we call atoms.  We denote by $\xi(x)$ the atom in $\xi$ containing $x\in X$, and let
\[
\delta_\xi(x,y)=\begin{cases}
    1 &\xi(x)\ne\xi(y)\\
    0&\xi(x)=\xi(y)
\end{cases}
\]
Given a nonempty set $F\Subset \Z$, consider the dynamically generated pseudometric
\[
d^{H,T}_{\xi,F}(x,y)=\frac{1}{|F|}\sum_{i\in F}\delta_{\xi}(T^i(x),T^i(y))
\]
Then we define a corresponding pseudoball as  \[B^{H,T}_{\xi,F}(x,\epsilon)=\{y\in X : d^{H,T}_{\xi,F}(x,y)<\epsilon\}, \ \ \epsilon>0\]
This allows us to define the covering numbers
\[S^H_{\xi}(T,F,\epsilon,\delta)=\min\{k\in\N : \text{there are }x_1,\dots,x_k\in X, \mu(\cup_{i=1}^k B^{H,T}_{\xi,F}(x_i,\epsilon))>1-\delta\}.\]
These covering numbers are monotone in the sense of the following remark. We say that $\eta$ refines $\xi$, denoted $\xi\leq\eta$, if every atom in $\eta$ is contained in some atom from $\xi$.
\begin{remark}\label{remark:monotonicity}
 $S^H_{\xi}(T,F,\epsilon,\delta)\leq S^H_{\eta}(T,F,\epsilon',\delta')$ whenever $\epsilon'\leq \epsilon$, $\delta'\leq\delta$, and $\xi\leq\eta$. 
\end{remark}
In the next result we write $\xi\vee\eta=\{P\cap Q : P\in\xi,Q\in\eta\}$ for two partitions $\xi$ and $\eta$. 
\begin{proposition}\label{prop:S-H-submultiplicative}
$S^H_{\xi\vee\eta}(T,F,4\epsilon,\delta_1+\delta_2)\leq S^H_\xi(T,F,\epsilon,\delta_1)\cdot S^H_{\eta}(T,F,\epsilon,\delta_2)$, where $\epsilon,\delta_1,\delta_2>0$.
\end{proposition}
\begin{proof}
First observe that $d^{H,T}_{\xi\vee\eta, F}(x,y)\leq d^{H,T}_{\xi,F}(x,y)+d^{H,T}_{\eta,F}(x,y)$. This follows by counting coding disagreements in each partition. Next, let $x_1,\dots,x_n$ be a collection of points such that $\cup_{i=1}^{n}B^{H,T}_{\xi,F}(x_i,\epsilon)$ has measure larger than $1-\delta_1$, and let $y_1,\dots,y_m$ be a collection of points such that $\cup_{i=1}^{m}B^{H,T}_{\eta,F}(y_i,\epsilon)$ has measure larger than $1-\delta_2$. Then  \[\cup_{i=1}^n\cup_{j=1}^m B^{H,T}_{\xi,F}(x_i,\epsilon)\cap B^{H,T}_{\eta,F}(y_j,\epsilon)
\]
has measure larger than $1-\delta_1-\delta_2$. For each nonempty intersection $B^{H,T}_{\xi,F}(x_i,\epsilon)\cap B^{H,T}_{\eta,F}(y_j,\epsilon)$ choose an element $z_{i,j}$ in this intersection. By the observation at the beginning of this proof, $B^{H,T}_{\xi,F}(x_i,\epsilon)\cap B^{H,T}_{\eta,F}(y_j,\epsilon)$ is contained in $B^{H,T}_{\xi\vee\eta, F}(z_{i,j},4\epsilon)$. Thus the union of at most $n\cdot m$ pseudoballs $B^{H,T}_{\xi\vee\eta, F}(z_{i,j},4\epsilon)$ covers measure greater than $1-\delta_1-\delta_2$.
\end{proof}
By a \textit{scale} $\mathbf{a}=\{a_n(t)\}_{n\in\N,t>0}$  we mean a sequence of nondecreasing functions $a_n\colon (0,\infty)\to (0,\infty)$ such that $a_n(t)$ converges to infinity as $n$ tends to infinity, for every fixed $t$. For notational convenience, given a possibly empty set $A\subset [0,\infty)$ we define
\[
\Sup(A)=\sup(A\cup\{0\})=\begin{cases}
\sup(A) &A\ne\emptyset\\
0 &A=\emptyset
\end{cases}
\]
Given a finite measurable partition $\xi$ of $X$ and positive real numbers $\epsilon,\delta$ we define
\begin{equation}\label{eq:ent-F-partition-parameters}
\ent_{\mu}^{\mathbf{a}}(T,\xi,\epsilon,\delta)=\Sup\{t> 0 : \liminf_{n\to\infty}\frac{S_{\xi}^H(T,\{0,\dots,n-1\},\epsilon,\delta)}{a_n(t)}>0\}
\end{equation}
Next we define
\begin{equation}\label{eq:ent-F-partition}
\ent_{\mu}^{\mathbf{a}}(T,\xi)=\lim_{\epsilon,\delta\to 0} \ent_{\mu}^{\mathbf{a}}(T,\xi,\epsilon,\delta)=\sup_{\epsilon,\delta>0} \ent_{\mu}^{\mathbf{a}}(T,\xi,\epsilon,\delta)
\end{equation}
Finally, we take the supremum over all finite measurable partitions of $X$
\begin{equation}\label{eq:ent-F}
\ent_{\mu}^{\mathbf{a}}(T)=\sup_{\xi}\ent_{\mu}^{\mathbf{a}}(T,\xi)
\end{equation}
This is the lower slow entropy of $T$ with respect to $\mathbf a$. 

We finish this subsection by proving the following elementary and useful result. 
\begin{proposition}\label{prop:partition-vanishes}
	Let $(X,\mathcal B_X,\mu,T)$ be an invertible measure-preserving system. Let $\xi$ be a finite partition of $X$ such that $n\mapsto S^H_{\xi}(T,\{0,\dots,n-1\},\epsilon,\delta)$ is bounded for all $\epsilon,\delta>0$. Then for every finite partition $\eta$ we have  $\ent^{\mathbf a}_{\mu}(T,\xi\vee\eta)=\ent^{\mathbf a}_{\mu}(T,\eta)$. 
\end{proposition}
\begin{proof}
	Since $\eta\leq \xi\vee\eta$, we have $\ent^{\mathbf a}_{\mu}(T,\eta)\leq \ent^{\mathbf a}_{\mu}(T,\xi\vee\eta)$. For the other inequality we use \Cref{prop:S-H-submultiplicative}. From this result we have
\begin{align*}
	\ent_{\mu}^{\mathbf a}(T,\xi\vee\eta,4\epsilon,2\delta)&=\Sup\{t>0 : \liminf_{n\to\infty} \frac{S^H_{\xi\vee\eta}(T,\{0,\dots,n-1\},4\epsilon,2\delta)}{a_n(t)}>0\}\\
	&\leq \Sup\{t>0 :  \liminf_{n\to\infty} S^H_{\xi}(T,\{0,\dots,n-1\},\epsilon,\delta)\cdot \frac{S^H_{\eta}(T,\{0,\dots,n-1\},\epsilon,\delta)}{a_n(t)}>0\}
\end{align*}
Since $S^H_{\xi}(T,\{0,\dots,n-1\},\epsilon,\delta)$ is at least 1 and bounded above by assumption, this term has no effect on whether the liminf in the equation above is zero or positive. We obtain
\begin{align*}
	&= \Sup\{t>0 :  \liminf_{n\to\infty} \frac{S^H_{\eta}(T,\{0,\dots,n-1\},\epsilon,\delta)}{a_n(t)}>0\}\\
	&=\ent_{\mu}^{\mathbf a}(T,\eta,\epsilon,\delta)
\end{align*}
Thus we have proved $\ent_{\mu}^{\mathbf a}(T,\xi\vee\eta,4\epsilon,2\delta)\leq \ent_{\mu}^{\mathbf a}(T,\eta,\epsilon,\delta)$. Taking $\epsilon,\delta\to 0$ we obtain $\ent_{\mu}^{\mathbf a}(T,\xi\vee\eta)\leq \ent_{\mu}^{\mathbf a}(T,\eta)$ as claimed.
\end{proof}
\subsection{Proof of \Cref{thm:zero-slow-entropy}} The following important estimate is basically a continuation of \Cref{prop:tau-qn-is-mostly-zero-for-even-iterates-iterated}, where we obtain a nontrivial upper bound for Hamming covering numbers.  We employ the usual notation for product partitions $\xi\times\eta=\{P\times Q : P\in\xi,\ Q\in\eta\}$. Observe that the upper bound in the next result depends on $\alpha$ and the cardinality of $\eta$, but not on $(X,\mathcal B_X,\mu,T)$. 
\begin{proposition}[continuation of \Cref{prop:tau-qn-is-mostly-zero-for-even-iterates-iterated}]\label{prop:evil-upper-bound} Let $(X,\mathcal B_X,\mu,T)$ be an invertible measure-preserving system. Consider the partition $\xi=\{[0,1/2),[1/2,1)\}$ of $\T$, and let  $\eta$ be an arbitrary finite partition of $X$. 
	Let $\alpha$, $s$, $q$, and $\delta$ be as in the statement of  \Cref{prop:tau-qn-is-mostly-zero-for-even-iterates-iterated} (so $4/s<\delta$), and let $\epsilon>0$. Then we can bound the covering numbers for $R_{\alpha}\rtimes_\tau T$ at time $sq$ as 
\[
S^{H}_{\xi\times\eta}(R_{\alpha}\rtimes_\tau T,\{0,\dots,s q-1\},\epsilon,\delta)\leq q|\eta|^q
\]
\end{proposition}
\begin{proof}
Let $J_1,\dots,J_q$ be the $q$ disjoint intervals given by \Cref{prop:tau-qn-is-mostly-zero-for-even-iterates-iterated}, so each one has measure at least $(1-4/s)/q>(1-\delta)/q$ and their union has measure at least $1-\delta$. Let us focus on one of these intervals, which we denote by $J$. In order to prove the statement, it is sufficient to show that $J\times X$ is contained in a union of at most $|\eta|^q$ pseudoballs of radius $\epsilon$ in the distance $d^{H,R_{\alpha}\rtimes_\tau T}_{\xi\times\eta,\{0,\dots,sq-1\}}$.

Take $\theta$ and $\theta'$ in $J$. Since $\theta,\theta'$ belong to the same atom of $\vee_{i=0}^{q-1}R_{\alpha}^{-i}(\xi)$, we have that $(\xi(R_{\alpha}^i(\theta)))_{i=0}^{q-1}$ is equal to $(\xi(R_{\alpha}^i(\theta')))_{i=0}^{q-1}$. Then item (3) in \Cref{prop:tau-qn-is-mostly-zero-for-even-iterates-iterated} implies that $(\xi(R_{\alpha}^i(\theta)))_{i=0}^{qs-1}$ is equal to $(\xi(R_{\alpha}^i(\theta')))_{i=0}^{qs-1}$. Thus \[d^{H,R_{\alpha}}_{\xi,\{0,\dots,sq-1\}}(\theta,\theta')=0.\]
Now let $x,x'\in X$ and consider the pairs $(\theta,x)$ and $(\theta',x')$. It follows that in order to  estimate $d^{H,R_{\alpha}\rtimes_\tau T}_{\xi\times\eta,\{0,\dots,sq-1\}}((\theta,x),(\theta',x'))$, it is sufficient to count discrepancies  associated to  $\eta$ in the second coordinate. The orbit of $(\theta,x)$ by $R_{\alpha}\rtimes_\tau T$ with times $0,\dots,sq-1$ can be written as 
\[
((R_{\alpha}^{i}(\theta),T^{\tau(i,\theta)}(x)))_{i=0}^{sq-1}
\]
Item (5) in \Cref{prop:tau-qn-is-mostly-zero-for-even-iterates-iterated} ensures that $(\tau(i,\theta))_{i=0}^{sq-1}$ equals $(\tau(i,\theta))_{i=0}^{q-1}$ repeated $s$ times. It follows that $(T^{\tau(i,\theta)}(x))_{i=0}^{sq-1}$ equals $(T^{\tau(i,\theta)}(x))_{i=0}^{q-1}$ repeated $s$ times. We also note that $(\tau(i,\theta))_{i=0}^{sq-1}$ is equal to $(\tau(i,\theta'))_{i=0}^{sq-1}$. Thus we can estimate the Hamming pseudodistance between $(\theta,x)$ and $(\theta',x')$ as 
\begin{align*}d^{H,R_{\alpha}\rtimes_\tau T}_{\xi\times\eta,\{0,\dots,sq-1\}}((\theta,x),(\theta',x'))&=\frac{1}{sq}\sum_{i=0}^{sq-1}\delta_{\eta}(T^{\tau(i,\theta)}(x),T^{\tau(i,\theta')}(x'))\\
&= \frac{1}{q}\sum_{i=0}^{q-1}\delta_{\eta}(T^{\tau(i,\theta)}(x),T^{\tau(i,\theta')}(x'))
\end{align*}
We can re-write the last term as 
\begin{equation}\label{eq:weighted}
=\sum_{i\in F}o(i)\delta_{\eta}(T^{i}(x),T^{i}(x'))
\end{equation}
Here $F=\{\tau(i,\theta) : i=0,\dots,q-1\}$, and $o(i)$ is the proportion of elements $j$ in $\{0,\dots,q-1\}$ such that $\tau(j,\theta)=i$. \Cref{eq:weighted} shows that 
\[
d^{H,T}_{\eta,F}(x,x')=0\Rightarrow d^{H,R_{\alpha}\rtimes_\tau T}_{\xi\times\eta,\{0,\dots,sq-1\}}((\theta,x),(\theta',x'))=0\]
Consider the partition $\eta^F=\vee_{i\in F}T^{-i}(\eta)$, which has at most $|\eta|^{|F|}$ atoms. We pick one element inside each nonempty atom. Let $x_1,\dots,x_k$ be the collection of points obtained in this manner. Thus every $x\in X$ satisfies $d^{H,T}_{\eta,F}(x,x_i)=0$ for some $i=1,\dots,k$, and for fixed $\theta\in J$ we have 
\[
J\times X\subset \bigcup_{i=1}^k B^{H,R_{\alpha}\rtimes_\tau T}_{\xi\times\eta,\{0,\dots,sq-1\}}((\theta,x_i),\epsilon)
\]
To finish the argument, note from the definition of $F$ that $|F|\leq q$. Thus $k\leq |\eta|^q$. 
\end{proof}
We are now ready to prove \Cref{thm:zero-slow-entropy}.
\begin{theorem}[\Cref{thm:zero-slow-entropy}]\label{thm:zero-slow-entropy-section}
	Let $\mathbf a$ be an arbitrary scale for slow entropy. Then there exists $\alpha\in\R\smallsetminus\Q$ with the following property. For every invertible measure-preserving system   $(X,\mathcal B_X,\mu,T)$ we have
	\[
	\ent_{m_{\T}\times\mu}^{\mathbf a}(R_{\alpha}\rtimes_\tau T)=0
	\]
\end{theorem}
\begin{proof}
    We define $\alpha$ using the construction from \Cref{sec:evil-irrational}. We define $\alpha$ by \Cref{eq:evil-irrational}, so we need to choose a sequence $(s_k)_{k\in\N}$. The basic idea is the same as that used in the proof of \Cref{thm:r-n-delta-upper-bound}. 
    
    Our definition of $(s_k)_{k\in\N}$ is recursive or inductive, in the sense that $s_k$ is chosen depending on $s_1,\dots,s_{k-1}$. We start by setting  $s_1=4$. Now let $k\geq 2$ and assume that $s_1,\dots,s_{k-1}$ have been defined. By \Cref{prop:dependence}, this also determines the finite sequences $(p_n)_{n=-1}^{2k}$ and $(q_n)_{n=-1}^{2k}$ associated to $\alpha$, and we can choose $s_k$ depending on them.

    Since $a_n(1/k)$ tends to infinity as $n\to\infty$, the same is true for $a_{s\cdot q_{2k}}(1/k)$ as  $s\to\infty$. It follows that $\log(a_{s\cdot q_{2k}}(1/k))$ tends to infinity as  $s\to\infty$. Thus for all $s$ large enough, $\log(a_{s\cdot q_{2k}}(1/k))$ is larger than $kq_{2k}$.  
    We choose $s_k$ as an even number larger than $k$ and such that 
    \begin{equation}\label{eq:defining-condition-s-k}
    \log(a_{s_kq_{2k}}(1/k))>kq_{2k}
    \end{equation}
    This finishes our definition of $(s_k)_{k\geq 1}$  and of $\alpha$. We shall now verify that $\alpha$ has the intended properties. 
    
    Let $(X,\mathcal B_X,\mu,T)$ be an invertible measure-preserving system. We start by proving  
    \begin{equation}\label{eq:zero-slow-entropy-nice-partition}
    \ent^{\mathbf a}_{m_{\T}\times\mu}(R_{\alpha}\rtimes_\tau T,\xi\times\eta)=0
    \end{equation}
    where $\xi=\{[0,1/2),[1/2,1)\}$ and $\eta$ is an arbitrary finite partition of $X$. According to the definition of slow entropy, it is sufficient to take $\epsilon,\delta\in(0,1)$, $t>0$, and prove that  
    \[
    \liminf_{n\to\infty}\frac{S^H_{\xi\times\eta}(R_{\alpha}\rtimes_\tau T,\{0,\dots,n-1\},\epsilon,\delta)}{a_n(t)}=0
    \]
    We will show that this is witnessed by the subsequence $(s_kq_{2k})_{k\geq 1}$, meaning that  
    \begin{equation}\label{eq:zero-limit-along-s-k-q-2-k}
\lim_{k\to\infty}\frac{S^H_{\xi\times\eta}(R_{\alpha}\rtimes_\tau T,\{0,\dots,s_kq_{2k}-1\},\epsilon,\delta)}{a_{s_kq_{2k}}(t)}=0
    \end{equation}
    For all $k$ large enough we have $4/s_k<\delta$. By \Cref{prop:evil-upper-bound} we have \[S^H_{\xi\times\eta}(R_{\alpha}\rtimes_\tau T,\{0,\dots,s_{k}q_{2k}-1\},\epsilon,\delta)\leq q_{2k} |\eta|^{q_{2k}}\]
    Recall that for every $n$, the function $a_n(t)$ is non-decreasing in $t$. If $k$ is large enough so that $1/k<t$, we have $a_{s_kq_{2k}}(1/k)\leq a_{s_kq_{2k}}(t)$. Thus 
\[
\frac{S^H_{\xi\times\eta}(R_{\alpha}\rtimes_\tau T,\{0,\dots,s_kq_{2k}-1\},\epsilon,\delta)}{a_{s_kq_{2k}}(t)}\leq \frac{q_{2k}|\eta|^{q_{2k}}}{a_{s_kq_{2k}}(1/k)}
\]
Let us rewrite the term at the right as
\[
\exp(\log(q_{2k})+\log(|\eta|)q_{2k}-\log(a_{s_kq_{2k}}(1/k)))
\]
In order to prove \Cref{eq:zero-limit-along-s-k-q-2-k}, it is sufficient to show that
\[
\lim_{k\to\infty}\log(q_{2k})+\log(|\eta|)q_{2k}-\log(a_{s_kq_{2k}}(1/k))=-\infty
\]
By \Cref{eq:defining-condition-s-k}, we can write 
\[
\log(q_{2k})+\log(|\eta|)q_{2k}-\log(a_{s_kq_{2k}}(1/k))\leq \log(q_{2k})+\log(|\eta|)q_{2k}-kq_{2k}
\]
But it is clear that
\[
\lim_{k\to\infty}\log(q_{2k})+\log(|\eta|)q_{2k}-kq_{2k}=-\infty
\]
Thus we have proved \Cref{eq:zero-slow-entropy-nice-partition}.

Now let $\xi'$ be an arbitrary finite partition of $\T$. Let us focus on the (seemingly uninteresting) partition $\xi'\times \{X\}$ of $\T\times X$. Let $(\theta,x)$ and $(\theta',x')$ be elements in $\T\times X$. It is clear that then for any $F\Subset \Z$ we have 
\[
d^{H,R_{\alpha}\rtimes_\tau T}_{\xi'\times\{X\},F}((\theta,x),(\theta',x'))=
d^{H,R_{\alpha}}_{\xi',F}(\theta,\theta')
\]
It follows that 
\[
S^{H}_{\xi'\times\{X\}}(R_{\alpha}\rtimes_\tau T,\{0,\dots,n-1\},\epsilon,\delta)=
S^{H}_{\xi'}(R_{\alpha},\{0,\dots,n-1\},\epsilon,\delta)
\]
For the irrational rotation $R_{\alpha}$, it is known that for any partition $\xi'$ and $\epsilon,\delta\in(0,1)$, the sequence $
S^{H}_{\xi'}(R_{\alpha},\{0,\dots,n-1\},\epsilon,\delta)$ is bounded in $n$. This is just a rephrasing of the  assertion that $R_{\alpha}$ has zero slow entropy at all scales. See \cite{ferenczi_measuretheoretic_1997} or \cite[\S 4.4]{kanigowski_survey_2024}.

Thus $
S^{H}_{\xi'\times\{X\}}(R_{\alpha}\rtimes_\tau T,\{0,\dots,n-1\},\epsilon,\delta)$ is bounded in $n$. Observe that 
\[
(\xi'\times \{X\})\vee(\xi\times\eta)=(\xi\vee\xi')\times\eta
\]
By \Cref{eq:zero-slow-entropy-nice-partition} and \Cref{prop:partition-vanishes} it follows that 
\[
\ent^{\mathbf a}_{m_{\T}\times\mu}(R_{\alpha}\rtimes_\tau T,(\xi\vee\xi')\times\eta)=0
\]
Since $(\xi\vee\xi')\times\eta$ refines $\xi'\times\eta$, it follows that $
\ent^{\mathbf a}_{m_{\T}\times\mu}(R_{\alpha}\rtimes_\tau T,\xi'\times\eta)=0$. Here $\xi'\times\eta$ is an arbitrary product partition of $\T\times X$. Finally, by the generator-like theorem for slow entropy, this implies $\ent^{\mathbf a}_{m_{\T}\times\mu}(R_{\alpha}\rtimes_\tau T)=0$ (see \cite{katok_slow_1997}, or Theorem 4.1.7 in \cite{kanigowski_survey_2024}). 
\end{proof}
\section{On the variational principle for slow entropy}\label{sec:variational}
The goal of this section is proving \Cref{thm:variational-0} and \Cref{thm:variational} on non-examples of the variational principle for slow entropy.  

An important first step for us will be to replace the irrational rotation $(\T,R_{\alpha})$ by a subshift  $(Y_{\alpha},S_{\alpha})$ which admits a continuous copy $\hat\tau$ of $\tau$. This allows us to define skew products that are topological dynamical systems. After this is done, \Cref{thm:variational-0} and \Cref{thm:variational} will easily follow from \Cref{thm:zero-slow-entropy}, combined with elementary or previously known results. 

Let us also review some definitions from topological dynamics that will be used in this section. By an invertible topological dynamical system we mean a pair $(X,T)$  of a compact metrizable topological space $X$, and a homeomorphism $T\colon X\to X$. The set of $T$-invariant Borel probability measures on $X$ will be denoted $\mathcal M(T)$. 

Among the topological dynamical systems that we consider, an important role will be played by subshifts. By a subshift $Y$ on a finite alphabet $\mathcal A$ we mean a topologically closed subset $Y\subset \mathcal A^\Z$ which is invariant by the shift transformation $S\colon \mathcal A^\Z\to \mathcal A^\Z$, $S((y_n)_{n\in\Z})=(y_{n+1})_{n\in\Z}$. Here $\mathcal A^{\Z}$ is given the prodiscrete topology. In order to regard a subshift as a topological dynamical system, we endow it with the homeomorphism $S|_Y\colon Y\to Y$. The collection of words of length $n$ in a subshift will be denoted by
\[
\mathcal L_n(Y)=\{y|_{\{0,\dots,n-1\}} : y\in Y\}
\]
\subsection{The subshift $(Y_{\alpha},S_{\alpha})$}
It is a standard procedure in dynamics to define a subshift as a set of codings of orbits by a partition. In order to obtain a continuous copy of $\tau$, we consider a partition such that $\tau$ is constant on the atoms. We consider the partition $\{[0,1/2),[1/2,1)\}$. It is notationally convenient to label these atoms according to the values of $\tau$:
\[
A_{1}=[0,1/2),  \ \ \ A_{-1}=[1/2,1). \] 
In this manner the sequence  $(\tau(R_{\alpha}^{n}(\theta)))_{n\in\Z}$ codes the orbit of $\theta$ by $R_{\alpha}$ and the partition $\{A_{-1},A_1\}$, in the sense that the $n$-th symbol in the sequence is the label of the atom containing $R^n_{\alpha}(\theta)$. 

The subshift $Y_{\alpha}$ associated to codings of $R_{\alpha}$ by the partition $\{A_{-1},A_1\}$ was already defined in the introduction, and the shift transformation on $Y_{\alpha}$ is denoted $S_{\alpha}$. We now state some equivalent characterizations of $Y_{\alpha}$. 
\begin{proposition}\label{prop:Y-alpha-characterizations}
The subshift $Y_{\alpha}\subset\{-1,1\}^{\Z}$ admits the following characterizations.  
\begin{enumerate}
\item It is the smallest subshift with alphabet $\{-1,1\}$ containing $(\tau(R_{\alpha}^{n}(0)))_{n\in\Z}$.
\item It is the smallest subshift with alphabet $\{-1,1\}$ containing $\{(\tau(R_{\alpha}^{n}(\theta)))_{n\in\Z} : \theta\in\T\}$.
\item It is the subshift obtained by forbidding all words $w=w_0\dots w_{n-1}\in\{-1,1\}^n$ that do not represent an itinerary by $R_{\alpha}$, meaning 
\[
\bigcap_{i=0}^{n-1}R_{\alpha}^{-i}(A_{w_i})=\emptyset.
\]
\item It is the subshift such that $w=w_0\dots w_{n-1}\in\{-1,1\}^n$ belongs to $\mathcal L_n(Y_{\alpha})$ if and only if 
\[
\bigcap_{i=0}^{n-1}R_{\alpha}^{-i}(A_{w_i})\ne\emptyset.
\]
\end{enumerate}  
\end{proposition}
By the smallest subshift containing an element $y\in\{-1,1\}^\Z$, we mean the intersection of all subshifts containing $y$. This is equivalent to taking the topological closure of the orbit $\{S^n(y) : n\in\Z\}$. 

We will now review some elementary properties of $(Y_{\alpha},S_{\alpha})$ and clarify its relationship with the irrational rotation $(\T,R_{\alpha})$. These results are well-known, so we omit some proofs. 
\begin{proposition}\label{prop:cardinality-language-subshift}
    $|\mathcal L_n(Y_{\alpha})|=2n$.
\end{proposition}
\begin{proof}
	By item 4 in \Cref{prop:Y-alpha-characterizations}, words in $\mathcal L_n(Y_{\alpha})$ are in correspondence with nonempty atoms of $\vee_{i=0}^{n-1}R_{\alpha}^{-i}(\{[0,1/2),[1/2,1)\})$. But these atoms are in visible correspondence with $\{R_{\alpha}^{-i}(0) : i=0,\dots,n-1\}\cup \{R_{\alpha}^{-i}(1/2) : i=0,\dots,n-1\}$. 
\end{proof}
In the next result, $\overline{B}$ denotes the topological closure of $B$, $B\subset\T$. 
\begin{proposition}\label{prop:symbolic-extension-of-rotation}
As topological systems, $(Y_{\alpha},S_{\alpha})$ factors onto  $(\T,R_{\alpha})$. We have the factor map $\pi\colon Y_{\alpha}\to \T$ defined as follows. Given $y\in Y_{\alpha}$, its image $\pi(y)$ is the unique point satisfying 
\begin{equation}\label{eq:def-f}
\bigcap_{i\in\Z} \overline{R_{\alpha}^{-i}(A_{y(i)})}=\{\pi(y)\}, \ \  y\in Y_{\alpha}
\end{equation}
Every $\theta$ outside the orbit by $R_{\alpha}$ of $\{0,1/2\}$ has a unique preimage by $\pi$. In fact, this preimage is given precisely by the coding map
\begin{equation}\label{eq:def-g}
    \varkappa\colon\T\smallsetminus \bigcup_{i\in\Z}R_{\alpha}^{-i}\{0,1/2\}\to Y_{\alpha}, \ \theta\mapsto \varkappa(\theta)=(\tau(R_{\alpha}^i(\theta)))_{i\in\Z}
\end{equation}
\end{proposition}
\begin{proposition}
$(Y_{\alpha},S_{\alpha})$ admits a unique invariant measure, which we denote $\hat m$. 
\end{proposition}
\begin{proposition}
    The measure-preserving systems $(\T,\mathcal B_{\T},m_{\T},R_{\alpha})$ and $(Y_{\alpha},\mathcal B_{Y_{\alpha}},\hat m,S_{\alpha})$ are isomorphic. Indeed, both $\pi$ and $\varkappa$ define measure-theoretic isomorphisms (see the commutative diagram in  \Cref{fig:isomorphism-rotation-and-symbolic-copy}).
\begin{figure}[h]
\begin{center}
\begin{tikzcd}
{Y_{\alpha}\smallsetminus \pi^{-1}(\bigcup_{i\in\mathbb Z}R_{\alpha}^{-i}\{0,1/2\})} \arrow[rr, "S_{\alpha}"] \arrow[dd, "\pi"', bend right] &  & {Y_{\alpha}\smallsetminus \pi^{-1}(\bigcup_{i\in\mathbb Z}R_{\alpha}^{-i}\{0,1/2\})} \arrow[dd, "\pi"', bend right] \\
                                                                                                                                        &  &                                                                                                                \\
{\mathbb T \smallsetminus \bigcup_{i\in\mathbb Z}R_{\alpha}^{-i}\{0,1/2\}} \arrow[rr, "R_{\alpha}"'] \arrow[uu, "\varkappa"', bend right]       &  & {\mathbb T \smallsetminus \bigcup_{i\in\mathbb Z}R_{\alpha}^{-i}\{0,1/2\}} \arrow[uu, "\varkappa"', bend right]       
\end{tikzcd}
\end{center}
\caption{Relations between $S_{\alpha}$, $R_{\alpha}$, $\pi$, and $\varkappa$.}\label{fig:isomorphism-rotation-and-symbolic-copy}
\end{figure}
\end{proposition}
\subsection{The map $\hat\tau$}
Consider the continuous function 
\[
\hat\tau\colon Y_{\alpha}\to\Z, y\mapsto y(0)
\]
This is a copy of $\tau$, via the isomorphisms $\varkappa$ and $\pi$ defined in \Cref{prop:symbolic-extension-of-rotation}, in the sense of the following result. 
\begin{proposition}\label{prop:tau-and-hat-tau}
    We have $\hat\tau=\tau\circ\pi$ and $\tau=\hat\tau\circ\varkappa$ almost everywhere. 
\end{proposition}
\begin{proof}
It is clear that $\hat\tau(y)=\tau(\pi(y))$ for all $y\in Y_{\alpha}$ outside the preimage by $\pi$ of $\bigcup_{i\in\Z}R_{\alpha}^{-i}\{0,1/2\}$, and that $\tau(\theta)=\hat\tau(\varkappa(\theta))$ for all $\theta$ outside $\bigcup_{i\in\Z}R_{\alpha}^{-i}\{0,1/2\}$. 
\end{proof}
If $(X,T)$ is an invertible topological system, then $S_{\alpha}\rtimes_{\hat \tau} T$ is the homeomorphism of the product space $Y_{\alpha}\times X$ defined by
\[
(y,x)\mapsto (S_{\alpha}(y),T^{\hat \tau(y)}(x))
\]
Thus we regard $(Y_{\alpha}\times X,S_{\alpha}\rtimes_{\hat \tau} T)$ as a topological system. When we endow this system with a product measure $\hat m\times\mu$, $\mu\in\mathcal M(T)$, we have the following result. 
\begin{proposition}\label{prop:isomorphism-symbolic-skew-product-and-the-other}
    Let $\mu\in\mathcal M(T)$. Then the measure-preserving systems $(\T\times X,\mathcal B_{\T}\otimes \mathcal B_X, m_{\T}\times\mu,R_{\alpha}\rtimes_{\tau}T)$ and  $(Y_{\alpha}\times X,\mathcal B_{Y_{\alpha}}\otimes\mathcal B_X, \hat m\times\mu,S_{\alpha}\rtimes_{\hat \tau}T)$ are isomorphic.
\end{proposition}
\begin{proof}
By \Cref{prop:tau-and-hat-tau} we have $\hat\tau(y)=\tau(\pi(y))$ for almost every $y\in Y_{\alpha}$. The map $Y_{\alpha}\times X\to\T\times X$ given by 
\[
(y,x)\mapsto (\pi(y),x)
\]
defines a measure-theoretic isomorphism. 
\end{proof} 
    \subsection{The slow entropy for $S_{\alpha}\rtimes_{\hat \tau} T$}
Recall that we defined measure-theoretic slow entropy in \Cref{sec:slow-entropy}, as well as the notion of scale. Let us now recall the definition of topological slow entropy. Let $(X,T)$ be an invertible topological dynamical system and  let $d$ be a compatible metric for the topology on $X$. Given $F\Subset\Z$ consider the dynamically generated Bowen metric
\[
d^{T}_{F}(x,y)=\max\{d(T^{i}(x),T^{i}(y)) : i\in F\}
\]
The corresponding open ball is denoted $B^T_F(x,\epsilon)=\{y\in X : d^T_F(x,y)<\epsilon\}$. This allows us to define covering numbers as 
\[
S(T,F,\epsilon)=\min\{k\in\N : \ \text{there are} \ x_1,\dots,x_k\in X \text{ such that } X\subset \cup_{i=1}^k B^{T}_F(x_i,\epsilon)\} 
\]
Given a scale $\mathbf a=\{a_n(t)\}_{n\in\N,t>0}$ we define 
\[
\ent^{\mathbf a}_{top}(T,
\epsilon)=\Sup\{t>0 : \liminf_{n\to\infty}\frac{S(T,\{0,\dots,n-1\},\epsilon)}{a_n(t)}>0\}
\]
Recall that for a possibly empty set $A\subset [0,\infty)$, we defined $\Sup(A)$ as $\sup(A\cup\{0\})$. Then the lower slow entropy of $(X,T)$ with respect to $\mathbf{a}$ is defined as
\[
\ent^{\mathbf a}_{top}(T)=\lim_{\epsilon\to 0}
\ent^{\mathbf a}_{top}(T,
\epsilon) = \sup_{\epsilon> 0}
\ent^{\mathbf a}_{top}(T,
\epsilon)
\]
We refer the reader to the survey  \cite{kanigowski_survey_2024} for further details. 

In the case of a subshift, we can ignore the $\epsilon$ in the definition of slow entropy, and the definition simplifies as follows (a similar statement holds for upper slow entropy).
\begin{proposition}\label{prop:topological-slow-entropy-subshifts}
    Let $Y\subset \mathcal A^\Z$ be a subshift endowed with the shift action $S_Y$. For an arbitrary scale $\mathbf a=\{a_n(t)\}_{n\in\N,t>0}$, we can compute the topological slow entropy of $(Y,S_Y)$ as 
    \[
\ent^{\mathbf a}_{top}(S_Y)=\Sup\{t>0 : \liminf_{n\to\infty}\frac{|\mathcal L_n(Y)|}{a_n(t)}>0\}
\]
\end{proposition}
\begin{proof}
    Endow $Y$ with the distance \[d(x,y)=\inf(\{2^{-n} : y(i)=x(i) \text{ for }i\in\{-n+1,\dots,n-1\}, n\geq 1\}\cup\{1\})\]
    Thus every pair of elements $x,y\in Y$ are at distance at most $1$, and they are at distance smaller than $1$ if they agree on the $0$-th coordinate. 

    We claim that for all $\epsilon,\epsilon'\in(0,1]$ we have $\ent^{\mathbf a}_{top}(S_Y,\epsilon)= \ent^{\mathbf a}_{top}(S_Y,\epsilon')$. Indeed, suppose $\epsilon\leq\epsilon'$. The  inequality $\ent^{\mathbf a}_{top}(S_Y,\epsilon)\geq \ent^{\mathbf a}_{top}(S_Y,\epsilon')$ is direct (if we take smaller balls, we need more of them to cover the space). For the other inequality, observe that there is a constant $C$ depending on $\epsilon$ and $\epsilon'$ but not on  $n$, such that $S(S_Y,\{0,\dots,n-1\},\epsilon)\leq C\cdot S(S_Y,\{0,\dots,n-1\},\epsilon')$ for all $n$. From this and the definitions we obtain $\ent^{\mathbf a}_{top}(S_Y,\epsilon)\leq \ent^{\mathbf a}_{top}(S_Y,\epsilon')$. 

    It follows that $\ent^{\mathbf a}_{top}(S_{Y})=\lim_{\epsilon\to 0}\ent^{\mathbf a}_{top}(S_{Y},\epsilon)=\ent^{\mathbf a}_{top}(S_{Y},1)$. Next, observe that with the Bowen distance associated to $d$, we have $d^{S_Y}_{\{0,\dots,n-1\}}(x,y)< 1$ if and only if $x,y$ agree when restricted to $\{0,\dots,n-1\}$. Thus  $S(S_Y,\{0,\dots,n-1\},1)$ equals the cardinality of $\mathcal L_n(Y)$, and  
    \[
\ent^{\mathbf a}_{top}(S_Y)=\ent^{\mathbf a}_{top}(S_Y,1)=\Sup\{t>0 : \liminf_{n\to\infty}\frac{|\mathcal L_n(Y)|}{a_n(t)}>0\}
    \]
    as claimed.
\end{proof}
Let us observe that the topological slow entropy of an infinite subshift admits a universal lower bound with the polynomial scale $\{n^t\}_{n\in\N,t>0}$.
\begin{proposition}\label{prop:polynomial-entropy-of-sturmian-shift}
Let $(Y,S)$ be an infinite subshift. The topological slow entropy of $(Y,S)$  with polynomial scale $\mathbf{a}=\{n^t\}_{n\in\N,t>0}$ satisfies
	\[
	\ent_{top}^{\mathbf a}(S)\geq 1
	\]
\end{proposition}
\begin{proof}
By the Morse-Hedlund Theorem \cite{MorseHedlund1938} we have $|\mathcal L_n(Y)|\geq n+1$ for all $n$, and hence 
\[\lim_{n\to\infty}\frac{|\mathcal L_n(Y)|}{n^t}=\infty, \ \ \ t<1.\]
By \Cref{prop:topological-slow-entropy-subshifts} it follows that $\ent^{\mathbf a}_{top}(S)\geq 1$.
\end{proof}
In fact, the topological slow entropy of the subshift $(Y_\alpha,S_{\alpha})$ with the polynomial scale $\{n^t\}_{n\in\N,t>0}$ is exactly 1, $\alpha\in\R\smallsetminus\Q$. This follows from the  argument used in \Cref{prop:polynomial-entropy-of-sturmian-shift}, and \Cref{prop:cardinality-language-subshift}. But for the next result we will only need the lower bound. 
\begin{theorem}[\Cref{thm:variational-0}]\label{thm:variational-0-section}
    Consider the polynomial slow entropy scale $\mathbf a = \{n^t\}_{n\in\N,t>0}$. There exists  $\alpha\in\R\smallsetminus\Q$ such that for every invertible topological system $(X,T)$ we have 
    \[\ent^{\mathbf a}_{top}(S_{\alpha}\rtimes_{\hat \tau} T)\geq 1 \text{ and }  \ \ 
\sup_{\nu\in\mathcal  M(S_{\alpha}\rtimes_{\hat \tau} T)} 
\ent_{\nu}^{\mathbf a}(S_{\alpha}\rtimes_{\hat \tau} T)=0\]
\end{theorem}
\begin{proof}
We apply  \Cref{thm:zero-slow-entropy} to the scale  $\mathbf a=\{n^t\}_{n\in\N,t>0}$. We obtain an irrational $\alpha$ such that for every invertible measure-preserving system $(X,\mathcal B_X,\mu,T)$ we have
    \[
    \ent^{\mathbf a}_{m_{\T}\times\mu}(R_{\alpha}\rtimes_{\tau} T)=0
    \]
Now fix an invertible topological system $(X,T)$. Consider the topological skew product $S_{\alpha}\rtimes_{\hat \tau} T$, and pick an arbitrary invariant measure $\nu\in\mathcal M(S_{\alpha}\rtimes_{\hat \tau} T)$. By \Cref{thm:product-measures}, $\nu$ is a product measure $\hat m\times \mu$, where $\mu\in\mathcal M(T)$. The measure-preserving systems $(\T\times X,\mathcal B_{\T}\otimes\mathcal B_X,m_{\T}\times\mu,R_{\alpha}\rtimes_\tau T)$ and $(Y_{\alpha}\times X,\mathcal B_{Y_{\alpha}}\otimes \mathcal B_X,\hat m\times\mu,S_{\alpha}\rtimes_{\hat \tau} T)$ are isomorphic by \Cref{prop:isomorphism-symbolic-skew-product-and-the-other}. Since measure-theoretic slow entropy is an invariant for isomorphism, we have 
    \[
    \ent^{\mathbf a}_{\hat m\times\mu}(S_{\alpha}\rtimes_{\hat \tau} T)=0
    \]
    Since $\nu=\hat m\times \mu$ was an arbitrary invariant measure for $S_{\alpha}\rtimes_{\hat \tau} T$, it follows that 
    \[
    \sup_{\nu\in\mathcal M(S_{\alpha}\rtimes_{\hat \tau} T)} 
    \ent^{\mathbf a}_{\nu}(S_{\alpha}\rtimes_{\hat \tau} T)=0
    \]
    To finish the argument, observe that topological slow entropy is monotone with respect to factor maps. Since $S_{\alpha}\rtimes_{\hat \tau} T$ factors onto $S_{\alpha}$ via the projection map to the first coordinate, we have $\ent^{\mathbf a}_{top}(S_{\alpha}\rtimes_{\hat \tau} T)\geq \ent^{\mathbf a}_{top}(S_{\alpha})$. But  $\ent^{\mathbf a}_{top}(S_{\alpha})\geq 1$ by \Cref{prop:polynomial-entropy-of-sturmian-shift}, so $\ent^{\mathbf a}_{top}(S_{\alpha}\rtimes_{\hat \tau} T)\geq 1$ as claimed. 
\end{proof}
Let us now turn to the scale $\mathbf b=\{b_n(t)\}_{n\in\N,t>0}$ mentioned in the introduction. 
Recall that given $\theta\in\T$ we defined the range $r_n(\theta)$ as the cardinality of $\{\tau(i,\theta) : i=0,\dots,n-1\}$ (\Cref{eq:r_n-theta}). For a word $w=w_0\dots w_{n-1}\in\mathcal L_n(Y_{\alpha})$ we define its range $r_n(w)$ by 
\begin{equation}\label{eq:def-r-n-symbolic}
    r_n(w)=r_n(\theta)
\end{equation}
where $\theta\in\T$ is any element satisfying $\tau(R_{\alpha}^i(\theta))=w_i$, $i=0,\dots,n-1$. This is well-defined thanks to \Cref{prop:tau-constant}. We now define $b_n(t)$ by 
\begin{equation}\label{eq:def-b-scale}
b_n(t)=\sum_{w\in \mathcal L_n(Y_{\alpha})}e^{r_n(w)t}
\end{equation}
We emphasize that this scale depends on the choice of $\alpha$. This scale was introduced in \cite{carrascovargas_topological_2025} to capture the topological entropy $h_{top}(T)$ of a topological system $(X,T)$ in a topological skew product $S_{\alpha}\rtimes_{\hat \tau} T$ in the sense of the following result.
\begin{theorem}[\cite{carrascovargas_topological_2025}]\label{thm:nice-scale-topological-slow-entropy}
Let $\alpha\in\R\smallsetminus\Q$ and let $\mathbf b$ be the scale defined by \Cref{eq:def-b-scale}. Then for every invertible topological system $(X,T)$ we have 
\[
\ent^{\mathbf b}_{top}(S_{\alpha}\rtimes_{\hat \tau} T)=h_{top}(T)
\]
\end{theorem}
We will now prove that for some irrationals, the same scale is dramatically far from capturing the measure-theoretic complexity of the same systems, and in fact we obtain non-examples of the variational principle by choosing $T$ with $h_{top}(T)>0$. 
\begin{theorem}[\Cref{thm:variational}]\label{thm:variational-section}
    For some $\alpha\in\R\smallsetminus\Q$ the following holds. For every invertible topological system $(X,T)$ we have 
    \[
\sup_{\nu\in\mathcal M(S_{\alpha}\rtimes_{\hat \tau} T)}\ent^{\mathbf b}_{\nu}(S_{\alpha}\rtimes_{\hat \tau} T)=0
    \]
\end{theorem}
\begin{proof}
Since the scale $\mathbf b$ depends on $\alpha$, we cannot directly apply \Cref{thm:zero-slow-entropy}. However, $\mathbf b$ admits a natural lower bound independent of $\alpha$. That is, since $r_n(w)\geq 1$ and $|\mathcal L_n(Y_{\alpha})|=2n$ (\Cref{prop:cardinality-language-subshift}), we have 
\[
b_n(t)\geq \sum_{w\in\mathcal L_n(Y_{\alpha})}e^t\geq |\mathcal L_n(Y_{\alpha})|e^t=2ne^t
\]
Consider the scale $\mathbf a=\{a_n(t)\}_{n\in\N,t>0}$ defined by
\[
a_n(t)=2ne^{t}
\]
We invoke \Cref{thm:zero-slow-entropy} and apply it to the scale $\mathbf a$. We obtain an irrational $\alpha$ such that for every invertible measure-preserving system $(X,\mathcal B_X,\mu,T)$ we have
\[
\ent_{m_{\T}\times\mu}^{\mathbf a}(R_{\alpha}\rtimes_\tau T)=0
\]
The irrational $\alpha$ is fixed from now on, together with the associated scale $\mathbf b$.

It follows from the definitions of slow entropy that since $b_n(t)\geq a_n(t)$ for all $n\in\N$ and $t>0$, the slow entropy with respect to $\mathbf b$ of any system is at most its slow entropy with respect to $\mathbf a$. Thus for every invertible measure-preserving system $(X,\mathcal B_X,\mu,T)$ we have 
\begin{equation}\label{eq:skew-products-zero-entropy}
\ent_{m_{\T}\times\mu}^{\mathbf b}(R_{\alpha}\rtimes_\tau T)=0
\end{equation}
The rest of the argument is the same as in \Cref{thm:variational-0} (applying \Cref{thm:product-measures} and \Cref{prop:isomorphism-symbolic-skew-product-and-the-other}). 
\end{proof}
We finish this section with the observation that, provided $T$ has no fixed point and no periodic orbit of length two, the topological system $S_{\alpha}\rtimes_{\hat\tau} T$ has no ergodic invariant measure that makes it a Kronecker system. As mentioned in the introduction, this shows that the non-examples of the variational principle for slow entropy that we construct are qualitatively different from previously known non-examples (see \cite{kanigowski_survey_2024,cheng_slow_2025}). 
\begin{proposition}\label{prop:discrete-spectrum-measures}
    Let $\alpha\in\R\smallsetminus\Q$ and let $(X,T)$ be an invertible topological system. Let $\nu$ be an ergodic $S_{\alpha}\rtimes_{\hat \tau} T$-invariant measure. Suppose that, endowed with $\nu$, $S_{\alpha}\rtimes_{\hat \tau} T$ is a Kronecker system. Then we can write $\nu=\hat m\times \mu$, where $\mu$ is $T$-invariant and supported either on a fixed point for $T$, or a periodic orbit of length two. 
\end{proposition}
\begin{proof}
    Suppose that $(Y_{\alpha}\times X,\mathcal B_{Y_{\alpha}}\otimes \mathcal B_X,\nu,S_{\alpha}\rtimes_{\hat\tau} T)$ is a Kronecker system, for $\nu\in\mathcal M(S_{\alpha}\rtimes_{\hat\tau} T)$. By \Cref{thm:product-measures} we can write $\nu=\hat m\times\mu$ for some $\mu\in \mathcal M(T)$. By \Cref{prop:isomorphism-symbolic-skew-product-and-the-other}, $(Y_{\alpha}\times X,\mathcal B_{Y_{\alpha}}\otimes \mathcal B_X,\hat m\times\mu,S_{\alpha}\rtimes_{\hat\tau} T)$ is isomorphic to $(\T\times X,\mathcal B_{\T}\otimes \mathcal B_X,m_{\T}\times\mu,R_{\alpha}\rtimes_{\tau} T)$. Thus the latter is a Kronecker system. By \Cref{prop:kronecker-characterization}, it follows that $(X,\mathcal B_X,\mu,T)$ is isomorphic to the one point system, or the system which exchanges two points with equal measures. It follows that $\mu$ is supported on a fixed point for $T$, or a periodic orbit of length two. 
\end{proof}
\bibliography{references}
\end{document}